\documentclass[11pt, reqno]{amsart}
\usepackage{geometry}
\usepackage[dvipsnames]{xcolor}
\usepackage{mathrsfs}
\usepackage[shortlabels]{enumitem}
\usepackage{graphicx}
\usepackage{amssymb, mathtools}
\usepackage{hyperref}

\title[Growth of Couette flow]{Long time instability of the Couette flow in low Gevrey spaces}
\author{Yu Deng, Nader Masmoudi}
\date{}

\newtheorem{theorem}{Theorem}[section]
\newtheorem{proposition}[theorem]{Proposition}
\newtheorem{remark}[theorem]{Remark}

\numberwithin{equation}{section}
\begin{document}
\maketitle

\begin{abstract}

We prove the instability of the Couette flow if the disturbances   is less smooth than 
the Gevrey space of class  2.  This shows that this is the critical regularity  for this problem
 since  it was proved in \cite{BM15}  that stability and inviscid damping hold for 
 disturbances   which are  smoother  than 
the Gevrey space of class  2.  A big novelty is that this critical space is due to 
an instability mechanism which is completely nonlinear and is due to some 
energy cascade. 

\end{abstract}

\section{Introduction} 

In recent years  a lot of effort was devoted to proving several 
stability results in fluid mechanics  \cite{BM15,BGM17,Wu11,GMS12,DIPP16}, plasma physics  \cite{MV11,Guo03,GM14,GIP16}  and general relativity \cite{LR10}. 
These results usually assume  some smallness in  some space with  high regularity  and/or 
spacial localization.  The main idea  is to prove that the linear problem has some good properties 
(dispersion, mixing, decay)  that allow to control the nonlinear term  and hence yield a global existence 
result.  Of course the particular properties of the nonlinearity are very important 
(energy conservation, null structure,...).    A natural question to be asked is  whether these  regularity,  decay and smallness   assumptions on the initial  data   
are  really needed: Is it  possible to  prove some 
instability if one of  these  assumptions is removed?  
  A second natural question is to describe the formation of the instability if stability does not hold.  
     This kind of questions turns out to be much more difficult (see of instance \cite{GHS07,GN12,BGM15II}).  
 One has to find a growth mechanism, construct an approximate solution that 
 mimics that growth and then prove that there is an exact solution that stays 
 close to this approximate solution.   In some sense one has to prove the "stability" of this 
 unstable profile.   In some cases the instability is so strong that it yields an ill-posedness result 
 \cite{GN12}. 
 
 In this paper,  we prove such  an instability result for the 2D Euler system close to the  Couette 
 flow.  Unlike previous results where the instability comes from the linear part, namely the presence of 
a growing mode  \cite{GHS07,GN12}  or  the presence of a  Jordan block structure   \cite{BGM15II}, the growth mechanism   here 
 is completely nonlinear.  Indeed, the linear problem is  in some sense stable. 
   The nonlinear mechanism   is a sort of  inverse energy cascade where energy 
  propagates 
 from higher modes in $x$  to lower modes. This movement  happens with some amplification 
 that causes the growth and yields the instability.

 Our result should be compared to the result of \cite{BM15}, where stability of the Couette flow was proved in Gevrey class $\mathcal{G}^{\lambda,1/2+}$; we are showing instability in the regularity class that is just below. The main goal here is to 
 prove   that weakly nonlinear effects  can create a self-sustaining process and  push 
the solution out of  the linear regime. 
The idea that the interaction between nonlinear effects and non-normal transient growth can lead to instabilities is classical in fluid mechanics (see e.g. \cite{TTRD93}).
The basic mechanism suggested in \cite{TTRD93} is that nonlinear effects can repeatedly excite growing modes 
and precipitate a sustained cascade or so-called `nonlinear bootstrap', studied further in the fluid mechanics context in, for example, \cite{BaggettEtAl,VMW98,Vanneste02}. 
Actually, this effect is very similar to what is at work behind \emph{plasma echos} in the Vlasov equations, first captured experimentally in \cite{MalmbergWharton68}. 
This phenomenon is referred to as an `echo' because the measurable result of nonlinear effects can occur long after the event.   
Very similar echos have been studied and observed in 2D Euler, both numerically \cite{VMW98,Vanneste02} and experimentally \cite{YuDriscoll02,YuDriscollONeil} (interestingly, non-neutral plasmas in certain settings make excellent realizations of 2D Euler). 

Consider the $2D$ Euler equation
\begin{equation}\label{euler0}
\left\{\begin{split}\partial_tu+(u\cdot\nabla)u&=-\nabla p,\\
\nabla\cdot u&=0,
\end{split}\right.
\end{equation} on $\mathbb{T}_x\times\mathbb{R}_y$. A \emph{shear flow} is a specific solution $u=(V(y),0)$ to (\ref{euler0}), and the specific case $V(y)=y$ is called the \emph{Couette flow}.

We are interested in long-time instability of small perturbations to the Couette flow. Writing the perturbation of (\ref{euler0}) in vorticity form, i.e. define $\omega=\nabla\times u+1$, then we can reduce (\ref{euler0}) to
\begin{equation}\label{euler}
\left\{
\begin{aligned}
\partial_t\omega&+y\partial_x\omega+U\cdot\nabla\omega=0,\\
U&=\nabla^{\perp}\Delta^{-1}\omega,
\end{aligned}
\right.
\end{equation} where $\nabla^{\perp}=(-\partial_y,\partial_x)$. The equation (\ref{euler}) will be the main object of study in this paper.
\subsection{Backgrounds, and the main theorem}
\subsubsection{Inviscid damping and stability} The stability problem for the Couette flow, as well as the more general shear flows, has been a topic of interest since the end of the nineteenth century  \cite{Case59,Case60,Kelvin87,MP77,Orr07,Rayleigh80,Rayleigh87,Rayleigh95}.

The linearization of (\ref{euler}) is the transport equation $\partial_t\omega+y\partial_x\omega=0$, which is stable, i.e., has no growing mode.
In fact the same is true for a class of shear flows \cite{Rayleigh80}, and for the corresponding Navier-Stokes equations. Nevertheless, in the case of sufficiently small viscosity, experiments realizing the Navier-Stokes system have exhibited evidence of long-time instability for the Couette and other linearly stable shear flows. This phenomenon, usually referred to as the \emph{Sommerfeld paradox}, suggests the presence of some subtle nonlinear effects in the small data problem for (\ref{euler}).

Two important observations, regarding the interplay between linear and nonlinear effects, are made by Orr \cite{Orr07}. Expressed in modern terminology, the idea is to invert the linearized flow of (\ref{euler}) by the change of coordinates
\begin{equation}\label{naive}f(t,z,y)=\omega(t,z+ty,y),
\end{equation} then $f$ is transported by the vector field $\nabla^{\perp}\phi$, where $\phi$ is the stream function $\Delta^{-1}\omega$ written in new coordinates, such that
\begin{equation}\label{naive2}\widehat{\phi}(t,k,\xi)=\frac{\widehat{f}(t,k,\xi)}{(\xi-tk)^2+k^2}.
\end{equation} The observations of Orr are then the followings:

(1) the equation (\ref{naive2}) incorporates $t^{-2}$ decay for $\phi$ at the cost of two derivatives. This is due to the effect of the linear transport equation moving $\widehat{\omega}$ to high frequencies, which is known as \emph{inviscid damping};

(2) the denominator of (\ref{naive2}) is minimized at $\xi=kt$. Thus the $(k,\xi)$-Fourier mode of the stream function $\phi$ is peaked at time $t=\xi/k$, known as the \emph{Orr critical times}. These peaks  cause the vorticity $f$ to exhibit transient growth near  the critical times, which Orr used as a possible explanation for the Sommerfeld paradox.

The evolution of small perturbations around the  Couette flow is then basically governed by these two effects. This is similar to the case of Vlasov equations in plasma physics, where  \emph{Landau damping} replaces inviscid damping  and  \emph{plasma echoes} replace the Orr mechanism.  
Despite the good understanding of the linear picture, the asymptotic behavior of small perturbations 
 for  both equations  remained open for a long time.

In \cite{MV11}, Mouhot and Villani were able to justify Landau damping for the full nonlinear system in analytic regularity, thus establishing stability of the equilibrium state (see also \cite{BMM16} for an improvement to Gevrey spaces). The corresponding result (i.e. nonlinear inviscid damping and stability) for the $2D$ Couette flow was then obtained by Bedrossian-Masmoudi \cite{BM15} (see also \cite{BMV16} for a corresponding result for Navier-Stokes). Precisely, they proved in \cite{BM15} the following\footnote{The precise statement proved in \cite{BM15} only covers the case $1/2<s<1$; however, the proof can be adapted, with small modifications, to the analytic case.}
\begin{theorem}[Bedrossian-Masmoudi \cite{BM15}]\label{backthm} Let $\lambda>\lambda'>0$ and $1/2<s\leq 1$, and let $\varepsilon>0$ be sufficiently small. Suppose the initial data $\omega(0)$ satisfies $\|\omega(0)\|_{\mathcal{G}^{\lambda,s}}\leq\varepsilon$, and the assumptions \begin{equation}\label{auxassum}\int_{\mathbb{T}\times\mathbb{R}}(1+|y|)\cdot|\omega(0,x,y)|\,\mathrm{d}x\mathrm{d}y\leq\varepsilon,\quad\int_{\mathbb{T}\times\mathbb{R}}\omega(0,x,y)\,\mathrm{d}x\mathrm{d}y=0,\end{equation} where the Gevrey norm is defined by 
\[\|f\|_{\mathcal{G}^{\lambda,s}}^2=\sum_k\int_{\mathbb{R}}e^{2\lambda(|k|+|\xi|)^s}|\widehat{f}(k,\xi)|^2\,\mathrm{d}\xi,
\] then there exists a unique global solution $\omega$ to (\ref{euler}). Moreover, the solution satisfies that
\begin{equation}\label{backthm1}\|\omega(t,x+t\Psi(y),y)\|_{\mathcal{G}^{\lambda',s}}\lesssim\varepsilon,
\end{equation} where $\Psi(y)=y+O(\varepsilon)$ is another shear flow near the Couette flow, and there exists 
  $  \omega_\infty \in {\mathcal{G}^{\lambda',s}}  $  such  that
\begin{equation}\label{backthm2}\|\omega(t,x+\Phi(t,y),y) -\omega_\infty(x+\Phi(t,y),y) \|_{\mathcal{G}^{\lambda',s}}\lesssim\varepsilon^2(1+|t|)^{-1},
\end{equation} where $\Phi(t,y)=t\Psi(y)+O(\varepsilon^2)$ is a correction term depending on the solution.
\end{theorem}
\subsubsection{Orr growth mechanism and instability}\label{instab00} Notice that the results in \cite{BM15,BMV16} and \cite{MV11} are all proved in Gevrey spaces, which are exponentially regular. As discussed before, it is natural to ask whether this regularity requirement is really necessary.

One reason to believe that exponential regularity is required is the presence of the Orr growth mechanism (or plasma echo for Vlasov equations). In \cite{BM15}, as a crucial step towards Theorem \ref{backthm}, it is proved that the total amplification factor caused by the growth at critical times has an upper bound of $e^{O(1)\sqrt{|\xi|}}$ at frequency $\xi$ (hence the assumption $s>1/2$ in Theorem \ref{backthm}). It is expected that this upper bound can in fact be saturated, which would then imply the optimality of the exponent in Theorem \ref{backthm}. The same thing happens for the Vlasov equations.

However, this is not easily justified, due to the complexity of the system and the possible cancellations in the growth mechanism that are hard to exclude. For the $2D$ Couette flow, the only known situation where the dynamics of Theorem \ref{backthm} does not happen is due to Lin and Zeng \cite{LinZeng11}, who constructed time periodic solutions to (\ref{euler}) in the Sobolev space $H^s$ where $s<3/2$ (thus proving asymptotic instability). For Vlasov equations,  Lin and Zeng \cite{LZ11b}  also proved a similar  result, and recently Bedrossian \cite{Bedrossian16} has established instability in the Sobolev space $H^s$ for  all $s$.

In the current paper,  we prove the criticality of the Gevrey  exponent $1/2$  
 by proving instability of the $2D$ Couette flow in any Gevrey space $\mathcal{G}^{\lambda,s}$ where $0<s<1/2$. In fact we will prove something slightly stronger, where $\mathcal{G}^{\lambda,s}$ is replaced by a log-corrected version of $\mathcal{G}^{\lambda,1/2}$.
\subsubsection{The main result} We now state the main result of this paper.
\begin{theorem}\label{main}
Let $N_0=9000$, $N_1=60000$, and denote $\log^+(x)=\log(2+|x|)$. For a function $f:\mathbb{T}\times\mathbb{R}\to\mathbb{R}$, define the Gevrey-type norm $\mathcal{G}^{*}$ by
\begin{equation}\label{gevrey}
\|F\|_{\mathcal{G}^{*}}^{2}=\sum_{k\in\mathbb{Z}}\int_{\mathbb{R}}e^{2\kappa(k,\xi)}|\widehat{F}(k,\xi)|^2\,\mathrm{d}\xi,\quad \kappa(k,\xi)=\frac{(|k|+|\xi|)^{1/2}}{(\log^+(|k|+|\xi|))^{N_1}}.
\end{equation}Then, for any sufficiently small $\varepsilon>0$, there exists a solution $\omega=\omega(t,x,y)$ to \eqref{euler}, such that:
\begin{enumerate}
\item The initial data $\omega(0)$ satisfies the assumptions (\ref{auxassum}), and that
\begin{equation}\label{init}\|\omega(0)\|_{\mathcal{G}^*}\leq\varepsilon;
\end{equation}
\item At some later time $T$, the solution $\omega$ satisfies that
\begin{equation}\label{final}\|\langle \partial_x\rangle^{N_0}\omega(T,x,y)\|_{L^2(\mathbb{T}\times\mathbb{R})}\geq\frac{1}{\varepsilon}.
\end{equation}
\end{enumerate}
\end{theorem}
\begin{remark} Note that \eqref{final} implies
\[\|\omega(T,x+\Phi(T,y),y)\|_{H^{N_0}(\mathbb{T}\times\mathbb{R})}\geq\frac{1}{\varepsilon},
\] for any function $\Phi=\Phi(t,y)$. Comparing this with (\ref{backthm1}) or (\ref{backthm2}), we see that no analog of Theorem \ref{backthm} can hold with initial data in the space $\mathcal{G}^*$.
\end{remark}
\subsection{Ideas of the proof} The proof starts by preforming a change of coordinates $(x,y)\mapsto (z,v)$. Following \cite{BM15}, we will use a slightly modified version of (\ref{naive}), see (\ref{changeva}) below, in order to handle the difficulties caused by the zeroth mode. This then reduces (\ref{euler}) to a system of equations satisfied by $g=(f,h,\theta)$, see (\ref{eulernew})$\sim$(\ref{eulersup}) below for details.
\subsubsection{The choice of data, and setup} The goal is to find a solution $(f,h,\theta)$ to (\ref{eulernew})$\sim$(\ref{eulersup}) that satisfies the required instability assumptions, see Proposition \ref{main1.5}. This solution will be constructed as the superposition of a background solution $(\underline{f},\underline{h},\underline{\theta})$, and a perturbation $(f^*,h^*,\theta^*)$ (which is a ``second order'' perturbation to the Couette flow). It turns out that $(h,\theta)$ plays a relatively less important role in the proof, so for simplicity, here we will consider $f$ only.

The background solution $\underline{f}$ is guaranteed to exist by Theorem \ref{backthm}; we will assume it has analytic regularity (i.e. $s=1$ in Theorem \ref{backthm}), and has size $\varepsilon_0\ll\varepsilon$, see Section \ref{background}. In practice we will think $\underline{f}$ as only having low frequency components, as it is much more regular than the perturbation $f^*$ we will construct.

The perturbation $f^*$ will be fixed by assigning the data at some time $t=T_0$, see Section \ref{perturbation}:
\[f^*(T_0)=\varepsilon_1 \cos(k_0z+\eta_0v)\varphi_p(k_0\sqrt{\sigma}v),
\] where $\varphi_p$ is a suitable Schwartz function. The relationship between the parameters $(\varepsilon_0,k_0,T_0,\eta_0,\sigma)$ are listed in (\ref{defpara}); for now it suffices to note that $\varepsilon_1\ll\varepsilon_0$, and $\widehat{f^*(T_0)}$ is concentrated near only two frequencies, $(k_0,\eta_0)$ and $(-k_0,-\eta_0)$, where $(k_0,\eta_0)$ is  the high frequency mode compared to $\underline{g}$. We also fix  two times $T_1\in[T_0,2T_0]$ and $T_2\leq T_0$, see (\ref{defineint}).
\subsubsection{The linearized system around $\underline{f}$}. Since $\varepsilon_1\ll\varepsilon_0$, it is natural to first study the linearization of (\ref{eulernew})$\sim$(\ref{eulersup}) at the background solution $\underline{f}$. This linear system has the form $\partial_tf'=\mathcal{L}f'$, where $\mathcal{L}$ is a linear operator defined in (\ref{eulernewl}), and $f'(T_0)=f^*(T_0)$. Following the observation made in \cite{BM15}, we know $\mathcal{L}$ consists of two parts: the first one is a ``transport'' term,
\begin{equation}\label{trans0}\mathcal{L}^Tf'=\underline{\Phi}\cdot\nabla f',
\end{equation} where $\underline{\Phi}$ is a combination of the background solution, which has much higher regularity than $f'$ (and thus contributes low frequencies only), and moreover decays like $t^{-2}$ (roughly $\underline{\Phi}\sim t^{-2}\underline{f}$).

The second one is a ``reaction'' term, which is responsible for the nonlinear  growth mechanism,
\begin{equation}\label{react0}\mathcal{L}^Rf'=\underline{F}\cdot\nabla\underline{\Delta_t}^{-1}f',
\end{equation} where $\underline{F}$ again comes from the background solution, but has no decay in time; the operator $\underline{\Delta_t}^{-1}$ is defined, up to some error terms, by
\[\widehat{\underline{\Delta_t}^{-1}f}(t,k,\xi) \sim \frac{-1}{(\xi-kt)^2+k^2}\widehat{f}(t,k,\xi);
\] see (\ref{kernel1}) for the precise expression. Notice that, if one compares (\ref{trans0}) and (\ref{react0}), say at a critical time $t=\xi/k$, and assume that $\underline{\Phi}\sim t^{-2}\underline{f}$, 
 then $\mathcal{L}^T$ dominates $\mathcal{L}^R$ if $t\lesssim k$ or equivalently $t\lesssim\sqrt{|\xi|}$, and $\mathcal{L}^R$ dominates $\mathcal{L}^T$ if $t\gtrsim\sqrt{|\xi|}$.

Our strategy is to show that the size of $\widehat{f'}$, say near $(\pm k_0,\pm\eta_0)$, exhibits growth at critical times \emph{between $T_0$ and $T_1$} by the Orr mechanism, and in fact saturates the upper bound proved in \cite{BM15}. Note that this also explains the seemingly strange choice of assigning data at $t=T_0$ instead of $t=1$, since we only know how to saturate the optimal growth on $[T_0,T_1]$.

Moreover, we need to go \emph{backwards} from $T_0$ and recover the control for $f'$ at time $t=1$. There are two regimes here: when $t$ is small (namely $t\leq T_2$; note that $T_2$ is almost $  \sqrt{\eta_0}$, see (\ref{defineint})), the transport term dominates, and the growth of $f'$ can be easily controlled by an energy-type inequality for transport equations, see Proposition \ref{gevrey2}; when $t$ is large, namely $t\in[T_2,T_0]$, the reaction term dominates and the situation will be much similar to what happens on $[T_0,T_1]$, except that only an upper bound is needed.

Summing up, we need to obtain a lower bound for $f'$ on $[T_0,T_1]$, and an upper bound for $f'$ on $[T_2,T_0]$, of form
\begin{equation}\label{compare}|\widehat{f'}(T_1,\pm k_1,\pm\eta_0)|\gtrsim e^{c\sqrt{\eta_0}}\varepsilon_1;\quad |\widehat{f'}(T_2, . ,\pm\eta_0)|\lesssim e^{c'\sqrt{\eta_0}}\varepsilon_1,
\end{equation} for some suitable $c>c'>0$. This would then imply that $f'(T_1)$ is large in $H^{N}$, and that $f'(T_2)$ (and hence $f'(1)$) is small in $\mathcal{G}^*$, upon choosing $\varepsilon_1$ appropriately. In both cases it is crucial to obtain precise bounds on the size of $\widehat{f'}$ near the  frequency $(\pm k_0,\pm\eta_0)$, which is the next step of the proof.
\subsubsection{Linear analysis, and a more precise toy model}\label{precisetoy} We may now restrict the linearized system to time $t\in[T_2,T_1]$, where the transport term plays essentially no role, so we will focus on the reaction term only. Recall the expression in (\ref{react0}); for simplicity we assume that $\underline{F}$ is independent of time and has only $k=\pm 1$ modes, say $\widehat{\underline{F}}(t,k,\xi)=\varepsilon_0\mathbf{1}_{k=\pm 1}\varphi(\xi)/2$ with a Schwartz function $\varphi$.

By (\ref{react0}), we then write down the equation
\begin{multline}\label{react01}\partial_t\widehat{f'}(t,k,\xi)=\int_{\mathbb{R}}\frac{\varepsilon_0\eta/2}{(\eta-t(k+1))^2+(k+1)^2}\widehat{\varphi}(\xi-\eta)\widehat{f'}(t,k+1,\eta)\,\mathrm{d}\eta\\-\int_{\mathbb{R}}\frac{\varepsilon_0\eta/2}{(\eta-t(k-1))^2+(k-1)^2}\widehat{\varphi}(\xi-\eta)\widehat{f'}(t,k-1,\eta)\,\mathrm{d}\eta
\end{multline} for $\widehat{f'}$. In \cite{BM15}, the authors replaced the function $\varphi$ on the right hand side of (\ref{react01}) by the $\delta$ function, obtaining an ODE \emph{toy model} which is essentially an ``envelope'' of (\ref{react01}) and can be solved explicitly. This is perfect for obtaining an \emph{upper bound} for solutions to (\ref{react01}), but in order to get a \emph{lower bound} a more accurate approximation will be needed - which is precisely what we are able to obtain here, under the assumption $\eta\approx\eta_0$ and $t\in[T_2,T_1]$.
Of course, we will also prove that the Fourier transform of $f'$ is localized around $\eta_0$. 

For simplicity, let us assume $t\sim T_0$; recall from (\ref{defpara}) that $T_0\sim\sqrt{\eta_0/\varepsilon_0}$. Since $\eta\approx\eta_0$ due to the definition of $f'(T_0)$, we know that (\ref{react01}) plays a significant role only near the critical times $\eta_0/m$, where $m\sim\sqrt{\varepsilon_0\eta_0}$. We thus cut the time interval into subintervals, each containing exactly one critical time. Define, see also (\ref{defineint}),
\[t_m=\frac{2\eta_0}{2m+1};\qquad \frac{\eta_0}{m}\in[t_m,t_{m-1}]:=I_m; 
\] then on each $I_m$, according to (\ref{react01}), only the  mode $k=m$ will be active and hence the 
 modes 
 $k=m\pm 1$ will get significant increments. Ineed,  \[\frac{1}{(\eta_0-kt)^2+k^2}\lesssim\frac{1}{t^2}\ll\frac{1}{m^2},\quad\forall t\in I_m,k\neq m.\] We can therefore solve (\ref{react01}) approximately and explicitly\footnote{Note that this argument works precisely when $t\in[T_2,T_1]$: when $t$ si too small transport terms will come in, and when $t$ is too large the $(m\pm 1)$ modes $\widehat{f'}(t,m\pm 1,\xi)$ will grow too much and destroy the approximate decoupling.}, obtaining an approximate recurrence relation (see (\ref{defpara}) for definition of parameters):
 \begin{equation}\label{recur0200}\mathscr{F}_zf'(t_{m-1},k,v)=\mathscr{F}_zf'(t_m,k,v)+\mathcal{R}+\left\{
 \begin{aligned}&0,&k&\neq m\pm 1;\\
 \mp\frac{\alpha k_0^2}{m^2}\varphi(v)&\cdot \mathscr{F}_zf{'}(t_{m},m,v),&k&=m\pm 1,
 \end{aligned}
 \right.
 \end{equation} after taking inverse Fourier transform in $\xi$, where the error term $\mathcal{R}$ is small in $L^2$; see Propositions \ref{newtoyback} and \ref{errorsmallprop}.
 
The recurrence relation (\ref{recur0200}) then plays the role of the toy model in \cite{BM15}. In fact, if we choose $\varphi$ such that $\|\varphi\|_{L^\infty}=1$, then this already suffices to prove the upper bound on $[T_2,T_0]$, since (\ref{recur0200}) essentially implies that
\[\sup_{k}\|\mathscr{F}_zf'(t_{m},k,\cdot)\|_{L^2}\leq\max\bigg(1,\frac{\alpha k_0^2}{m^2}\bigg)\cdot \sup_{k}\|\mathscr{F}_zf'(t_{m-1},k,\cdot)\|_{L^2},
\] and thus by iteration,
\begin{equation}\label{backgrow}\sup_{k}\|\mathscr{F}_zf'(T_2,k,\cdot)\|_{L^2}\leq\varepsilon_1\prod_{m=k_2}^{k_0}\max\bigg(1,\frac{\alpha k_0^2}{m^2}\bigg)\sim e^{c'\sqrt{\eta_0}}\varepsilon_1.
\end{equation} See Proposition \ref{interval30} for details.

We turn to the lower bound for $f'$ on $[T_0,T_1]$. If $\varphi$ were identically $1$, then in view of the \emph{smallness} of $\mathcal{R}$, we can use the same argument to obtain that 
\[\sup_{k}\|\mathscr{F}_zf'(t_{m-1},k,\cdot)\|_{L^2}\geq\max\bigg(1,\frac{\alpha k_0^2}{m^2}\bigg)\cdot \sup_{k}\|\mathscr{F}_zf'(t_{m},k,\cdot)\|_{L^2},
\] and hence
\begin{equation}\label{backgrow2}\sup_{k}\|\mathscr{F}_zf'(T_1,k,\cdot)\|_{L^2}\geq \varepsilon_1\prod_{m=k_0}^{k_1}\max\bigg(1,\frac{\alpha k_0^2}{m^2}\bigg) \sim e^{c\sqrt{\eta_0} } \varepsilon_1.
\end{equation} Comparing (\ref{backgrow}) and (\ref{backgrow2}) we obtain the desired inequality (\ref{compare}) by direct computations, due to our choice of parameters.

Nevertheless $\varphi$ cannot be identically $1$, and moreover the error term $\mathcal{R}$ is not local. To recover (\ref{backgrow2}), in view of the factor $\varphi(v)$ on the right hand side of (\ref{recur0200}), we thus need to localize $v$ in the region where $\varphi(v)$ is equal or close to $1$. This localization is achieved by switching to physical space and performing an energy-type estimate for an $L^2$ norm with exponential weights of $v$; see Proposition \ref{phys} for details.
\subsubsection{Nonlinear analysis, and the Taylor expansion} Up to now we have only considered $f'$, which is the solution to the linearized system $\partial_tf'=\mathcal{L}f'$. The full nonlinear system (\ref{eulernew})$\sim$(\ref{eulersup}), in terms of $f^*$, can be written as 
\begin{equation}\label{onetwo}\partial_ff^*=\mathcal{L}f^*+\mathcal{N}(f^*,f^*),
\end{equation} if, say, we consider only quadratic nonlinearities. Note that $f'$ can also be regarded as the first order term in a formal Taylor expansion of $f^*$; we may write out the higher order terms by $f^{(1)}=f'$ and 
\[\partial_tf^{(n)}=\mathcal{L}f^{(n)}+\sum_{q_1+q_2=n-1}\mathcal{N}(f^{(q_1)},f^{(q_2)}),\quad f^{(n)}(T_0)=0;
\] for a precise description, see Section \ref{taylorexp}. Our next step is to prove that, in some sense, we have\footnote{Note however that $f^{(n)}$ is supported at higher and higher frequencies, namely $(nk_0,n\eta_0)$; thus this fact cannot be captured by a bootstrap argument in a single Gevrey norm, and this formal Taylor expansion seems necessary.}
\begin{equation}\label{higher00}(\textnormal{the size of }f^{(n)})\lesssim(\textnormal{the size of }f^{(1)})^n,
\end{equation} see Proposition \ref{l2allgn}. Since the size of $f^{(1)}$ is $O(\varepsilon_1)$, the bound (\ref{higher00}) guarantees that the contribution of $f^{(n)}$ with $n\geq 2$ will be negligible, and thus Theorem \ref{main} follows from the estimates for $f^{(1)}$ obtained above.

The proof of (\ref{higher00}) follows from an inductive argument, where at each step we combine the multilinear estimates for $\mathcal{N}$ (see Propositions \ref{multilnearest} and \ref{multilinearest2}) with the linear estimates for the inhomogeneous equation
\[\partial_tf=\mathcal{L}f+\mathcal{N},
\] which is proved in the same way as the linear homogeneous case. Here the main difficulty is that $f^{(n)}$, being essentially the $n$-th power of $f^{(1)}$, is supported in Fourier space at (say) the frequency $(nk_0,n\eta_0)$. We thus need to run the arguments above for this particular choice of frequency, instead of $(k_0,\eta_0)$. Fortunately this just corresponds to changing of parameters in the Orr growth mechanism, and most of the arguments above still go through; see Section \ref{largeeta} for details. There are few exceptional cases, though, and they can be treated by using the energy-type estimates in Propositions \ref{energy1}$\sim$\ref{lastone}.

Finally, to avoid the divergence issue caused by doing the Taylor expansion directly, we will close the whole proof by fixing some very large $n_0$ and claiming that
\[f^{(1)}+f^{(2)}+\cdots+f^{(n_0)}
\] is an approximate solution to (\ref{onetwo}), with error term so small that an actual solution to (\ref{onetwo}) can be constructed by a perturbative argument on the interval $[1,T_1]$. This is done in Section \ref{error}. 
\subsection{Further discussions} We mention two possible further questions related to Theorem \ref{main}.
\subsubsection{Asymptotic instability} Given Theorem \ref{main}, an immediate question is whether asymtotic instability can also be proved for (\ref{euler}). We believe this can be done by repeatedly applying the arguments in this paper.

Roughly speaking, we fix the background solution $\underline{f}$ and construct the perturbation $f^*=f_1^*$ as in Theorem \ref{main}. Note that $f_1^*$ grows from some time $T_0^1$ to some later time $T_1^1$; We now take $\underline{f}+f_1^*$ as the new background and construct a further perturbation $f_2^*$ which grows from time $T_0^2$ to $T_1^2$, and so on. We then pile up a sequence of perturbations and define
\[f:=\underline{f}+f_1^*+f_2^*+\cdots,
\] which we expect to satisfy that
\[\|f(1)\|_{\mathcal{G}^*}\leq\varepsilon,\quad \lim_{t\to\infty}\|\langle \partial_x\rangle^{N_0}f(t)\|_{L^2}=+\infty.
\]	

The main difficulty here is to control the evolution of $f_1^*$ after time $T_1^1$; we then have to extend our arguments, which now covers only critical times $\eta_0/m$ with $m\gtrsim k_0$, to \emph{all} critical times up to $m=1$. We believe that a suitable combination of the techniques used in this paper and the weighted energy method used in \cite{BM15} should be the key to solving this problem.
\subsubsection{Genericity} Another natural question is whether the Orr growth mechanism is generic, i.e., whether the full upper bound of growth can be saturated for ``most'' solutions in a suitable sense.

To study this problem, we have to consider solutions with general distribution in frequencies, instead of the $f^*$ we choose here, which essentially has only two modes. In such cases we no longer have the simple decoupling as in Section \ref{precisetoy}, nor the recurrence relation (\ref{recur0200}); the main challenge is thus to find a substitute to (\ref{recur0200}) and to approximate (\ref{react01}), and it would be crucial to be able to separate the different components of the solution that evolve differently. It seems that some further physical-space based techniques will be needed.

Another challenge is the possible cancellations for the toy model (if we can find one) in the generic case. This also depends on how well  different frequencies and different physical space locations are  separated - if they are mixed together then we would have less control of the solution.
\subsection{Plan of the paper} In Section \ref{changeofcoor} we preform the change of variables, transform (\ref{euler}) to (\ref{eulernew})$\sim$(\ref{eulersup}) and reduce Theorem \ref{main} to Proposition \ref{main1.5}. In Section \ref{buy} we fix the norms and parameters needed in the proof, and also fix the target solution by constructing the data at time $T_0$. In Section \ref{linearize} we linearize (\ref{eulernew})$\sim$(\ref{eulersup}) at the background solution, and collect some supporting estimates for this linear system. The proof of these estimates will be left to Appendix \ref{appproof}.

In Section \ref{for}, which is the core part of the proof, we analyze the linearized system and obtain the desired upper and lower bounds for solutions to this system. In Section \ref{largeeta} we apply similar arguments as in Section \ref{for} to the corresponding inhomogeneous linear system and obtain control of the solutions; this estimate is then used in Section \ref{approximate} to construct, via Taylor expansion, an approximate solution to (\ref{eulernew})$\sim$(\ref{eulersup}). Finally, in Section \ref{error}, we construct an exact solution to (\ref{eulernew})$\sim$(\ref{eulersup}) based on this approximate solution, thus finishing the proof of Theorem \ref{main}.

\section{Choice of coordinates}\label{changeofcoor}
Before coming to the analysis we first introduce, in a similar way as \cite{BM15}, a change of variables that reduces \eqref{euler} to a system whose structure is more transparent.

Let $\omega=\omega(t,x,y)$, and let $\gamma=\gamma(y)$ be a function of $y$. For $t\geq 1$ we make the change of variables $(t,x,y)\mapsto (t,z,v)$, where
\begin{equation}\label{changeva}
\left\{
\begin{aligned}
z(t,x,y)&=x-tv(t,y),\\
v(t,y)&=y-\frac{1}{t}\bigg(\gamma(y)+\int_{1}^t (\partial_y^{-1}\mathbb{P}_0\omega)(\tau,y)\,\mathrm{d}\tau\bigg),
\end{aligned}
\right.
\end{equation}
and $\mathbb{P}_0$ is the projection onto the zero frequency in $x$ variable, and $\partial_y^{-1}F$ denotes the unique antiderivative of $F$ that vanishes at infinity, as long as $F$ has vanishing integral (which $\mathbb{P}_0\omega$ does due to the second equation in \eqref{auxassum}. Define the new functions
\begin{equation}\label{newfunc}
f(t,z,v)=\omega(t,x,y),\quad h(t,v)=\partial_yv(t,y)-1,\quad\theta(t,v)=\partial_tv(t,y).
\end{equation}
\begin{proposition}\label{mapping}
Let $J\subset[1,+\infty)$ be an interval, $\omega$ be a solution to \eqref{euler} on $J$ that satisfies the second equation in \eqref{auxassum} for any $t\in I$, and $\gamma(y)$ be a function of $y$ such that \begin{equation}\label{termgm}\gamma(\pm\infty):=\lim_{y\to\pm\infty}\gamma(y)=0.\end{equation} 

Define $(f,h,\theta)$ as in \eqref{changeva} and \eqref{newfunc}, then they satisfy the following system of equations:
\begin{equation}\label{eulernew}
\left\{
\begin{aligned}
\partial_tf&=-\theta\cdot\partial_vf-(h+1)\nabla^{\perp}\phi\cdot\nabla f,\\
\partial_th&=-\theta\partial_vh-\frac{\mathbb{P}_0f+h}{t},\\
\partial_t\theta&=-\frac{2\theta}{t}-\theta\partial_v\theta+\frac{1}{t}\mathbb{P}_0(f\cdot\partial_z\phi),
\end{aligned}
\right.
\end{equation}
where the relevant quantities are defined as 
\begin{equation}\label{eulersup}
\left\{
\begin{aligned}
\phi&=\mathbb{P}_{\neq 0}\Delta_t^{-1}f,\\
\Delta_t&=\partial_z^2+(h+1)^2(\partial_v-t\partial_z)^2+(h+1)(\partial_vh)(\partial_v-t\partial_z).
\end{aligned}
\right.
\end{equation} Moreover, for $t\in I$ we have
\begin{equation}\label{aux}
t(h+1)\partial_v\theta(t)+\mathbb{P}_0f(t)+h(t)=0;\quad \int_{\mathbb{R}}\frac{h(t,v)}{h(t,v)+1}\,\mathrm{d}v=0.
\end{equation}
 
 Conversely, suppose $(f,h,\theta)$ is a solution to the system \eqref{eulernew}$\sim$\eqref{eulersup} on $J$ that satisfies \eqref{aux} at one time of $J$, then \eqref{aux} is true on all of $J$, and there exists a unique solution $\omega$ to \eqref{euler} on $J$ that satisfies the second equation in \eqref{auxassum}, and a unique $\gamma$ satisfying $\gamma(\pm\infty)=0$, such that \eqref{changeva} and \eqref{newfunc} hold.
\end{proposition}
\begin{proof} Much of the computation is contained in \cite{BM15}, we include a proof here for the sake of completeness. Given $\omega$ and $\gamma$, by \eqref{changeva} we have
\[
\partial_tv(t,y)+\frac{v(t,y)}{t}+\frac{\partial_y^{-1}\mathbb{P}_0\omega(t,y)}{t}=\frac{y}{t},
\]
taking a derivative in $y$ yields
\[
\partial_th(t,v(t,y))+\frac{h(t,v(t,y))}{t}+\frac{\mathbb{P}_0\omega(t,y)}{t}=0,
\]
which gives that
\[
\partial_th+\frac{\mathbb{P}_0f+h}{t}+\partial_vh\cdot\theta=0,
\] noticing that $\theta(t,v(t,y))=\partial_tv(t,y)$. This proves the second equality in \eqref{eulernew}. For the derivative of $f$, we compute
\[
\partial_tf(t,z,v)=\partial_t\omega+v\cdot\partial_x\omega+\partial_ty\cdot\partial_y\omega.
\]
We next have $\partial_ty=-(h+1)^{-1}\theta$, and
\[\partial_x\omega=\partial_zf,\quad \partial_y\omega=(h+1)(\partial_v-t\partial_z)f.
\] 
Moreover let $\nu=\Delta^{-1}\omega$, then $(\mathbb{P}_{\neq 0}\nu)(t,x,y)=\phi(t,z,v)$, where $\phi=\mathbb{P}_{\neq 0}\Delta_t^{-1}f$.
Plugging into the equation \eqref{euler}, we get that
\[
\begin{split}
\partial_tf&=-y\cdot \partial_zf-\nabla^{\perp}\nu\cdot\nabla\omega+v\cdot\partial_zf-\theta(\partial_v-t\partial_z)f\\
&=-\theta\cdot\partial_vf-(h+1)\nabla^{\perp}\phi\cdot\nabla f+(v-y+t\theta+\partial_y^{-1}\mathbb{P}_0\omega(t,y))\cdot \partial_zf,
\end{split}
\] and the last term is zero due to \eqref{changeva}, which proves the first equality in \eqref{eulernew}. For the derivative of $\theta$, differentiating in $t$ the equality $\theta(t,v(t,y))=\partial_tv(t,y)$ yields
\[
\partial_t\theta+\theta\cdot \partial_v\theta=\partial_t^2v(t,y)=\partial_t\bigg(\frac{y-v(t,y)}{t}-\partial_y^{-1}\frac{\mathbb{P}_0\omega(t,y)}{t}\bigg)=-\frac{2\theta}{t}-\frac{1}{t}\partial_t\partial_y^{-1}\mathbb{P}_0\omega,
\] where using \eqref{euler} one computes that
\[
-\partial_t\partial_y^{-1}\mathbb{P}_0\omega=\partial_y^{-1}\mathbb{P}_0(\partial_x\nu\cdot\partial_y\omega-\partial_y\nu\cdot\partial_x\omega)=\mathbb{P}_0(\omega\cdot\partial_x\nu)
\] by integrating by parts in $x$, and thus
\[
\partial_t\theta+\theta\cdot \partial_v\theta+\frac{2\theta}{t}=\frac{1}{t}\mathbb{P}_0(f\cdot\partial_z\phi),
\] which proves the third equality in \eqref{eulernew}. Moreover, differentiating in $y$ the equality $\theta(t,v(t,y))=\partial_tv(t,y)$ we get
\[
(h+1)\partial_v\theta=\partial_t\partial_yv(t,y)=\partial_th(t,v(t,y))=\partial_th+\theta\partial_vh=\frac{-\mathbb{P}_0f-h}{t},
\]
so the first equality in \eqref{aux} holds. For the second equality, just notice that by \eqref{changeva} and the fact that $\partial_y^{-1}\mathbb{P}_0\omega(t,\pm\infty)=0$ (which follows from the definition of $\partial_y^{-1}$) we have
\[
\lim_{y\to\pm\infty}(v(t,y)-y)=0,
\] and thus
\[
0=\int_{\mathbb{R}}(\partial_yv(t,y)-1)\,\mathrm{d}y=\int_{\mathbb{R}}\frac{h(t,v)}{h(t,v)+1}\,\mathrm{d}v.
\]

Conversely, given $(f,h,\theta)$ that solves \eqref{eulernew}$\sim$\eqref{eulersup}, let $\chi=t(h+1)\partial_v\theta+\mathbb{P}_0f+h$, we compute
\begin{align*}
\partial_t\chi&=(h+1)\partial_v\theta+t\big((h+1)\partial_v\partial_t\theta+\partial_th\cdot\partial_v\theta\big)+\mathbb{P}_0\partial_tf+\partial_th\\
&=\partial_v\theta\cdot(h+1-t\theta\partial-vh-\mathbb{P}_0f-h)-\theta\cdot\partial_vh-\frac{\mathbb{P}_0f+h}{t}+\mathbb{P}_0\partial-tf+t(h+1)\partial_v\partial_t\theta\\
&=\partial_v\theta\cdot(1-\mathbb{P}_0f-t\theta\partial_vh)-\theta\cdot\partial_vh-\frac{\mathbb{P}_0f+h}{t}-2(h+1)\partial_v\theta-t(h+1)\partial_v(\theta\partial_v\theta)\\
&-(h+1)\mathbb{P}_0(\partial_z\phi\cdot\partial_vf-\partial_v\phi\cdot\partial_zf)+(h+1)\partial_v\mathbb{P}_0(f\cdot\partial_z\phi)\\
&=-\theta\partial_v(\mathbb{P}_0f+h)-\theta\cdot t(h+1)\partial_v^2\theta-t(h+1)(\partial_v\theta)^2-\frac{\mathbb{P}_0f+h}{t}+\partial_v\theta\cdot(-2h-1-\mathbb{P}_0f-t\theta\partial_vh)\\
&+(h+1)\mathbb{P}_0\partial_z(f\cdot\partial_v\phi)\\
&=-\theta\partial_v\chi+\partial_v\theta\cdot(-2h-1-\mathbb{P}_0f)-\frac{\mathbb{P}_0f+h}{t}-t(h+1)(\partial_v\theta)^2\\
&=-\theta\partial_v\chi-\frac{\chi}{t}+\partial_v\theta(-h-\mathbb{P}_0f)-t(h+1)(\partial_v\theta)^2\\
&=-\partial_v(\theta\chi)-\frac{\chi}{t},
\end{align*}
 so if $\chi$ vanishes at one time, it must vanish at all times. Moreover, when $\chi=0$ we have
\[
\partial_t\int_{\mathbb{R}}\frac{h}{h+1}=\int_{\mathbb{R}}\frac{-\partial_th}{(h+1)^2}=\int_{\mathbb{R}}\frac{1}{(h+1)^2}\bigg(\theta\partial_vh+\frac{\mathbb{P}_0f+h}{t}\bigg)=\int_{\mathbb{R}}\frac{\chi}{t(h+1)^2}=0,
\]
so the second equality in \eqref{aux} is also preserved. Now when \eqref{aux} holds, we define for each $t$ the diffeomorphism $v=v(t,y)$ such that $\partial_yv(t,y)=h(t,v(t,y))+1$, and
\[
\lim_{y\to\pm\infty}(v(t,y)-y)=0.
\] Such $v$ exists due to the second equality in \eqref{aux}; if one defines $\omega$ by 
\[
\omega(t,x,y)=f(t,z,v):=f(t,x-tv(t,y),v(t,y)),
\] then by essentially reverting the arguments above, one can show that \eqref{newfunc} holds, and $\omega$ solves \eqref{euler}. The second equality in (\ref{auxassum}) is a consequence of (\ref{aux}). Moreover, via similar arguments one can show 
\[
\partial_t(tv(t,y))=y-\partial_y^{-1}\mathbb{P}_0\omega(t,y),
\] so there exists $\gamma(y)$ satisfying \eqref{changeva}, clearly they are unique. This completes the proof.
\end{proof}
With Proposition \ref{mapping}, we can then reduce Theorem \ref{main} to the following
\begin{proposition}\label{main1.5} There exists a time $T$ and solution $(f,h,\theta)$ to (\ref{eulernew})$\sim$(\ref{eulersup}) on $[1,T]$ that satisfies (\ref{aux}), and that
\begin{equation}\label{newbd}
\|(f(1),h(1))\|_{\mathcal{G}_{3/2}^*}\leq\varepsilon^2,\quad \int_{\mathbb{T}\times\mathbb{R}}(1+|v|)\cdot|f(1,z,v)|\,\mathrm{d}z\mathrm{d}v\leq \varepsilon^2;
\end{equation}
\begin{equation}\label{newbd2}\big\|\langle \partial_z\rangle^{N_0}f(T)\big\|_{L^2}\geq \varepsilon^{-2},\quad \|h(T)\|_{H^{20}}\leq \varepsilon^2.
\end{equation}
\end{proposition}
 \begin{proof}[Proof of Theorem \ref{main} assuming Proposition \ref{main1.5}] Suppose we have the solution $(f,h,\theta)$ as in Proposition \ref{main1.5}. By Proposition \ref{mapping} we can find a solution $\omega=\omega(t,x,y)$ to \eqref{euler} satisfying the second equality in \eqref{auxassum} and $\gamma=\gamma(y)$ such that \eqref{changeva} and \eqref{newfunc} hold.
 
 The change of coordinates $v=v(t,y)$ (and the corresponding inverse $y=y(t,v)$) is uniquely determined by 
 \[\frac{\partial v}{\partial y}=h(v)+1,\quad \lim_{y\to\pm\infty}(v(t,y)-y)=0,
 \] and therefore at time $T$ we have
 \begin{equation}\label{difffin}\|\partial_vy(T,v)-1\|_{L^{\infty}}\leq C\varepsilon^2. 
 \end{equation} Notice that \[\omega(T,x,y)=f(T,x-Tv(T,y),v(T,y)),\] and since the change of coordinates $v\mapsto y$ has Jacobian close to $1$, we have
 \[\|\langle\partial_x\rangle^{N_0}\omega(T)\|_{L^2(\mathbb{T}\times\mathbb{R})}\geq\frac{1}{2}\big\|\langle \partial_z\rangle^{N_0}f(T)\big\|_{L^2(\mathbb{T}\times\mathbb{R})}\geq \frac{1}{\varepsilon}.
 \]
 
Similarly, using \eqref{newbd}, at time $t=1$ we can prove that 
\[\|\omega(1)\|_{\mathcal{G}^*}=\|\omega(1)\|_{\mathcal{G}_1^*}\lesssim\|f(1)\|_{\mathcal{G}_1^*}\lesssim \varepsilon^2,
\] using properties of compositions and inverses in Gevrey spaces which are well-known (except that we are working in slightly $\log$-corrected Gevrey spaces, but the proof still goes through). Also, using the second inequality in \eqref{newbd} and the fact that $v(1,y)\approx y$, we get that
\begin{equation}\label{local}\int_{\mathbb{T}\times\mathbb{R}}(1+|y|)\cdot|\omega(1,x,y)|\,\mathrm{d}x\mathrm{d}y\lesssim \varepsilon^2.
\end{equation}
Since the Euler equation \eqref{euler} is locally well-posed in $\mathcal{G}^*$, we can extend the solution to time $0$, and $\|\omega(0)\|_{\mathcal{G}^*}\lesssim\varepsilon^2\ll\varepsilon$. By a simple Gronwall argument, one can also show that \eqref{local} implies the first inequality in \eqref{auxassum}. We omit the details.
 \end{proof}

\begin{remark}\label{simnota}For convenience, below we will use the notation
\[g=(g_1,g_2,g_3)=(f,h,\theta),
\] and possibly with sub or superscripts, which will be introduced later. We may also occasionally regard $h$ and $\theta$, and their variants with sub or superscripts, as functions of $(t,z,v)$ that do not depend on $z$.
\end{remark}

\section{The target solution}\label{buy} In this section we will define the solution $g=(f,h,\theta)$ to the system \eqref{eulernew}$\sim$\eqref{eulersup}, and reduce Proposition \ref{main1.5} to Proposition \ref{main2}, which will then be proved in Sections \ref{linearize} to \ref{error}.
\subsection{Notations} We start by fixing some notations.
\subsubsection{Cutoff functions}
It is well known that there exists a function $\psi=\psi(\xi)$ such that $\psi$ is compactly supported and that
\begin{equation}\label{decay00}
|\mathscr{F}^{-1}\psi(v)|\lesssim \exp\bigg(-\frac{|v|}{(\log^+|v|)^2}\bigg).
\end{equation}
Denote by $\mathcal{E}$ the space of functions $\psi$ satisfying \eqref{decay00}; by standard convolution tricks, one can show that, for any fixed compact set $K\subset\mathbb{R}$ and any open set $U\supset K$, there exists $\psi\in\mathcal{E}$ such that $\psi=0$ in $K$ and $\psi=1$ in $U^c$. Below, $\psi$ will be used to denote such cutoff functions only.
\subsubsection{Choice of norms and parameters}\label{choicenorm}
Recall $N_0=9000$ and $N_1=60000$, we also define
\begin{equation}
N=3000,\quad N'=30,\quad N_2=30,\quad N_3=30000,
\end{equation} so we have $N_1=2N_3$, $N_3=1000N_2$ and $N=100N'$. Fix a large enough constant $C_0$; \emph{unless explicitly stated}, all the implicit constants are allowed to depend on $C_0$.

We use $C$ to denote any absolute constant that depends only on $C_0$, and $D$ to denote any absolute constant that satisfies $D\gg C$. Occasionally we will also need an absolute constant $E$ such that $D\gg E\gg C$; all of them will ultimately depend only on $C_0$.

Given $\varepsilon>0$ as in the statement of Theorem \ref{main}, we choose a positive integer $k_0$ large enough depending on $(C_0,\varepsilon)$, and define $(\sigma,\alpha,\varepsilon_0,\eta_0,T_0)$ by 
\begin{equation}\label{defpara}
\begin{aligned}
\sigma&=(\log k_0)^{-N_2},&\alpha&=1+(\log k_0)^{-2N_2},&\varepsilon_0&=(\log k_0)^{-N_3};\\
\eta_0&=\frac{2k_0^2\alpha}{\pi\varepsilon_0},&T_0&=\frac{2\eta_0}{2k_0+1}.
\end{aligned}
\end{equation}
Note that $\alpha-1=\sigma^2$ and $\varepsilon_0=\sigma^{1000}$. For later uses we will also define 
\begin{equation}\label{defineint}
\begin{aligned}
t_{m}&=\frac{2\eta_0}{2m+1},&T_j&=t_{k_j},\,\,1\leq j\leq 3;\\
\quad k_1&=(1-\sigma)k_0,& k_2&=\varepsilon_0^{-1/40}\sqrt{\eta_0}, &k_3&=\varepsilon_0^{-1/40}k_0.
\end{aligned}
\end{equation}
Note that $k_2>k_0>k_1$ and $T_1>T_0>T_2$. For simplicity we will assume all the $k_i$'s are integers (otherwise take their integer parts).

Recall the norm $\|\cdot\|_{\mathcal{G}^*}$ defined in \eqref{gevrey}; we will define its variant $\|\cdot\|_{\mathcal{G}_{\lambda}^*}$ by 
\begin{equation}\label{defglambda}\|F\|_{\mathcal{G}_{\lambda}^{*}}^{2}=\sum_{k\in\mathbb{Z}}\int_{\mathbb{R}}e^{2\lambda\kappa(k,\xi)}|\widehat{F}(k,\xi)|^2\,\mathrm{d}\xi,
\end{equation}so $\mathcal{G}^*=\mathcal{G}_1^*$. We also define the analytic norm $\mathcal{A}_\lambda$ for $F:\mathbb{T}\times\mathbb{R}\to\mathbb{R}$ by
 \begin{equation}
 \|F\|_{\mathcal{A}_{\lambda}}^2=\sum_{k\in\mathbb{Z}}\int_{\mathbb{R}}e^{2\lambda(|k|+|\xi|)}|\widehat{F}(k,\xi)|^2\,\mathrm{d}\xi.
 \end{equation} Similarly we can define the same norms for $F:\mathbb{R}\to\mathbb{R}$ (without summation in $k$). Throughout this paper the Fourier transform will be defined as
 \begin{equation}
 \widehat{F}(k,\xi)=\frac{1}{(2\pi)^2}\int_{\mathbb{T}\times\mathbb{R}}e^{-i(kz+\xi v)}F(z,v)\,\mathrm{d}z\mathrm{d}v.
 \end{equation}
 We also use the notation $\mathscr{F}(F)$, and define one-dimensional versions $\mathscr{F}_z(F)$ and $\mathscr{F}_v(F)$ similarly. For a statement $Q$, $\mathbf{1}_{Q}$ will denote the function that equals $1$ if $Q$ is true and $0$ otherwise.
  
\subsection{Choice of data}
The target solution $g=(f,h,\theta)$ to \eqref{eulernew}$\sim$\eqref{eulersup} will be constructed as the superposition of an analytic background $\underline{g}$,  which (essentially) has only two modes and a perturbation.

\subsubsection{The background solution}\label{background}
We fix the function
\[\varphi_b(v)=e^{-(C_0^{-1}v)^{18}}.\] Define the background solution $\underline{g}:=(\underline{f},\underline{h},\underline{\theta})$, which solves \eqref{eulernew}$\sim$\eqref{eulersup} with initial data 
\begin{equation}
\underline{f}(1,z,v)=\varepsilon_0\cos z\cdot\varphi_b(v),\quad \underline{h}(1,z,v)=\underline{\theta}(1,z,v)=0.
\end{equation} By Theorem \ref{backthm}, we know that $\underline{g}$ exists on $[1,+\infty)$, and satisfies the following properties, where recall that all constants here depend on $C_0$:
\begin{enumerate}
\item $\underline{f}$ and $\underline{h}$ are real-valued and even, $\underline{\theta}$ is real-valued and odd, and
\begin{equation}\label{ass1}
\|\underline{f}(t)\|_{\mathcal{A}_{C_0}}+\|\underline{h}(t)\|_{\mathcal{A}_{C_0}}+\|\underline{\theta}(t)\|_{\mathcal{A}_{C_0}}\lesssim \varepsilon_0,\quad \|\underline{\theta}(t)\|_{\mathcal{A}_{C_0-1}}\lesssim\frac{\varepsilon_0}{t^2};
\end{equation}
\item $\underline{f}$ and $\underline{h}$ converge as $t\to\infty$,
\begin{equation}\label{ass2}
\|\underline{f}(t)-f_{\infty}\|_{\mathcal{A}_{C_0-1}}\lesssim \frac{\varepsilon_0^2}{t},\quad \|\underline{h}(t)+\mathbb{P}_0f_{\infty}\|_{\mathcal{A}_{C_0-1}}\lesssim \frac{\varepsilon_0}{t};
\end{equation}
\item The limit $f_{\infty}$ is close to the specific profile we choose, namely
\begin{equation}\label{ass3}
\big\|f_{\infty}-\varepsilon_0\cos z\cdot\varphi_b(v)\big\|_{\mathcal{A}_{C_0-1}}\lesssim\varepsilon_0^2.
\end{equation}
\end{enumerate}
We also define the function $\underline{\phi}$ and the operator $\underline{\Delta_t}$, corresponding to $(\underline{f},\underline{h},\underline{\theta})$, as in \eqref{eulersup}.

Using the equations \eqref{eulernew}$\sim$\eqref{eulersup} once more, one can deduce from \eqref{ass1}$\sim$\eqref{ass3} the more convenient pointwise estimates in Fourier space, namely:
\begin{align}\label{useass1}
\big|\widehat{\underline{f}}(t,k,\xi)\big|&\lesssim \varepsilon_0 e^{-(C_0-2)(|k|+|\xi|)},& \big|\widehat{\underline{f}}(t,k,\xi)-\widehat{f_{\infty}}(k,\xi)\big|&\lesssim \varepsilon_0^2 \frac{e^{-(C_0-2)(|k|+|\xi|)}}{t},\\
\label{useass2}\big|\widehat{f_\infty}(k,\xi)\big|&\lesssim \varepsilon_0^2 e^{-(C_0-2)(|k|+|\xi|)},|k|\neq 1;&\bigg|\widehat{f_\infty}(\pm 1,\xi)-\frac{\varepsilon_0\widehat{\varphi_b}(\xi)}{2}\bigg|&\lesssim \varepsilon_0^2 e^{-(C_0-2)(|k|+|\xi|)},\\
\label{useass3}\big|\widehat{\underline{h}}(t,\xi)\big|&\lesssim \varepsilon_0 e^{-(C_0-2)|\xi|}, &\big|\widehat{\underline{h}}(t,\xi)+\widehat{f_{\infty}}(0,\xi)\big|&\lesssim\varepsilon_0\frac{e^{-(C_0-2)|\xi|}}{t},
\\
\label{useass4}\big|\widehat{\underline{\theta}}(t,\xi)\big|&\lesssim \varepsilon_0 \frac{e^{-(C_0-2)|\xi|}}{t^2}, &\big|\widehat{\underline{\phi}}(t,k,\xi)\big|&\lesssim \varepsilon_0 \mathbf{1}_{k\neq 0}\frac{e^{-(C_0-2)(|k|+|\xi|)}}{t^2}.
\end{align}
\subsubsection{The perturbation}\label{perturbation}
Recall $T_0$ and $T_1$, and other parameters defined in \eqref{defpara} and \eqref{defineint}. We perturb the solution $\underline{g}$ by assigning new data at time $T_0$, and define a solution $g=(f,h,\theta)$ to \eqref{eulernew}$\sim$\eqref{eulersup} by
\begin{equation}\label{perturbdata}
g(T_0)=\underline{g}(T_0)+(\varepsilon_1 \cos(k_0z+\eta_0v)\varphi_p(k_0\sqrt{\sigma}v),0,0),
\end{equation}
where $\varphi_p$ is a fixed, real valued Schwartz function such that 
\begin{equation}
\widehat{\varphi_p}\in\mathcal{E}, \quad\mathrm{supp}(\widehat{\varphi_p})\subset[-1,1], \quad\|\varphi_p\|_{L^2}=1,
\end{equation}
 and $0<\varepsilon_1\ll\varepsilon_0$ is a parameter that will be fixed later in Proposition \ref{choiceeps1}. Proposition \ref{main1.5} is then reduced to the  following
\begin{proposition}\label{main2}
There exists $\varepsilon_1$, satisfying $e^{-2\sigma^2k_0}\leq\varepsilon_1\leq e^{-\sigma^2k_0/2}$, such that \eqref{eulernew}$\sim$\eqref{eulersup} has a unique smooth solution $g=(f,h,\theta)$ on $[1,T_1]$ with data $g(T_0)$ as in \eqref{perturbdata} at time $T_0$. Moreover, this solution satisfies the conclusions of Proposition \ref{main1.5}, with $T=T_1$.
\end{proposition}
\section{Linearization of the transformed system} \label{linearize} In this section we will linearize \eqref{eulernew}$\sim$\eqref{eulersup} at $\underline{g}$, and list some basic properties of this linearized system. They will be used as the base of the more precise analysis in Sections \ref{for} and \ref{largeeta}. The proofs are somewhat technical and not very related to the main ideas, so we leave them to Appendix \ref{appproof}.

Let $g'=(f',h',\theta')$ denote the unknown functions in this linearized system, then the system can be written as:
\begin{equation}\label{eulernewl}
\left\{
\begin{aligned}
\partial_tf'&=&&\!\!\!\!\!-\underline{\theta}\cdot\partial_v f'-\theta'\partial_v\underline{f}-(\underline{h}+1)\nabla^{\perp}\underline{\phi}\cdot \nabla f'-\big((\underline{h}+1)\nabla^{\perp}\phi'+\nabla^{\perp}\underline{\phi}\cdot h'\big)\cdot\nabla\underline{f},\\
\partial_th'&=&&\!\!\!\!\!-\big(\underline{\theta}\partial_vh'+\theta'\partial_v\underline{h}\big)-\frac{\mathbb{P}_0f'+h'}{t},\\
\partial_t\theta'&=&&\!\!\!\!\!-\frac{2\theta'}{t}-\partial_v(\theta'\underline{\theta})+\frac{1}{t}(\underline{f}\partial_z\phi'+f'\partial_z\underline{\phi}),\end{aligned}
\right.
\end{equation}
where the relevant quantities are defined as
\begin{equation}\label{eulersupl}
\left\{
\begin{aligned}
\phi'&=&&\!\!\!\mathbb{P}_{\neq 0}\underline{\Delta_{t}}^{-1}f'-\underline{\Delta_{t}}^{-1}\Delta_t'\underline{\phi},\\
\Delta_t'&=&&\!\!\!2(\underline{h}+1)h'(\partial_v-t\partial_z)^2+(h'\partial_v\underline{h}+(\underline{h}+1)\partial_vh')(\partial_v-t\partial_z).
\end{aligned}
\right.
\end{equation}
We will denote the right hand side of \eqref{eulernewl} by $\mathcal{L}g'$, where \begin{equation}\label{matrix}\mathcal{L}=(\mathcal{L}_1,\mathcal{L}_2,\mathcal{L}_3),\quad \mathcal{L}_ig=\sum_{j=1}^3\mathcal{L}_{ij}g_j\end{equation} represents $3\times 3$ matrix valued linear operator. Sometimes (like in the Appendix \ref{appproof}) we will also regard $(g_2',g_3')=(h',\theta')$ as functions of $(t,z,v)$ that does not depend on $z$; see Remark \ref{simnota}.
\subsection{Properties of the operator $\underline{\Delta_t}^{-1}$}
 Properties of the operator $\underline{\Delta_t}^{-1}$ are crucial in the analysis of \eqref{eulernewl}$\sim$\eqref{eulersupl}. We summarize them in the following \begin{proposition}\label{linstep0}
 For any function $F=F(z,v)$, we have
\begin{equation}\label{kernel1}
\widehat{\underline{\Delta_t}^{-1}F}(k,\xi)=\frac{-1}{(\xi-tk)^2+k^2}\bigg(\widehat{f}(k,\xi)+\int_{\mathbb{R}}M(t,k,\xi,\eta)\widehat{F}(k,\eta)\,\mathrm{d}\eta\bigg)
\end{equation}
 for $k\neq 0$, where $M$ satisfies that
\begin{equation}\label{kernel2}
|M(t,k,\xi,\eta)|\lesssim\varepsilon_0 e^{-(C_0-4)|\xi-\eta|}.
\end{equation}
 Moreover, we also have that
\begin{equation}\label{kernel3}
\mathscr{F}_z\underline{\Delta_t}^{-1}F(k,v)=\int_{\mathbb{R}}e^{itk(v-w)}K(t,k,v,w)\cdot \mathscr{F}_zF(k,w)\,\mathrm{d}w
\end{equation}
 for $k\neq 0$, where $c$ is a constant and $K$ satisfies that
 \begin{equation}\label{kernel4}
 |K(t,k,v,w)|\lesssim |k|^{-1}e^{-|k||v-w|/2},\quad |\partial_{v,w}K(t,k,v,w)|\lesssim e^{-|k||v-w|/2},\end{equation}
  and that
  \begin{equation}\label{kernel5}
  \partial_{w}^2 K(t,k,v,w)=
 K'(t,k,v)\delta(v-w)+K''(t,k,v,w),
 \end{equation}where
 \begin{equation}\label{kernel6}
 |K'(t,k,v)|\lesssim1,\quad |K''(t,k,v,w)|\lesssim |k|e^{-|k||v-w|/2}.
 \end{equation}
 \end{proposition}
\subsection{Symbols of the linearized system}
Proposition \ref{linstep0} implies the following estimates for the symbols of the system \eqref{eulernewl}$\sim$\eqref{eulersupl}.
\begin{proposition}\label{systemrough} (1) The system \eqref{eulernewl}$\sim$\eqref{eulersupl} can be written as
\begin{equation}\label{simpsys}\left\{\begin{aligned}\partial_t\widehat{g_1'}(t,k,\xi)&=\widehat{\mathcal{L}_1g'}(t,k,\xi):=\sum_{j=1}^3\sum_{l}\int_{\mathbb{R}}q_{1j}(t,k,l,\xi,\eta)\widehat{g_j'}(t,l,\eta)\,\mathrm{d}\eta,\\ 
\partial_t\widehat{g_2'}(t,k,\xi)&=\widehat{\mathcal{L}_2g'}(t,k,\xi):=\sum_{j=1}^3\sum_{l}\int_{\mathbb{R}}q_{2j}(t,k,l,\xi,\eta)\widehat{g_j'}(t,l,\eta)\,\mathrm{d}\eta-\frac{\widehat{g_2'}(t,k,\xi)+\mathbf{1}_{k=0}\widehat{g_1'}(t,k,\xi)}{t},\\
\partial_t\widehat{g_3'}(t,k,\xi)&=\widehat{\mathcal{L}_3g'}(t,k,\xi):=\sum_{j=1}^3\sum_{l}\int_{\mathbb{R}}q_{3j}(t,k,l,\xi,\eta)\widehat{g_j'}(t,l,\eta)\,\mathrm{d}\eta-\frac{2\widehat{g_3'}(t,k,\xi)}{t},\end{aligned}\right.\end{equation} where, for $t\geq 1$, the symbols satisfy the following bounds:
\allowdisplaybreaks
\begin{align}\label{symbolbd1}|q_{11}|&\lesssim\varepsilon_0 e^{-(C_0/4)(|k-l|+|\xi-\eta|)}\bigg(\frac{|k|+|\xi|+1}{t^2}+\frac{\mathbf{1}_{l\neq 0}}{(\xi-tl)^2+l^2}\big(\mathbf{1}_{l\neq k}\cdot|\xi|+|k|+1\big)\bigg),\\
\label{symbolbd2}|q_{12}|&\lesssim\varepsilon_0^2\mathbf{1}_{l=0}e^{-(C_0/4)(|k|+|\xi-\eta|)}\bigg(\frac{1}{t}+\sum_{q}e^{-(C_0/4)|q|}\frac{\mathbf{1}_{q\neq 0}}{(\xi-qt)^2+q^2}\big(\mathbf{1}_{q\neq k}\cdot|\xi|+1\big)\bigg),\\ 
\label{symbolbd3}|q_{13}|&\lesssim\varepsilon_0\mathbf{1}_{l=0}e^{-(C_0/4)(|k|+|\xi-\eta|)},\quad q_{21}=0,\\
\label{symbolbd4}|q_{22}|&\lesssim\varepsilon_0\mathbf{1}_{k=0}\mathbf{1}_{l=0}e^{-(C_0/4)|\xi-\eta|}\frac{|\xi|+1}{t^2},\quad|q_{23}(t,\xi,\eta)|\lesssim \bigg(\varepsilon_0^2+\frac{\varepsilon_0}{t}\bigg)\mathbf{1}_{k=0}\mathbf{1}_{l=0}e^{-(C_0/4)|\xi-\eta|},\\
\label{symbolbd5}|q_{31}|&\lesssim\varepsilon_0\mathbf{1}_{k=0}\mathbf{1}_{l\neq0}e^{-(C_0/4)(|l|+|\xi-\eta|)}\bigg(\frac{1}{t^3}+\frac{|l|}{t((\xi-lt)^2+l^2)}\bigg),\\
\label{symbolbd6}|q_{32}|&\lesssim\varepsilon_0^2\mathbf{1}_{k=0}\mathbf{1}_{l=0}e^{-(C_0/4)|\xi-\eta|}\sum_{q}\mathbf{1}_{q\neq 0}e^{-(C_0/4)|q|}\bigg(\frac{1}{t^3}+\frac{|q|}{t((\xi-tq)^2+q^2)}\bigg),\\
\label{symbolbd7}|q_{33}|&\lesssim\varepsilon_0\mathbf{1}_{k=0}\mathbf{1}_{l=0}e^{-(C_0/4)|\xi-\eta|}\frac{|\xi|+1}{t^2}.\end{align}

(2) For $(i,j)\in\{(1,1),(1,3)\}$ we also have a decomposition $q_{ij}=q_{ij}'+q_{ij}''$, such that
 \begin{equation}\label{defq11'}
 q_{11}'(t,k,l,\xi,\eta)=\mathbf{1}_{l\neq 0}\mathbf{1}_{|l-k|=1}\frac{\eta(l-k)}{(\eta-tl)^2+l^2}\cdot\frac{\varepsilon_0\widehat{\varphi_b}(\xi-\eta)}{2},
 \end{equation}
\begin{equation}\label{defq11''}
|q_{11}''|\lesssim \varepsilon_0 e^{-(C_0/4)(|k-l|+|\xi-\eta|)}\bigg(\frac{|k|+|\xi|+1}{t^2}+\frac{\varepsilon_0\mathbf{1}_{l\neq 0}}{(\xi-tl)^2+l^2}\big(\mathbf{1}_{l\neq k}\cdot|\xi|+|k|+1\big)\bigg);
\end{equation}
\begin{equation}\label{defq13'}q_{13}'(t,k,l,\xi,\eta)=-\mathbf{1}_{(k,l)=(\pm 1,0)}\frac{\varepsilon_0\widehat{\partial_v\varphi_b}(\xi-\eta)}{2},\quad |q_{13}''|\lesssim\varepsilon_0^2e^{-(C_0/5)(|k-l|+|\xi-\eta|)}.
\end{equation} Moreover, if we decompose $\mathcal{L}_1g$, see the first equation in \eqref{eulernewl}, into
\begin{equation}\label{decompose}\mathcal{L}_1^Tg:=-\underline{\theta}\cdot\partial_v f'-(\underline{h}+1)\nabla^{\perp}\underline{\phi}\cdot \nabla f'\quad\mathrm{and}\quad \mathcal{L}_1^Rg:=-\theta'\partial_v\underline{f}-\big((\underline{h}+1)\nabla^{\perp}\phi'+\nabla^{\perp}\underline{\phi}\cdot h'\big)\cdot\nabla\underline{f},
\end{equation}
and decompose $q_{11}=q_{11}^T+q_{11}^R$ accordingly, then we have better estimates for $q_{11}^R$,
\begin{equation}\label{reactonly}|q_{11}^R|\lesssim\varepsilon_0 e^{-(C_0/4)(|k-l|+|\xi-\eta|)}\cdot\frac{\mathbf{1}_{l\neq 0}}{(\xi-tl)^2+l^2}\big(\mathbf{1}_{l\neq k}\cdot|\xi|+|k|+1\big).
\end{equation}
\end{proposition}
 \subsection{Energy estimates} In Sections \ref{for} and \ref{largeeta}, we will be studying solutions $g'=(f',h',\theta')$ to the inhomogeneous system, see also the formulas in (\ref{simpsys}):
\begin{equation}\label{inhomo}\partial_tg'=\mathcal{L}g'+\rho',\quad \rho'=(\rho_1',\rho_2',\rho_3').\end{equation}
In this subsection we state three auxiliary Fourier-weighted energy estimates for (\ref{inhomo}). We fix a solution $g'$ to (\ref{inhomo}) on some interval $J\subset[1,T_1]$ with source term $\rho'$, and also fix a function $Z(t)$ on $J$, which is decreasing for $T\leq T_0$ and increasing for $T\geq T_0$.

The first estimate allows us to bound the analytic norm of $g'$, as well as ``shifted'' analytic norms, which controls the localization of $g'$ in Fourier space.
\begin{proposition}\label{energy1} Suppose $J=[1,T_1]$, $\tau\in[1/2,2]$, and fix $(k_*,\eta_*)\in\mathbb{Z}\times\mathbb{R}$. Define
\[\lambda_0(t)=1-\frac{|\log t-\log T_0|}{(\log T_0)^2}.
\] Define the weights
\begin{equation}\label{defakwei}A_k(t,\xi)=\frac{1}{(\xi-kt)^2+\varepsilon_0^{-1}k^2},\,k\neq 0,\quad A_0(t,\xi)=\frac{1}{\xi^2+\varepsilon_0^{-1}};\quad A_{*}(t,\xi)=\sum_{k}e^{-2|k|}A_{k}(t,\xi),\end{equation} and the Fourier-weighted energy
\begin{equation}\label{diffineq}M_0(t)=\sum_{k}\int_{\mathbb{R}}A_k(t,\xi)e^{\tau\lambda_0(t)(|k- k_*|+|\xi- \eta_*|)}\big|\widehat{g_1'}(t,k,\xi)\big|^2\,\mathrm{d}\xi+\sum_{j=2}^3\int_{\mathbb{R}}A_*(t,\xi)e^{\tau\lambda_0(t)(|k_*|+|\xi- \eta_*|)}\big|\widehat{g_j'}(t,0,\xi)\big|^2\,\mathrm{d}\xi.\end{equation} If for $t\in J$ we have
\[\sum_{j=1}^3\sum_{k}\int_{\mathbb{R}}e^{\tau\lambda_0(t)(|k-k_*|+|\xi-\eta_*|)}\big|\widehat{\rho_j'}(t,k,\xi)\big|^2\,\mathrm{d}\xi\leq (Z(t))^2,
\]then we have (recall $C$ is any absolute constant depending on $C_0$)\begin{equation}\label{diffineq2}\mathrm{sgn}(t-T_0)\cdot\partial_tM_0(t)\leq C\bigg(\varepsilon_0^{1/2}M_0(t)+\frac{M_0(t)}{t}+\varepsilon_0\frac{|k_*|+|\eta_*|}{t^2}M_0(t)+\sqrt{M_0(t)}\cdot Z(t)\bigg).\end{equation}

\end{proposition}
The second estimate essentially bounds $g'$ in $L^2$, and is more useful if $g'$ becomes less localized in frequency space (i.e. when the time $t$ is small).
\begin{proposition}\label{gevrey2}
Suppose $J=[1,T_1]$, and define \begin{equation}\label{defm1}M_1(t)=\sum_{k}\int_{\mathbb{R}}A_k(t,\xi)\big|\widehat{g_1'}(t,k,\xi)\big|^2\,\mathrm{d}\xi+\sum_{j=2}^3\int_{\mathbb{R}}A_*(t,\xi)\big|\widehat{g_j'}(t,0,\xi)\big|^2\,\mathrm{d}\xi ,\end{equation}where the weights $A_k$ and $B$ are as in (\ref{defakwei}). If for $t\in J$ we have
 \begin{equation}\sum_{j=1}^3\sum_{k}\int_{\mathbb{R}}\big|\widehat{\rho_j'}(t,k,\xi)\big|^2\,\mathrm{d}\xi\leq (Z(t))^2,
 \end{equation} then we have
 \begin{equation}\label{monotinicity2}
 \mathrm{sgn}(t-T_0)\cdot\partial_tM_1(t)\leq C\bigg(\varepsilon_0^{1/2}M_1(t)+\frac{M_1(t)}{t}+\sqrt{M_1(t)}\cdot Z(t)\bigg).
 \end{equation}
\end{proposition}
The third estimate also provides $L^2$ bounds, but is designed for more specific situations and are thus more precise.
\begin{proposition}\label{lastone} Recall the definition of $T_2$ and $T_3$ in (\ref{defineint}). Suppose $J=[T_2,T_1]$, assume that
\[\mathrm{supp}\big(\widehat{\rho'(t)}\big)\subset\big\{(k,\xi):|k|+|\xi|\leq T_0^{21/20}\big\}
\] for all $t\in J$, and that $g'(T_0)=0$. Define
\[
\widetilde{A_k}(t,\xi)=\frac{1}{(\xi-kt)^2+Q^2k^2},\,k\neq 0,\quad \widetilde{A_0}(t,\xi)=\frac{1}{\xi^2+Q^2};\quad \widetilde{A_*}(t,\xi)=\sum_{k}e^{-2|k|}A_{k}(t,\xi),
\] where $Q=\varepsilon_0^{-1}(\log k_0)^{-3}$, and 
\begin{equation}\label{defm4} M_2(t)=\sum_{k}\int_{\mathbb{R}}\widetilde{A_k}(t,\xi)\big|\widehat{g_1'}(t,k,\xi)\big|^2\,\mathrm{d}\xi+\sum_{j=2}^3\int_{\mathbb{R}}\widetilde{A_*}(t,\xi)\big|\widehat{g_j'}(t,0,\xi)\big|^2\,\mathrm{d}\xi.
\end{equation} Suppose $t',t''\in J$; if $t'$ is between $T_0$ and $t''$, and we have
 \begin{equation}\sum_{j=1}^3\sum_{k}\int_{\mathbb{R}}\big|\widehat{\rho_j'}(t,k,\xi)\big|^2\,\mathrm{d}\xi\leq (Z(t))^2\quad \mathrm{and}\quad |\partial_t\log Z(t)|\gg \varepsilon_0(\log k_0)^3.
\end{equation} for all $t$ in the interval between $t'$ and $t''$, then we have 
\begin{equation}\label{improvedgrowth}
M_2(t'')\lesssim  \max\big( e^{C\varepsilon_0(\log k_0)^3|t'-t''|}M_2(t'),(T_0Z(t''))^2\big).
\end{equation}


Moreover, if we assume, for all $t\in[T_3,T_1]$, that 
\begin{equation}\mathrm{supp}\big(\widehat{\rho'(t)}\big)\subset\big\{(k,\xi):|k|+|\xi|\leq \varepsilon_0^{1/30}T_0\big\};
\end{equation}
 \begin{equation}\sum_{j=1}^3\sum_{k}\int_{\mathbb{R}}\big|\widehat{\rho_j'}(t,k,\xi)\big|^2\,\mathrm{d}\xi\leq (Z(t))^2\quad \mathrm{and}\quad |\partial_t\log Z(t)|\gg  T_0^{-1/4},
\end{equation}then for $t\in[T_3,T_1]$ we have
\begin{equation}\label{improve}\sum_{j=1}^3\sum_{k}\int_{\mathbb{R}}|\widehat{g_j'}(t,k,\xi)|^2\,\mathrm{d}\xi\lesssim T_0^{1/2}(Z(t))^2.\end{equation}
\end{proposition}

 \section{Linear analysis I: the homogeneous case}\label{for} 
 In this and the next section we will perform detailed analysis of the system (\ref{inhomo}) on the time interval $[T_2,T_1]$. Two cases of (\ref{inhomo}) will be studied:
 \begin{enumerate}[(a)]
 \item when $\rho'(t)\equiv 0$ and
 \[g'(T_0)=(\varepsilon_1e^{ i(k_0z+\eta_0v)}\varphi_p(k_0\sqrt{\sigma}v)/2,0,0);
 \]
 \item when $g'(T_0)=0$.
 \end{enumerate} Note that, since
  \begin{equation}\label{deccos}\cos(k_0z+\eta_0v)=1/2\big(e^{ i(k_0z+\eta_0v)}+e^{ -i(k_0z+\eta_0v)}\big),
 \end{equation}if $g'$ is the solution in case (a), then $g'+\overline{g'}$ is precisely the solution to the linearization of (\ref{eulernew})$\sim$(\ref{eulersup}) at $\underline{g}$, with data \[((\varepsilon_1\cos(k_0z+\eta_0v)\varphi_p(k_0\sqrt{\sigma}v)/2,0,0))\] at time $T_0$.
 
 In this section we will study case (a); case (b) will be discussed in Section \ref{largeeta}.
 \subsection{Localization in Fourier space} The first step is to localize $g'$ in Fourier space; this will also allow us to essentially get rid of $h'$ and $\theta'$ on $[T_0,T_1]$. Such information can be provided by Proposition \ref{energy1}, but here more precise control is needed.
\begin{proposition}\label{movexi} For $t\in[T_2,T_0]$ we have
\begin{equation}\label{farsmall2}
\sum_{j=1}^3\sum_k\int_{\mathbb{R}}(\mathbf{1}_{|k-k_0|\geq D k_0}+\mathbf{1}_{|\xi-\eta_0|\geq Dk_0})e^{|k-k_0|+|\xi-\eta_0|}|\widehat{g_j'}(t,k,\xi)|^2\,\mathrm{d}\xi\lesssim \varepsilon_1^2e^{-Dk_0/2}.\end{equation}
For $t\in [T_0,T_1]$ we have
\begin{equation}\label{farsmall}
\begin{aligned}
\sum_k\int_{\mathbb{R}}(\mathbf{1}_{|k-k_0|\geq D\sigma k_0}+\mathbf{1}_{|\xi-\eta_0|\geq k_0/10})e^{(|k-k_0|+\sqrt{\sigma}|\xi-\eta_0|)/2}|\widehat{f'}(t,k,\xi)|^2\,\mathrm{d}\xi&\lesssim \varepsilon_1^2e^{-D\sigma k_0/2},\\
\int_{\mathbb{R}}e^{\sqrt{\sigma}|\xi-\eta_0|/2}(|\widehat{h'}(t,\xi)|^2+|\widehat{\theta'}(t,\xi)|^2)\,\mathrm{d}\xi&\lesssim \varepsilon_1^2e^{-k_0/2}.
\end{aligned}
\end{equation}
\end{proposition}
\begin{proof}
Recall that $E$ is an absolute constant such that $D\gg E\gg C$. We consider the time intervals $[T_2,T_0]$ and $[T_0,T_1]$ separately, with different definition of energy functionals but similar proofs.
 
\emph{Case 1: on $[T_2,T_0]$.} In this part we will use implicit constants that are \emph{independent of} $C_0$. We will use $\lesssim$ and $O(1)$ in situations where the constants depend on $C_0$, and $\lesssim'$ and $\widetilde{O}(1)$ where the constants do not depend on $C_0$.

Define the function $\lambda(t)=C_0^{2/3}-C_0^{1/3}|t-T_0|/T_0$, and
\begin{equation}\label{defh30}
    H(t,\xi)=E\bigg(\varepsilon_0+\frac{\varepsilon_0\eta_0}{t^2}\bigg)+C_0^{1/3}\sum_{|k|> k_0/2}H(t,k,\xi),\quad H(t,k,\xi)=\sum_{q\neq 0}e^{-2|k-q|}\frac{\varepsilon_0\eta_0}{(\xi-tq)^2+q^2},
    \end{equation}and define the energy\begin{equation}\label{defmod0}
 M(t)=\sum_{j=1}^3\sum_k\int_{\mathbb{R}}e^{\mu(t,\xi)+\lambda(t)(|k-k_0|+|\xi-\eta_0|)}|\widehat{g_j'}(t,k,\xi)|^2\,\mathrm{d}\xi,\quad \mu(t,\xi)=-\int_{t}^{T_0} H(t',\xi)\,\mathrm{d}t'.
 \end{equation} We will show that $M(t)$ is nondecreasing on $[T_2,T_0]$.

 First, decompose $q_{11}=q_{11}'+q_{11}''$ as in Proposition \ref{systemrough}; by \eqref{symbolbd1}$\sim$\eqref{symbolbd7} and (\ref{defq11''}), and using elementary algebra, we obtain the following estimates:
\begin{equation}\label{qijest201} |q_{11}''|+|q_{ij}|\lesssim e^{-(C_0/5)(|k-l|+|\xi-\eta|)}\cdot\left\{
 \begin{split}
&\bigg(\varepsilon_0+\varepsilon_0\frac{|k-k_0|+|\xi-\eta_0|}{T_0}\bigg),&&\textrm{if $|\xi-\eta_0|\geq |\eta_0|/6$, or $|k|\leq k_0/2$};\\
&\bigg(\varepsilon_0+\frac{\varepsilon_0\eta_0}{t^2}+\varepsilon_0H(t,k,\xi)\bigg),&&\textrm{if $|\xi-\eta_0|<|\eta_0|/6$, and $|k|>|k_0|/2$.}
 \end{split}\right.
\end{equation} For $q_{11}'$, using the definition of $\varphi_b$, we know that
 \begin{equation}\label{qijest202}
|q_{11}'|\lesssim' \mathbf{1}_{|k-l|=1}\cdot C_0 e^{-C_0|\xi-\eta|}\cdot\left\{
\begin{split}
&\frac{|k-k_0|+|\xi-\eta_0|}{T_0},&&\textrm{if $|\xi-\eta_0|\geq |\eta_0|/6$, or $|k|\leq k_0/2$};\\
&H(t,k,\xi),&&\textrm{if $|\xi-\eta_0|<|\eta_0|/6$, and $|k|>k_0/2$.}
\end{split}
\right.
 \end{equation}Note that the implicit constant in this inequality does \emph{not} depend on $C_0$. Now we may use \eqref{simpsys} to compute
 \[\partial_tM(t)=I_1(t)+I_2(t)+\sum_{i,j=1}^3J_{ij}(t),
 \] where 
\begin{align}\label{i1}I_1(t)&=\sum_{j=1}^3\sum_{k}\int_{\mathbb{R}}\bigg(H(t,\xi)+C_0^{1/3}\frac{|k- k_0|+|\xi- \eta_0|}{T_0}\bigg)e^{\mu(t,\xi)+\lambda(t)(|k- k_0|+|\xi- \eta_0|)}\big|\widehat{g_j'}(t,k,\xi)\big|^2\,\mathrm{d}\xi,\\
\label{i2}I_2(t)&=-2t^{-1}\Re\sum_{k}\int_{\mathbb{R}}e^{\mu(t,\xi)+\lambda(t)(|k- k_0|+|\xi- \eta_0|)}\big(\big|\widehat{g_2'}(t,k,\xi)\big|^2+2\big|\widehat{g_3'}(t,k,\xi)\big|^2+\overline{\widehat{g_2'}(t,k,\xi)}\cdot\widehat{g_1'}(t,k,\xi)\big)\,\mathrm{d}\xi,\\
\label{jij}J_{ij}(t)&=2\Re\sum_{k,l}\int_{\mathbb{R}^2}e^{\mu(t,\xi)+\lambda(t)(|k- k_0|+|\xi- \eta_0|)}q_{ij}(t,k,l,\xi,\eta)\cdot\overline{\widehat{g_i'}(t,k,\xi)}\cdot \widehat{g_j'}(t,l,\eta)\,\mathrm{d}\xi\mathrm{d}\eta.\end{align}
On the time interval $[T_0,T_1]$, we clearly have $|I_2(t)|\leq C\varepsilon_0M(t)$ and $I_1(t)\geq E\varepsilon_0M(t)$. Moreover, by (\ref{qijest201}) and (\ref{qijest202}) we have that
 \begin{multline*}
 |J_{ij}(t)|\lesssim\sum_{k,l}\int_{\mathbb{R}^2}e^{\mu(t,\xi)+\lambda(t)(|k- k_0|+\sqrt{\sigma}|\xi- \eta_0|)} e^{-(C_0/5)(|k-l|+|\xi-\eta|)}\times\bigg[O(1)\varepsilon_0+O(1)\frac{\varepsilon_0\eta_0}{t^2}\\
+(\widetilde{O}(1)C_0+O(1)\varepsilon_0)\cdot\bigg(\frac{|k-k_0|+|\xi-\eta_0|}{T_0}+\mathbf{1}_{|\xi-\eta_0|<\eta_0/6}\mathbf{1}_{|k|>k_0/2}H(t,k,\xi)\bigg)\bigg]\big|\widehat{g_i'}(t,k,\xi)\big|\cdot\big|\widehat{g_j'}(t,l,\eta)\big|\,\mathrm{d}\xi\mathrm{d}\eta,
 \end{multline*} where recall that $\widetilde{O}(1)$ represent constants independent of $C_0$, and $O(1)$ represent constants depending on $C_0$. Notice that 
 \[\mathbf{1}_{|\xi-\eta_0|<\eta_0/6}\mathbf{1}_{|k|>k_0/2}H(t,k,\xi)\leq H(t,\xi), \quad H(t,\xi)\lesssim e^{|\xi-\eta|}H(t,\eta);
 \]using Cauchy-Schwartz, the H\"{o}lder estimate $|\mu(t,\xi)-\mu(t,\eta)|\leq\widetilde{O}(1) C_0^{1/3}|\xi-\eta|$ which follows from\begin{multline}\label{estdevmu}|\partial_{\xi}\mu(t,\xi)|\leq \widetilde{O}(1)\frac{\varepsilon_0\eta_0}{t^2} C_0^{1/3}\sum_{|k|> |\eta_*|/(2T_0)}\int_{\mathbb{R}}|\partial_{\xi}H_3(t,k,\xi)|\,\mathrm{d}t\\\leq \widetilde{O}(1) C_0^{1/3}\sum_{|k|> |\eta_*|/(2T_0)}\varepsilon_0\eta_*\sup_{|q|\sim k_0;\xi}\int_{\mathbb{R}}\frac{|\xi-tq|}{((\xi-tq)^2+q^2)^2}\,\mathrm{d}t\leq \widetilde{O}(1) C_0^{1/3},\end{multline} and the fact that
\[\int_{\mathbb{R}}C_0e^{-(C_0/10)|\zeta|}\,\mathrm{d}\zeta\leq \widetilde{O}(1),
\] we obtain that
\[
|J_{ij}(t)|\leq\big(\widetilde{O}(1)C_0^{-1/3}+O(1)\varepsilon_0\big)|I_1(t)|,
\]  which then implies $\partial_tM(t)\geq 0$.

Now that $M(t)$ is nondecreasing on $[T_2,T_0]$, we know for $t\in[T_2,T_0]$ that
\[M(t)\leq M(T_0)=\frac{1}{4}\varepsilon_1^2(k_0\sqrt{\sigma})^{-1}\int_{\mathbb{R}}\widehat{\varphi_p}\bigg(\frac{\xi-\eta_0}{k_0\sqrt{\sigma}}\bigg)^2e^{C_0^{2/3}|\xi-\eta_0|}\,\mathrm{d}\xi\lesssim\varepsilon_1^2 e^{2k_0}.
\] Since $\lambda(t)\geq 2$ on $[T_2,T_0]$ and \[|\mu(t,\xi)|\lesssim E\varepsilon_0T_0+ \sum_{|k|>k_0/2}\int_{\mathbb{R}}H(t,k,\xi)\,\mathrm{d}t+E\int_{T_2}^{T_0}\frac{\varepsilon_0|\eta_*|}{t^2}\,\mathrm{d}t\lesssim Ek_0+\frac{\varepsilon_0\eta_0}{T_2}\lesssim Ek_0,\] we get that
\[\sum_{j=1}^3\sum_k\int_{\mathbb{R}}e^{2(|k-k_0|+|\xi-\eta_0|)}|\widehat{g_j'}(t,k,\xi)|^2\,\mathrm{d}\xi\lesssim \varepsilon_1^2e^{CEk_0},
\] which implies (\ref{farsmall2}).

\emph{Case 2: on $[T_0,T_1]$.} Here all constants are allowed to depend on $C_0$. Define the function $\lambda(t)=1-E(t-T_0)/(\sqrt{\sigma} T_0)$, and (note that they are different from the corresponding functions defined on $[T_0,T_1]$)
\[
\]\begin{equation}\label{defh}H(t,k,\xi)=\sum_{q\neq 0}e^{-2|k-q|}\frac{\varepsilon_0\eta_0}{(\xi-tq)^2+q^2},\quad H(t,\xi)=\varepsilon_0+\sum_{|k-k_0|\leq 2\sigma k_0}H(t,k,\xi),\end{equation} and define the energy
\begin{equation}\label{defenergy}M(t)=\sum_{j=1}^3\sum_{k}\int_{\mathbb{R}}e^{\mu(t,\xi)+\lambda(t)(|k- k_0|+\sqrt{\sigma}|\xi- \eta_0|)}\big|\widehat{g_j'}(t,k,\xi)\big|^2\,\mathrm{d}\xi,\quad \mu(t,\xi)=-E\int_{T_0}^t H(t',\xi)\,\mathrm{d}t',\end{equation}We will show that $M(t)$ is nonincreasing on $[T_0,T_1]$.

 We argue in the same way as in \emph{Case 1}. First we have the following estimates, again by using (\ref{symbolbd1})$\sim$(\ref{symbolbd7}):
 \begin{equation}\label{qijest} |q_{ij}|\lesssim e^{-(C_0/5)(|k-l|+|\xi-\eta|)}\cdot\left\{
 \begin{split}
&\bigg(\varepsilon_0+\frac{|k-k_0|+\sqrt{\sigma}|\xi-\eta_0|}{\sqrt{\sigma} T_0}\bigg),&&\textrm{if $|\xi-\eta_0|\geq \sigma\eta_0$, or $|k-k_0|\geq 2\sigma k_0$};\\
&\big(\varepsilon_0+H(t,k,\xi)\big),&&\textrm{if $|\xi-\eta_0|<\sigma\eta_0$, and $|k-k_0|<2\sigma k_0$.}
 \end{split}\right.
 \end{equation} Now by (\ref{simpsys}) we have\[\partial_tM(t)=I_1(t)+I_2(t)+\sum_{i,j=1}^3J_{ij}(t),\] where 
\begin{align*}I_1(t)&=-E\sum_{j=1}^3\sum_{k}\int_{\mathbb{R}}\bigg(H(t,\xi)+\frac{|k- k_0|+\sqrt{\sigma}|\xi- \eta_0|}{\sqrt{\sigma}T_0}\bigg)e^{\mu(t,\xi)+\lambda(t)(|k- k_0|+\sqrt{\sigma}|\xi- \eta_0|)}\big|\widehat{g_j'}(t,k,\xi)\big|^2\,\mathrm{d}\xi,\\I_2(t)&=-2t^{-1}\Re\sum_{k}\int_{\mathbb{R}}e^{\mu(t,\xi)+\lambda(t)(|k- k_0|+\sqrt{\sigma}|\xi- \eta_0|)}\big(\big|\widehat{g_2'}(t,k,\xi)\big|^2+2\big|\widehat{g_3'}(t,k,\xi)\big|^2+\overline{\widehat{g_2'}(t,k,\xi)}\cdot\widehat{g_1'}(t,k,\xi)\big)\,\mathrm{d}\xi,\\J_{ij}(t)&=2\Re\sum_{k,l}\int_{\mathbb{R}^2}e^{\mu(t,\xi)+\lambda(t)(|k- k_0|+\sqrt{\sigma}|\xi- \eta_0|)}q_{ij}(t,k,l,\xi,\eta)\cdot\overline{\widehat{g_i'}(t,k,\xi)}\cdot \widehat{g_j'}(t,l,\eta)\,\mathrm{d}\xi\mathrm{d}\eta.\end{align*} We then have $|I_2(t)|\lesssim\varepsilon_0M(t)$, $I_1(t)\leq -E\varepsilon_0M(t)$, and $|J_{ij}(t)|\lesssim E^{-1}|I_1(t)|$, by Following the same lines as in \emph{Case 1}, using Cauchy-Schwartz and the estimate
\[|\partial_{\xi}\mu(t,\xi)|\lesssim E\sum_{|k-k_0|\leq 2\sigma k_0}\int_{\mathbb{R}}|\partial_{\xi}H(t,k,\xi)|\,\mathrm{d}t\lesssim E\sigma k_0\varepsilon_0\eta_0\sup_{|q|\sim k_0;\xi}\int_{\mathbb{R}}\frac{|\xi-tq|}{((\xi-tq)^2+q^2)^2}\,\mathrm{d}t\lesssim E\sigma,\] which implies $|\mu(t,\xi)-\mu(t,\eta)|\lesssim E\sigma|\xi-\eta|$.
This proves that $\partial_tM(t)\leq 0$.
 
 Now, in the same way as in \emph{Case 1}, we can compute $M(t)\leq M(T_0)\lesssim\varepsilon_1^2 e^{2\sigma k_0}$ for $t\in[T_0,T_1]$, and use that \[|\mu(t,\xi)|\lesssim E\sum_{|k-k_0|\lesssim 2\sigma k_0}\int_{\mathbb{R}}H(t,k,\xi)\,\mathrm{d}t\lesssim E\sigma k_0\] to conclude that \[\int_{\mathbb{R}}e^{(3/4)\sqrt{\sigma}|\xi-\eta_0|}(|\widehat{h'}(t,\xi)|^2+|\widehat{\theta'}(t,\xi)|^2)\,\mathrm{d}\xi\lesssim \varepsilon_1^2e^{-3k_0/4}e^{CE\sigma k_0},\] for $t\in[T_0,T_1]$, which implies \eqref{farsmall}.
\end{proof}

 \subsection{Approximation by a recurrence relation} With the good localization properties provided by Proposition \ref{movexi}, we can then approximate the system (\ref{inhomo}) by a system of recurrence relations, by approximately solving it on each critical interval $[t_m,t_{m-1}]$ (recall the definitions in (\ref{defineint})). This will be done on $[T_2,T_0]$ and on $[T_0,T_1]$ separately.
 \subsubsection{Approximation on $[T_2,T_0]$} Here, only an upper bound is needed. 
 \begin{proposition}\label{newtoyback} For $k\in\mathbb{Z}$, define $\beta_k(t,v)=\mathscr{F}_zf'(t,k,v)$; for simplicity we also write $\beta_h=h'$ and $\beta_\theta=\theta'$. Define, for $k_0\leq m\leq k_2$, that
  \[F_m=\sup_{k}\|\beta_k(t_m)\|_{L^2};
 \] then for $k_0+1\leq m\leq k_2$, we have
 \begin{equation}\label{int}\sup_{t\in[t_m,t_{m-1}]}\sup_{k}\|\beta_k(t_m)\|_{L^2}\lesssim DF_{m-1}+\varepsilon_1e^{-k_0}.
 \end{equation} Moreover, we have the following approximate recurrence relation:
 \begin{equation}\label{recur0}\beta_k(t_m,v)=\beta_{k}(t_{m-1},v)+\mathcal{R}_{m,k}(v)+\left\{
 \begin{aligned}&\qquad\qquad\quad0,&k&\neq m\pm 1\textnormal{ or }\pm 1;\\
 &\pm\frac{\alpha k_0^2}{m^2}\varphi_b(v)\cdot\beta_m(t_{m-1},v),&k&=m\pm 1;\\
 &+\frac{\alpha k_0^2}{\pi m^2}\partial_v\varphi_b(v)\cdot \beta_{\theta}(t_{m-1},v),&k&=\pm 1,
 \end{aligned}
 \right.
 \end{equation} where the error term satisfies that 
 \begin{equation}\label{errest00}\sup_{k}\|\mathcal{R}_{m,k}\|_{L^2}\lesssim\varepsilon_0D F_{m-1}+\varepsilon_1e^{-k_0}.
 \end{equation}
 \end{proposition}
 \begin{proof} Consider the system (\ref{inhomo}), with $\rho'\equiv0$, on $[t_{m},t_{m-1}]$, where $k_0+1\leq m\leq k_2$. Recall the expression (\ref{simpsys}); using (\ref{farsmall2}) and the fact that $q_{ij}$ decays exponentially in $|k-l|$ and $|\xi-\eta|$, see (\ref{symbolbd1})$\sim$(\ref{symbolbd7}), we can freely insert cutoffs on the right hand side of (\ref{simpsys}) that restricts the sum-integral to the region \begin{equation}\label{restriction}\max(|k|,|l|)\leq Dk_0,\quad \max(|\xi-\eta_0|,|\eta-\eta_0|)\leq Dk_0,\end{equation} up to an error $R=R(t,k,\xi)$ that satisfies
 \[\|R(t)\|_{L^2}\lesssim\varepsilon_1e^{-4k_0}.
 \] 
 
  In the region defined by (\ref{restriction}), we can use (\ref{symbolbd2})$\sim$(\ref{symbolbd7}) to control $q_{ij}$ for $(i,j)\not\in \{(1,1),(1,3)\}$; for $(i,j)\in\{(1,1),(1,3)\}$ we decompose $q_{ij}=q_{ij}'+q_{ij}''$ as in the proof of Proposition \ref{movexi}. Note also that
  \[|\eta-\eta_*|\lesssim Dk_0\ll T_2\lesssim t,\quad\textnormal{which implies that}\quad \frac{1}{(\eta-lt)^2+l^2}\lesssim\frac{1}{t^2}\textnormal{ if $l\neq m$,}
\] we conclude that
\begin{equation}\label{newtoy0}
\left\{\begin{aligned}\partial_t\widehat{\beta_k}(t,\xi)&=\widehat{\mathcal{R}_k}(t,\xi)-\varepsilon_0\mathbf{1}_{k=\pm 1}\int_{\mathbb{R}}\frac{\widehat{\partial_v\varphi_b}(\xi-\eta)}{2}\widehat{\beta_{\theta}}(t,\eta)\,\mathrm{d}\eta,&k&\neq m\pm 1,\\
\partial_t\widehat{\beta_{k}}(t,\xi)&=\widehat{\mathcal{R}_{k}}(t,\xi)\mp\int_{\mathbb{R}}\frac{\widehat{\varphi_b}(\xi-\eta)}{2}\frac{\varepsilon_0\eta_0}{(\eta-tm)^2+m^2}\widehat{\beta_m}(t,\eta)\,\mathrm{d}\eta,&k&=m\pm 1,\\
\partial_t\widehat{\beta_{k}}(t,\xi)&=\widehat{\mathcal{R}_k}(t,\xi),&k&=h,\theta.
\end{aligned}\right.
\end{equation}
and each of the error terms $\mathcal{R}_k$ satisfies either
\begin{equation}\label{errorest20}
|\widehat{\mathcal{R}_k}(t,\xi)|\lesssim D^2\sum_{l}\int_{\mathbb{R}}\bigg(\varepsilon_0^2+\frac{\varepsilon_0\eta_0}{t^2}+\varepsilon_0^2\frac{\eta_0}{(\eta_0-tl)^2+l^2}\bigg)e^{-(C_0/8)(|k-l|+|\xi-\eta|)}|\widehat{\beta}_l(t,\eta)|\,\mathrm{d}\eta
\end{equation} (with suitable changes when $l\in\{h,\theta\}$), or
\begin{equation}\label{errorest2+0}
\|\widehat{\mathcal{R}_k}(t,\xi)\|_{L_{\xi}^2}\lesssim\varepsilon_1e^{-4k_0}.
\end{equation} Here we have also used the fact that 
\[\frac{1}{(\eta-tl)^2+l^2}\lesssim D^2 \frac{1}{(\eta_0-tl)^2+l^2},
\] which also follows from the inequality $|\eta-\eta_0|\lesssim Dk_0$.

Recall that at time $t_{m-1}$ we have 
\[\sup_{k}\|\widehat{\beta_k}(t_{m-1},\xi)\|_{L_{\xi}^2}=F_{m-1};
\]we shall use a bootstrap argument to prove, for $t\in[t_m,t_{m-1}]$, that
\begin{equation}\label{boot020}
\sup_k\bigg\|\sup_{ t\leq t'\leq t_{m-1}}|\widehat{\beta}_k(t',\xi)|\bigg\|_{L_{\xi}^2}\leq
\left\{\begin{split}
&2(F_{m-1}+\varepsilon_1e^{-2k_0}), &k&\in\{m,\theta\};\\
&D(F_{m-1}+\varepsilon_1e^{-2k_0}), &k&\not\in\{m,\theta\}.
\end{split}\right.
\end{equation} In fact, suppose \eqref{boot020} is true for some $t$, then by \eqref{errorest20} and \eqref{errorest2+0} we have either
\begin{multline*}
\sup_k\bigg\|\int_{t}^{t_{m-1}}|\widehat{\mathcal{R}_k}(t',\xi)|\,\mathrm{d}t'\bigg\|_{L_{\xi}^2}\lesssim\sup_k \sum_{l}e^{-|k-l|}\int_{t}^{t_{m-1}}\bigg(\varepsilon_0^2+\frac{\varepsilon_0\eta_0}{{t'}^2}+\varepsilon_0^{2}\frac{\eta_0}{(\eta_0-t'l)^2+l^2}\bigg)\\\times\bigg\|\sup_{t\leq t'\leq t_{m-1}}|\widehat{\beta_l}(t',\eta)|\bigg\|_{L_{\eta}^2}\,\mathrm{d}t'
\lesssim\varepsilon_0D(F_{m-1}+\varepsilon_1e^{-2k_0}),
\end{multline*} 
or
\[\sup_k\bigg\|\int_{t}^{t_{m-1}}|\widehat{\mathcal{R}_k}(t',\xi)|\,\mathrm{d}t'\bigg\|_{L_{\xi}^2}\lesssim \varepsilon_1 e^{-3k_0};
\]
moreover we have
\[\bigg\|\int_{t}^{t_{m-1}}\int_{\mathbb{R}}\bigg|\varepsilon_0\frac{\widehat{\partial_v\varphi_b}(\xi-\eta)}{2}\widehat{\beta_{\theta}}(t',\eta)\bigg|\,\mathrm{d}\eta\mathrm{d}t'\bigg\|_{L_\xi^2}\lesssim\frac{\varepsilon_0\eta_0}{m^2} \bigg\|\sup_{t\leq t'\leq t_{m-1}}|\widehat{\beta_{\theta}}(t',\eta)|\bigg\|_{L_{\eta}^2}\lesssim F_{m-1}+\varepsilon_1e^{-2k_0};
\]
\[\bigg\|\int_{t}^{t_{m-1}}\int_{\mathbb{R}}\bigg|\frac{\widehat{\varphi_b}(\xi-\eta)}{2}\frac{\varepsilon_0\eta_0}{(\eta-t'm)^2+m^2}\widehat{\beta_m}(t',\eta)\bigg|\,\mathrm{d}\eta\mathrm{d}t'\bigg\|_{L_\xi^2}\lesssim\frac{\varepsilon_0\eta_0}{m^2} \bigg\|\sup_{t'\leq t\leq t_{m-1}}|\widehat{\beta_{m}}(t',\eta)|\bigg\|_{L_{\eta}^2}\lesssim F_{m-1}+\varepsilon_1e^{-2k_0}.
\] 
plugging into \eqref{newtoy0}, we get that
\begin{equation}\label{boot030}
\sup_k\bigg\|\sup_{t\leq t'\leq t_{m-1}}|\widehat{\beta_k}(t',\xi)|\bigg\|_{L_{\xi}^2}\leq
\left\{\begin{split}
&F_{m-1}+\varepsilon_1e^{-2k_0}+O(\varepsilon_0)D(F_{m-1}+\varepsilon_1e^{-2k_0}), &k&\in\{m,\theta\};\\
&C(F_{m-1}+\varepsilon_1e^{-2k_0})+O(\varepsilon_0)D(F_{m-1}+\varepsilon_1e^{-2k_0}),&k&\not\in\{m,\theta\},
\end{split}\right.
\end{equation} which proves \eqref{boot020} and hence \eqref{int}.

Next, by \eqref{boot020} and \eqref{newtoy0} we can see that
\[\bigg\|\sup_{t\leq t'\leq t_{m-1}}|\widehat{\beta_k}(t,\xi)-\widehat{\beta_k}(t_m,\xi)|\bigg\|_{L_{\xi}^2}\lesssim \varepsilon_0D F_{m-1} +\varepsilon_1e^{-2k_0}
\] for $k\not\in \{m\pm 1,\pm 1\}$ (in particular for $k\in\{m,\theta\}$); for $k=\pm 1$, using \eqref{newtoy} again we get that
\[\widehat{\beta_k}(t_m,\xi)=\widehat{\beta_k}(t_{m-1},\xi)+\varepsilon_0(t_{m-1}-t_m)\int_{\mathbb{R}}\frac{\widehat{\partial_v\varphi_b}(\xi-\eta)}{2}\widehat{\beta_{\theta}}(t,\eta)\,\mathrm{d}\eta+\widehat{\mathcal{R}_{m,k}}(t,\xi),
\] with $\|\mathcal{R}_{m,k}\|_{L^2}\lesssim \varepsilon_0DF_{m-1}+\varepsilon_1e^{-2k_0}$, which implies (\ref{recur0}) since
\[t_{m-1}-t_m=\frac{\varepsilon_0\eta_0}{m^2}+O(m^{-1})=\frac{2\alpha k_0^2}{\pi m^2}+O(m^{-1})
\] in view of (\ref{defpara}).
For $k=m\pm 1$ we argue similarly, replacing in the formula
\[\widehat{\beta_k}(t_{m},\xi)=\widehat{\mathcal{R}_k}(t_m,\xi)+\widehat{\beta_k}(t_{m-1},\xi)\pm\int_{t_m}^{t_{m-1}}\int_{\mathbb{R}}\frac{\widehat{\varphi_b}(\xi-\eta)}{2}\frac{\varepsilon_0\eta_0}{(\eta-tm)^2+m^2}\widehat{\beta_m}(t,\eta)\,\mathrm{d}\eta\mathrm{d}t
\] the function $\widehat{\beta_m}(t,\eta)$ by $\widehat{\beta_m}(t_{m-1},\eta)$ and the interval $[t_m,t_{m-1}]$ of time integration by $\mathbb{R}$ (notice the estimate $|\xi-\eta_0|\lesssim Dk_0\ll t$ and the definition of $[t_m,t_{m-1}]$), we get that
\[
\widehat{\beta_k}(t_{m},\xi)=\widehat{\beta_k}(t_{m-1},\xi)\pm\frac{\pi\varepsilon_0\eta_*}{2m^2}\mathscr{F}_v(\varphi_b\cdot\beta_m)(t_{m-1},\xi)+\widehat{\mathcal{R}_{m,k}}(\xi),
\] where $\|\widehat{\mathcal{R}_{m,k}}(\xi)\|_{L_{\xi}^2}\lesssim \varepsilon_0D F_{m-1} +\varepsilon_1e^{-2k_0}$. This completes the proof.
\end{proof}

 \subsubsection{Approximation on $[T_0,T_1]$}
 \begin{proposition}\label{errorsmallprop} Define $f''$ by\begin{equation}f''(T_0)=f'(T_0),\quad \partial_tf''=\mathcal{L}_*f'',\end{equation}\begin{equation}\label{deflstar}\widehat{\mathcal{L}_*F}(t,k,\xi)=\sum_{l}\mathbf{1}_{|k-k_0|\leq D\sigma k_0}\mathbf{1}_{|l-k_0|\leq D\sigma k_0}\int_{\mathbb{R}}q_{11}(t,k,l,\xi,\eta)\psi\bigg(\frac{\xi-\eta_0}{k_0}\bigg)\psi\bigg(\frac{\eta-\eta_0}{k_0}\bigg)\widehat{F}(t,l,\eta)\,\mathrm{d}\eta,\end{equation} where $\psi(\zeta)\in\mathcal{E}$ is a cutoff function that equals $1$ for $|\zeta|\leq 1/6$ and equals $0$ for $|\zeta|\geq 1/4$. Then, for $t\in[T_0,T_1]$, we have\begin{equation}\label{errorsmall0}\sum_{k}\int_{\mathbb{R}}e^{(|k-k_0|+\sqrt{\sigma}|\xi-\eta_0|)/4}|\big(\widehat{f'}-\widehat{f''}\big)(t,k,\xi)|^2\,\mathrm{d}\xi\lesssim\varepsilon_1^2 e^{-D\sigma k_0/4}.\end{equation}
 
 Moreover, let $\beta_k(t,v)=\mathscr{F}_zf''(t,k,v)$, and define for $k_0\geq m\geq k_1$ that
 \[F_m=\sup_{k}\|\beta_k(t_m)\|_{L^2},
 \] as in Proposition \ref{newtoyback}. Then, for $k_0\geq m\geq k_1+1$ we have
  \begin{equation}\label{int2}\sup_{t\in[t_m,t_{m-1}]}\sup_{k}\|\beta_k(t_m)\|_{L^2}\lesssim DF_m+\varepsilon_1e^{-k_0}.
 \end{equation} Moreover, we have the following approximate recurrence relation:
 \begin{equation}\label{recur02}\beta_k(t_{m-1},v)=\beta_{k}(t_m,v)+\mathcal{R}_{m,k}(v)+\left\{
 \begin{aligned}&\qquad\qquad\quad0,&k&\neq m\pm 1;\\
 &\mp\frac{\alpha k_0^2}{m^2}\varphi_b(v)\cdot\beta_m(t_{m},v),&k&=m\pm 1.
 \end{aligned}
 \right.
 \end{equation} where the error term satisfies that 
 \[\sup_{k}\|\mathcal{R}_{m,k}\|_{L^2}\lesssim\varepsilon_0 DF_m+\varepsilon_1e^{-k_0}.
 \]

 \end{proposition}
 \begin{remark} The truncated solution $f''$ plays the role of the cutoff (\ref{restriction}) in the proof of Proposition \ref{newtoyback}. It is done in this way (compared to directly estimating $f'$ as in proposition \ref{newtoyback}) here, because later we will need this $f''$ in the space localization estimate (Proposition \ref{phys}),  we do not know how to prove for $f'$.
 \end{remark}
 \begin{proof} We first prove (\ref{errorsmall0}). Write $\partial_tf'=\mathcal{L}_*f'+\mathcal{R},$ where \begin{equation*}\begin{split}\widehat{\mathcal{R}}(t,k,\xi)&=\sum_{l}\int_{\mathbb{R}}\big(q_{12}(t,k,l,\xi,\eta)\widehat{h'}(t,l,\eta)+q_{13}(t,k,l,\xi,\eta)\widehat{\theta'}(t,l,\eta)\big)\,\mathrm{d}\eta\\
 &+\sum_{l}\int_{\mathbb{R}}\bigg[1-\mathbf{1}_{|k-k_0|\leq D\sigma k_0}\mathbf{1}_{|l-k_0|\leq D\sigma k_0}\psi\bigg(\frac{\xi-\eta_0}{k_0}\bigg)\psi\bigg(\frac{\eta-\eta_0}{k_0}\bigg)\bigg]q_{11}(t,k,l,\xi,\eta)\widehat{f'}(t,l,\eta)\,\mathrm{d}\eta.\end{split}\end{equation*} By \eqref{farsmall} we have that
 \begin{equation}\label{bdofr}\sum_{k}\int_{\mathbb{R}}e^{(|k-k_0|+\sqrt{\sigma}|\xi-\eta_0|)/2}|\widehat{\mathcal{R}}(t,k,\xi)|^2\,\mathrm{d}\xi\lesssim\varepsilon_1^2 e^{-D\sigma k_0/4}\end{equation} for $t\in[T_0,T_1]$. Let $f'-f''=\widetilde{f}$, then we have $\widetilde{f}(T_0)=0$ and $\partial_t\widetilde{f}=\mathcal{L'}\widetilde{f}+\mathcal{R}$. If we define\[\widetilde{M}(t)=\sum_{k}\int_{\mathbb{R}}e^{\mu(t,\xi)+\lambda(t)(|k- k_0|+\sqrt{\sigma}|\xi- \eta_0|)/2}\big|\mathscr{F}\widetilde{f}(t,k,\xi)\big|^2\,\mathrm{d}\xi,\] where the functions $\lambda$ and $\mu$ are the same as in Proposition \ref{movexi}, then by repeating the proof of Proposition \ref{movexi}, and using also \eqref{bdofr} we deduce that\[\partial_t\widetilde{M}(t)\leq C\sqrt{\widetilde{M}(t)}\cdot \varepsilon_1 e^{-D\sigma k_0/6}.\] Since $\widetilde{M}(T_0)=0$ we obtain for $t\in[T_0,T_1]$ that $\widetilde{M}(t)\lesssim   \varepsilon_1  e^{-D\sigma k_0/3}$. Using the fact that $\lambda(t)\geq 1/2$ and $|\mu(t,\xi)|\ll D\sigma k_0$, we obtain \eqref{errorsmall0}.
 
 The proof of (\ref{int2}) and (\ref{recur02}) is done in the same way as in the proof of Proposition \ref{newtoyback}; instead of going backwards from $t_{m-1}$ to $t_m$, we go from $t_m$ to $t_{m-1}$. The only differences are the absence of $\beta_h$ and $\beta_{\theta}$, which makes the proof easier, and the presence of the cutoff factors $\psi$ in (\ref{deflstar}), which is in any case harmless.
 \end{proof}
\subsection{Growth of the recurrence relation} Now we proceed to analyze the systems (\ref{recur0}) and (\ref{recur02}). We will obtain an \emph{upper} bound for solutions to (\ref{recur0}) on $[T_2,T_0]$, and \emph{upper and lower} bounds for solutions to (\ref{recur02}) on $[T_0,T_1]$.
\subsubsection{Upper bound of growth on $[T_2,T_0]$} We start with the upper bound of growth on $[T_2,T_0]$, which is a direct consequence of Proposition \ref{newtoyback}. 
\begin{proposition}\label{interval30} Recall the definition of $\beta_k$ in Proposition \ref{newtoyback}. For each $m\in[k_0+1,k_2]$ and $t\in[t_{m},t_{m-1}]$, we have
\begin{equation}\label{nextest10}
\sup_k\|\beta_k(t)\|_{L^2}\lesssim\varepsilon_1\cdot\prod_{j=k_0+1}^m \nu(\alpha_j),
\end{equation} where $\alpha_m=(\alpha k_0^2)/(m^2)$, and $\nu$ is a fixed increasing function on $[0,2]$ satisfying
\begin{equation}\label{defnu0}
\nu(\beta)=\left\{
\begin{split}&(1+\varepsilon_0^{2/5})\max(1,\beta),&\beta\geq 2\varepsilon_0^{3/5};\\
&1+4\varepsilon_0^{3/5},&\beta\leq \varepsilon_0^{3/5},
\end{split}
\right.
\quad \mathrm{and}\quad \partial_{\beta}\nu(\beta)\leq \varepsilon_0^{-1}.
\end{equation}
\end{proposition}
\begin{proof} By \eqref{int}, we only need to consider $t=t_m$. Let $B_{k,m}=\|\beta_k(t_m)\|_{L^2}$ (where $k$ could be $h$ or $\theta$), recall that $F_m=\sup_k B_{k,m}$, and define $E_m=B_{m+1,m}$. Without loss of generality we will also assume $F_m\geq \varepsilon_1$, so the term $\varepsilon_1e^{-k_0}$ can be dropped from (\ref{errest00}). We then have
\begin{equation} \label{ineq0}
\left\{\begin{split}
B_{k,m}&\leq B_{k,m-1}+D\varepsilon_0F_{m-1},&k&\not\in\{m\pm 1,\pm1\};\\
B_{k,m}&\leq B_{k,m-1}+\alpha_mB_{m,m-1}+D\varepsilon_0F_{m-1},&k&=m\pm 1;\\
B_{k,m}&\leq B_{k,m-1}+\alpha_mB_{\theta,m-1}+D\varepsilon_0F_{m-1},&k&=\pm1,
\end{split}\right.
\end{equation} by using (\ref{recur0}), (\ref{errest00}) and the fact that $\|\varphi_b\|_{L^\infty}\leq 1$. By iteration of \eqref{ineq0}, we easily see that
\[B_{k,m}\leq D\varepsilon_0(F_{m-1}+F_{m-2}+\cdots+F_{k_0})+F_{k_0}:=L_m,\quad \textnormal{if $k\geq m+2$ or $k\in\{h,\theta\}$},
\]and moreover that
\[E_{m}\leq\alpha_m E_{m-1}+L_m,\quad F_{m}\leq \max\bigg(F_{m-1},\alpha_mF_{m-1},\alpha_mE_{m-1}+E_{m-2}\bigg)+L_m.
\] Define 
\[G_m=\varepsilon_0^{-1}F_{k_0}\prod_{j=k_0+1}^m \nu(\alpha_j)
\] such that $G_{m}=G_{m-1}\cdot \nu(\alpha_m)$. Since $\varepsilon_0^{-1}F_{k_0}\lesssim\varepsilon_1$, it then suffices to prove that $E_{m}\leq G_{m}$ and $F_m\leq 4G_m$ for all $m$. This is certainly true for $m=k_0$; assume that it holds for all $j\leq m-1$, then, if $\alpha_m\geq 2\varepsilon_0^{3/5}$ we have
\[L_m\leq 4D\varepsilon_0\sum_{j<m}G_j+F_{k_0}\leq \varepsilon_0^{1/2}G_{m-1},
\] hence 
\[E_m\leq \alpha_m G_{m-1}+ \varepsilon_0^{1/2}G_{m-1}\leq G_m
\] and 
\[F_{m}\leq\max\bigg(4G_{m-1},4\alpha_mG_{m-1},\alpha_mG_{m-1}+G_{m-2}\bigg)+ \varepsilon_0^{1/2}G_{m-1}\leq 4G_m;
\]
if $\alpha_m\leq 2\varepsilon_0^{3/5}$ we similarly have that $L_m\leq \varepsilon_0^{1/3}G_{m-1}$, and
\[E_{m}\leq \alpha_m G_{m-1}+\varepsilon_0^{1/3}G_{m-1}\leq G_m,\quad F_{m}\leq (1+\alpha_m+D\varepsilon_0)F_{m-1}\leq 4G_m,
\]where the last inequality follows directly from (\ref{recur0}). This proves \eqref{nextest10} by induction.
\end{proof}
\subsubsection{Upper and lower bounds on $[T_0,T_1]$} To obtain upper and lower bounds for the growth of $f''$ on $[T_0,T_1]$, we need to combine (\ref{recur02}) with an estimate localizing $f''$ in physical space, which is provided in the following proposition.
\begin{proposition}\label{phys} Recall the definition of $f''$ and $\beta_k$ in Proposition \ref{errorsmallprop}. Then for each $t\in[T_0,T_1]$ we have that
\begin{equation}\label{phys3}\sum_{k}\int_{|v|\geq \sigma^{1/3}}|\beta_k(t,v)|^2\,\mathrm{d}v\lesssim \varepsilon_1^2e^{-D\sigma k_0}.\end{equation}
\end{proposition}
\begin{proof}Recall that $\beta_k(t,v)$ is supported in $|k-k_0|\leq D\sigma k_0$, by the definition of $f''$. Define \[M'(t)=\sum_{|k-k_0|\leq D\sigma k_0}\int_{\mathbb{R}}\zeta(k_0\sqrt{\sigma}v)|\beta_k(t,v)|^2\,\mathrm{d}v,\] with the function $\zeta(v)=e^{|v|(\log^+(|v|))^{-3}}$. We will prove that \begin{equation}\label{phys1}\partial_tM'(t)\lesssim K(t)M'(t)\end{equation}for $t\in[T_0,T_1]$, where \[K(t)=\frac{\varepsilon_0\eta_0}{t^2}+\sum_{|k-k_0|\leq D\sigma k_0}\frac{\varepsilon_0\eta_0}{(\eta_0-tk)^2+k_0^2}.\] 

Recall the definition (\ref{deflstar}) of $\mathcal{L}_*$; let the operators $\mathcal{P}$ and $\mathcal{Q}$ be defined by
\[\widehat{\mathcal{P}g}(\xi)=\psi\bigg(\frac{\xi-\eta_0}{k_0}\bigg)\widehat{g}(\xi),\quad \widehat{\mathcal{Q}g}(k,\xi)=\mathbf{1}_{|k-k_0|\leq D\sigma k_0}\psi\bigg(\frac{\xi-\eta_0}{k_0}\bigg)\widehat{g}(k,\xi),\] then we have $\mathcal{L}_*=\mathcal{Q}\mathcal{L}_{11}\mathcal{Q}$, see \eqref{matrix}. This in particular gives
\begin{equation}\label{cutform}\partial_tf''=\mathcal{L}_*f''=-\mathcal{Q}(\underline{\theta}\cdot\partial_v \mathcal{Q}f'')-\mathcal{Q}\big((\underline{h}+1)\nabla^{\perp}\underline{\phi}\cdot\nabla \mathcal{Q}f''\big)-\mathcal{Q}\big((\underline{h}+1)(\partial_v\underline{f}\cdot\partial_z-\partial_z\underline{f}\cdot\partial_v)\underline{\Delta_t}^{-1}\mathcal{Q}f''\big).\end{equation} Let $K(t,k,v,w)$ be defined as in Proposition \ref{linstep0}, and write \begin{equation}\label{defrhoj}\begin{aligned}Y_1(t,k,v)&=t^2\mathscr{F}_z((\underline{h}+1)\partial_v\underline{\phi})(t,k,v),& Y_2(t,k,v)&=-t^2(\mathbf{1}_{k=0}\underline{\theta}(v)+\mathscr{F}_z((\underline{h}+1)\partial_z\underline{\phi})(t,k,v));\\
Y_3(t,k,v)&=-\mathscr{F}_z((\underline{h}+1)\partial_v\underline{f})(t,k,v),&Y_4(t,k,v)&=\mathscr{F}_z((\underline{h}+1)\partial_z\underline{f})(t,k,v),\end{aligned}\end{equation} then we have, by (\ref{useass1})$\sim$(\ref{useass4}), that \begin{equation}\label{physest}\sup_{|\alpha|\leq 2}\|\partial_v^{\alpha}Y_j(t,k,v)\|_{L^{\infty}}\lesssim \varepsilon_0 e^{-(C_0/4)|k|}.\end{equation}Using \eqref{kernel3} we can write that
\begin{equation}\label{weight}\partial_tM'(t)=\sum_{\substack{|k-k_0|\leq D\sigma k_0\\|l-k_0|\leq D\sigma k_0}}\int_{\mathbb{R}^2}\zeta(k_0\sqrt{\sigma}v)\mu(t,k,l,v,w)\overline{\beta_k(t,v)}\beta_l(t,w)\, \mathrm{d}v \mathrm{d}w,
\end{equation} where 
\allowdisplaybreaks
\begin{align*}\mu(t,k,l,v,w)&=\frac{il}{t^2}\int_{\mathbb{R}}e^{i\eta_0(v-r)}k_0(\mathscr{F}^{-1}\psi)(k_0(v-r))Y_1(t,k-l,r)e^{i\eta_0(r-w)}k_0(\mathscr{F}^{-1}\psi)(k_0(r-w))\,\mathrm{d}r\\
&+\frac{1}{t^2}\int_{\mathbb{R}}e^{i\eta_0(v-r)}k_0(\mathscr{F}^{-1}\psi)(k_0(v-r))Y_2(t,k-l,r)\partial_r\big(e^{i\eta_0(r-w)}k_0(\mathscr{F}^{-1}\psi)(k_0(r-w))\big)\,\mathrm{d}r\\
&+il\int_{\mathbb{R}^2}e^{i\eta_0(v-r)}k_0(\mathscr{F}^{-1}\psi)(k_0(v-r))Y_3(t,k-l,r)e^{itl(r-s)}K(t,l,r,s)\\
&\qquad\times e^{i\eta_0(s-w)}k_0(\mathscr{F}^{-1}\psi)(k_0(s-w))\,\mathrm{d}r\mathrm{d}s\\
&+\int_{\mathbb{R}^2}e^{i\eta_0(v-r)}k_0(\mathscr{F}^{-1}\psi)(k_0(v-r))Y_4(t,k-l,r)\partial_r\big(e^{itl(r-s)}K(t,l,r,s)\\
&\qquad\times e^{i\eta_0(s-w)}k_0(\mathscr{F}^{-1}\psi)(k_0(s-w))\big)\,\mathrm{d}r\mathrm{d}s\end{align*} In the first three terms we will estimate the integral directly, using \eqref{kernel4}, \eqref{physest}, and the fact that $\psi\in\mathcal{E}$; in the last term we will integrate by parts in $s$ and use in addition \eqref{kernel5} and \eqref{kernel6}. In either case, we can compute that \[|\mu(t,k,l,v,w)|\lesssim\bigg(\frac{\varepsilon_0\eta_0}{t^2}+\frac{\varepsilon_0\eta_0}{(\eta_0-tl)^2+k_0^2}\bigg)e^{-(C_0/4)|k-l|}k_0\exp\bigg(-\frac{k_0|v-w|}{(\log^+(k_0|v-w|))^2}\bigg).\] Using Cauchy-Schwarz and the fact that \[|\log\zeta(k_0\sqrt{\sigma}v)-\log\zeta(k_0\sqrt{\sigma}w)|\lesssim \frac{k_0|v-w|}{D(\log^+(k_0|v-w|))^2}+O(1),\] which follows from the definition of $\zeta$, we get \eqref{phys1}.

 Now, using \eqref{phys1} and the estimate\[\int_{T_0}^{T_1}K(t)\,\mathrm{d}t\lesssim D\sigma k_0\] which is obvious by the definition of $K(t)$, and the bound\[N(T_0)=\frac{1}{4}\int_{\mathbb{R}}\varepsilon_1^2\zeta(k_0\sqrt{\sigma}v)\varphi_p^2(k_0\sqrt{\sigma}v)\,\mathrm{d}v\lesssim \varepsilon_1^2(k_0\sqrt{\sigma})^{-1}\lesssim \varepsilon_1^2\] which follows from the fact that $\widehat{\varphi_p}\in\mathcal{E}$, we obtain that
 \begin{equation}\label{phys2}\sum_{|k-k_0|\leq D\sigma k_0}\int_{\mathbb{R}}\zeta(k_0\sqrt{\sigma}v)|\mathscr{F}_zf''(t,k,v)|^2\,\mathrm{d}v\lesssim \varepsilon_1^2e^{CD\sigma k_0}\end{equation}
 for $t\in[T_0,T_1]$. Since we have \[\log\zeta(k_0\sqrt{\sigma}v)\geq \frac{k_0\sigma^{5/6}}{2(\log k_0)^3}\gg D\sigma k_0\] when $|v|\geq \sigma^{1/3}$, due to our choice $\sigma=(\log k_0)^{-30}$, \eqref{phys3} is then a simple consequence of \eqref{phys2}.
\end{proof} 
\begin{proposition}\label{grow0} Let $L=\sigma^{-3},\delta=\sigma^6$, and recall the definition of $\beta_k$ and $F_m$ in Proposition \ref{errorsmallprop}. Let $B_m=\|\beta_m(t_m)\|_{L^2}$, then for each $k_0\geq m\geq k_1+1$ we have $B_m\leq F_m\leq LB_m$, and \begin{equation}\label{grow}\bigg(\frac{\alpha k_0^2}{m^2}\bigg)(1+2\sigma^6) \geq \frac{B_{m-1}}{B_m}\geq \bigg(\frac{\alpha k_0^2}{m^2}\bigg)(1-2\sigma^6).\end{equation}
\end{proposition}
\begin{proof} We will induct in $m$ (backwards from $k_0$ to $k_1+1$) to prove the following statement:
\begin{equation}\label{induct}B_m\geq k_0^{-1}\varepsilon_1,\quad \|\beta_k(t_m)\|_{L^2}\leq LB_m,\,\forall|k-k_0|\leq E\sigma k_0;\quad \|\beta_{k}(t_m)\|_{L^2}\leq \delta B_m,\,\forall k<m,\end{equation} and will prove (\ref{grow}) in this process.

When $m=k_0$, \eqref{induct} is true since $\beta_{k_0}(T_0,v)=\varepsilon_1e^{i\eta_0v}\varphi_p(k_0\sqrt{\sigma}v)$ and $\beta_{k}(T_0,v)=0$ for $k\neq k_0$; suppose \eqref{induct} is true for $m$, let us consider $m-1$. By \eqref{recur02} we have \[\sup_{k}\|\beta_k(t_{m-1})\|_{L^2}\leq (L+DL\varepsilon_0+C)B_m,\quad \sup_{k<m-1}\|\beta_k(t_{m-1})\|_{L^2}\leq (\delta+DL\varepsilon_0)B_m,\] and (since $\|\varphi_b\|_{L^\infty}\leq 1$)\begin{equation}\label{growpre}\bigg(\frac{\alpha k_0^2}{m^2}+\delta+DL\varepsilon_0\bigg)B_m\geq B_{m-1}\geq \frac{\alpha k_0^2}{m^2}\|\varphi_b(v)\cdot \beta_m(t_m,v)\|_{L^2}-(\delta+DL\varepsilon_0)B_m.\end{equation} Now, noticing that $\alpha-1=\sigma^2$, and that \[\|\varphi_b(v)\cdot \beta_m(t_m,v)\|_{L^2}\geq (1-\sigma^6)\|\beta_m(t_m,v)\cdot\mathbf{1}_{|v|\leq \sigma^{1/3}}\|_{L^2}\geq (1-\sigma^6)B_m-C\varepsilon_1e^{-D\sigma k_0}\] by the definition of $\varphi_b$ and \eqref{phys3}, we get that \begin{equation}\label{grow0}B_{m-1}\geq \frac{\alpha k_0^2}{m^2}(1-\sigma^6)B_m-(\delta +DL\varepsilon_0)B_m\geq (\alpha-\alpha\sigma^6-\delta-DL\varepsilon_0)B_m,\end{equation} since $k_0\geq m$. By elementary computations one can show 
 \begin{equation*}\left\{\begin{aligned}L+DL\varepsilon_0+C&\leq L(\alpha-\alpha\sigma^6-\delta-DL\varepsilon_0),\\
 \delta+DL\varepsilon_0&\leq \delta(\alpha-\alpha\sigma^6-\delta-DL\varepsilon_0)\end{aligned}\right.\end{equation*} due to our choice of parameters $\alpha=1+\sigma^2$, $L=\sigma^{-3}$, $\delta=\sigma^6$ and $\varepsilon_0=\sigma^{1000}$, so we know that \eqref{induct} is true for $m-1$, completing the inductive proof. By (\ref{growpre}) and \eqref{grow0}, this completes the proof of \eqref{grow} at the same time.
\end{proof}
\subsection{Summary} We summarize some properties of the solution $g'$ in the following proposition. These are immediate consequences of the estimates obtained in this section, so we omit the proof.
\begin{proposition}\label{summarya}
Consider case (a), we have the following estimates.
\begin{enumerate}
\item For $t\in[T_2,T_0]$ we have
\begin{equation}\label{farsmall4}
\sum_{j=1}^3\sum_k\int_{\mathbb{R}}(\mathbf{1}_{|k-k_0|\geq D k_0}+\mathbf{1}_{|\xi-\eta_0|\geq Dk_0})e^{(|k-k_0|+\sqrt{\sigma}|\xi-\eta_0|)/2}|\widehat{g_j'}(t,k,\xi)|^2\,\mathrm{d}\xi\lesssim \varepsilon_1^2e^{-Dk_0/2}.\end{equation}
For $t\in [T_0,T_1]$ we have
\begin{equation}\label{farsmall3}
\begin{aligned}
\sum_k\int_{\mathbb{R}}(\mathbf{1}_{|k-k_0|\geq D\sigma k_0}+\mathbf{1}_{|\xi-\eta_0|\geq k_0/10})e^{(|k-k_0|+\sqrt{\sigma}|\xi-\eta_0|)/2}|\widehat{f'}(t,k,\xi)|^2\,\mathrm{d}\xi&\lesssim \varepsilon_1^2e^{-D\sigma k_0/2},\\
\int_{\mathbb{R}}e^{\sqrt{\sigma}|\xi-\eta_0|/2}(|\widehat{h'}(t,\xi)|^2+|\widehat{\theta'}(t,\xi)|^2)\,\mathrm{d}\xi&\lesssim \varepsilon_1^2e^{-k_0/2}.
\end{aligned}
\end{equation}
\item There exists a sequence $\{B_m\}$, $k_0\geq m\geq k_1$, such that for all $k_0\geq m\geq k_1+1$, we have
\begin{equation}\label{summarya4}
\|g'(t)\|_{L^2}\leq \eta_0B_m,\forall t\in[t_m,t_{m-1}];\quad \|f'(T_1)\|_{L^2}\geq B_{k_1}/2,
\end{equation}
\begin{equation}\label{summarya5}
\varepsilon_1\geq B_{k_0}\geq \eta_0^{-1}\varepsilon_1,\quad \bigg(\frac{\alpha k_0^2}{m^2}\bigg)(1+2\sigma^6) \geq \frac{B_{m-1}}{B_m}\geq \bigg(\frac{\alpha k_0^2}{m^2}\bigg)(1-2\sigma^6).
\end{equation}
\item For all $k_2\geq m\geq k_0+1$ and $t\in[t_m,t_{m-1}]$, we have
\begin{equation}\label{summarya6}
\|g'(t)\|_{L^2}\leq \varepsilon_1\cdot\prod_{j=k_0+1}^m\nu(\alpha_j).
\end{equation}
\end{enumerate}
\end{proposition}

\section{Linear analysis II: the inhomogeneous case}\label{largeeta} In this subsection we study case (b) of the system (\ref{inhomo}), where $g'(T_0)=0$. We will moreover fix a function $Z(t)$, which is decreasing on $[T_2,T_0]$ and increasing on $[T_0,T_1]$, and assume that the forcing term $\rho'=(\rho_1',\rho_2',\rho_3')$ in (\ref{inhomo}) satisfies that
\begin{equation}\label{support}\mathrm{supp}(\widehat{\rho'(t)})\subset\{k_*\}\times[\eta_*-2,\eta_*+2],\quad \|\rho'(t)\|_{L^2}\leq Z(t),
\end{equation} where $(k_*,\eta_*)\in\mathbb{Z}\times\mathbb{R}$ is fixed, and satisfies that
\begin{equation}\label{straight}|k_*-qk_0|+|\eta_*-q\eta_0|\leq Dn\sqrt{\varepsilon_0}T_0,\quad 1\leq q\leq n\leq \varepsilon_0^{-21/40}.
\end{equation}

The arguments in this section are in general similar to those in Section \ref{for}, except for two major differences: (a) the presence of the forcing term $\rho'$, and (b) the fact that the solution $g'$ is localized in Fourier space around $(k_*,\eta_*)$, as opposed to $(k_0,\eta_0)$ in Section \ref{for}. Thus, below we will focus on how these two differences affect the proof, and will omit parts of the proof that are the same as in Section \ref{for}.
\subsection{Localization in Fourier space} As in Section \ref{for}, we start by localizing $g'$ in Fourier space.
\begin{proposition}\label{movexi2}
For each $t\in[T_2,T_1]$ we have
\begin{equation}\label{movexibd2}
\sum_{j=1}^3\sum_k\int_{\mathbb{R}}(\mathbf{1}_{|k-k_*|\geq D\varepsilon_0^{-1}k_0}+\mathbf{1}_{|\xi-\eta_*|\geq D\varepsilon_0^{-1/20}k_0})e^{|k-k_*|+|\xi-\eta_*|}\big|\widehat{g_j'}(t,k,\xi)\big|^2\,\mathrm{d}\xi\lesssim (Z(t))^2 e^{-D\varepsilon_0^{-1/20}k_0/2}.
\end{equation}
\end{proposition}
\begin{proof}The proof is similar to the proof Proposition \ref{movexi}, Case 1. In particular we will use the notation of $\lesssim'$ and $\widetilde{O}(1)$ that denote constants not depending on $C_0$. Fix $D\gg E\gg C$, and define \begin{equation}\label{defmod}
 M(t)=\sum_{j=1}^3\sum_k\int_{\mathbb{R}}e^{\mu(t,\xi)+\lambda(t)(|k-k_*|+|\xi-\eta_*|)}|\widehat{g_j'}(t,k,\xi)|^2\,\mathrm{d}\xi,
 \end{equation}
  where 
  \begin{equation}
  \lambda(t)=C_0^{2/3}-C_0^{1/3}\frac{|t-T_0|}{T_0},
  \end{equation} and 
  \begin{equation}
  \mu(t,\xi)=-\mathrm{sgn}(t-T_0)\int_{T_0}^t H(t',\xi)\,\mathrm{d}t',
  \end{equation} 
    \begin{equation}\label{defh3}
    H(t,\xi)=E\bigg(\varepsilon_0+\frac{\varepsilon_0\eta_*}{t^2}\bigg)+C_0^{1/3}\sum_{|k|> \eta_*/(2T_0)}H(t,k,\xi),\quad H(t,k,\xi)=\sum_{q\neq 0}e^{-2|k-q|}\frac{\varepsilon_0\eta_*}{(\xi-tq)^2+q^2}.
    \end{equation}
Note that when $(k_*,\eta_*)=(k_0,\eta_0)$, this coincides with the definitions (\ref{defh30}) and (\ref{defmod0}). We will prove, for each $t\in[T_2,T_1]$, that
    \begin{equation}\label{diffineqnew0}
    \mathrm{sgn}(t-T_0)\cdot\partial_tM_3(t)\leq C\sqrt{M_3(t)}\cdot Z(t).
    \end{equation}

To prove (\ref{diffineqnew0}), first notice that, by (\ref{symbolbd1})$\sim$(\ref{symbolbd7}) and (\ref{defq11''}), the analog of (\ref{qijest201}) and (\ref{qijest202}) are still true here:
\begin{equation}\label{qijest301} |q_{11}''|+|q_{ij}|\lesssim e^{-(C_0/5)(|k-l|+|\xi-\eta|)}\cdot\left\{
 \begin{split}
&\bigg(\varepsilon_0+\varepsilon_0\frac{|k-k_*|+|\xi-\eta_*|}{T_0}\bigg),&&\textrm{if $|\xi-\eta_*|\geq \eta_*/6$, or $|k|\leq \eta_*/(2T_0)$};\\
&\bigg(\varepsilon_0+\frac{\varepsilon_0\eta_*}{t^2}+\varepsilon_0H(t,k,\xi)\bigg),&&\textrm{if $|\xi-\eta_*|<\eta_*/6$, and $|k|>\eta_*/(2T_0)$,}
 \end{split}\right.
\end{equation} 
 \begin{equation}\label{qijest302}
|q_{11}'|\lesssim' \mathbf{1}_{|k-l|=1}\cdot C_0 e^{-C_0|\xi-\eta|}\cdot\left\{
\begin{split}
&\frac{|k-k_*|+|\xi-\eta_*|}{T_0},&&\textrm{if $|\xi-\eta_*|\geq \eta_*/6$, or $|k|\leq \eta_*/(2T_0)$};\\
&H(t,k,\xi),&&\textrm{if $|\xi-\eta_0|<|\eta_0|/6$, and $|k|>\eta_*/(2T_0)$.}
\end{split}
\right.
 \end{equation} Next we compute the time derivative of $M(t)$,
 \[\partial_tM(t)=I_1(t)+I_2(t)+\sum_{i,j=1}^3J_{ij}(t)+K(t),
 \] where $I_j(t)$ and $J_{ij}(t)$ are the same as in (\ref{i1})$\sim$(\ref{jij}), and the new term $K(t)$ comes from the forcing $\rho'(t)$, which is
 \[K(t)=2\Re\sum_{j=1}^3\sum_k\int_{\mathbb{R}}e^{\mu(t,\xi)+\lambda(t)(|k-k_*|+|\xi-\eta_*|)}\overline{\widehat{g_j'}(t,k,\xi)}\cdot\widehat{\rho_j'}(t,k,\xi)\,\mathrm{d}\xi.
 \]
Notice that we have the same estimate for $\partial_{\xi}\mu(t,\xi)$ as (\ref{estdevmu}),
\begin{multline}|\partial_{\xi}\mu(t,\xi)|\leq \widetilde{O}(1) C_0^{1/3}\sum_{|k|> |\eta_*|/(2T_0)}\int_{\mathbb{R}}|\partial_{\xi}H(t,k,\xi)|\,\mathrm{d}t\\\lesssim \widetilde{O}(1) C_0^{1/3}\sum_{|k|> |\eta_*|/(2T_0)}\varepsilon_0\eta_*\sup_{|q|\sim |k|;\xi}\int_{\mathbb{R}}\frac{|\xi-tq|}{((\xi-tq)^2+q^2)^2}\,\mathrm{d}t\lesssim \widetilde{O}(1) C_0^{1/3},\end{multline}in the same way as the proof of Proposition \ref{movexi}, we obtain that 
 \[|I_2(t)|\leq O(1)\varepsilon_0M_3(t),\quad \mathrm{sgn}(t-T_0)\cdot I_1(t)\leq -D\varepsilon_0M_3(t),\quad |J_{ij}(t)|\leq\big(\widetilde{O}(1)C_0^{-1/3}+O(1)\varepsilon_0\big)|I_1(t)|.\] Finally for $K(t)$, by Cauchy-Schwartz we have In case (b), similarly we have that 
\[|K(t)|\lesssim \sqrt{M(t)}\bigg(\sum_{j=1}^3\sum_k\int_{\mathbb{R}}e^{\mu(t,\xi)+\lambda(t)(|k-k_*|+|\xi-\eta_*|)}|\widehat{\rho_j'}(t,k,\xi)|^2\bigg)^{1/2}\lesssim \sqrt{M(t)}Z(t).
\] (here the implicit constants may depend on $C_0$), using (\ref{support}) and the fact that $\mu(t,\xi)\leq 0$. This proves (\ref{diffineqnew0}).

Now, since $Z(t)$ is decreasing on $[T_2,T_0]$ and increasing on $[T_0,T_1]$, and $M(T_0)=0$, using (\ref{diffineqnew0}) we obtain that $M(t)\lesssim T_0^2(R(t))^2$ for $t\in[T_2,T_1]$. Since for $t\in[T_2,T_1]$ we have $\lambda(t)\geq 2$ and \[|\mu(t,\xi)|\lesssim E\varepsilon_0T_1+ \sum_{|k|> \eta_*/(2T_0)}\int_{\mathbb{R}}H(t,k,\xi)\,\mathrm{d}t+E\int_{T_2}^{T_1}\frac{\varepsilon_0\eta_*}{t^2}\,\mathrm{d}t\lesssim Ek_0+E\frac{\varepsilon_0\eta_*}{T_2}\lesssim E\varepsilon_0^{-1/20}k_0,\] where we have also used (\ref{straight}), we obtain
\[\sum_{j=1}^3\sum_k\int_{\mathbb{R}}e^{2(|k-k_*|+|\xi-\eta_*|)}|\widehat{g_j'}(t,k,\xi)|^2\,\mathrm{d}\xi\lesssim (Z(t))^2e^{CE\varepsilon_0^{-1/20}k_0},
\] which implies (\ref{movexibd2}).
\end{proof}
\subsection{Approximation by a recurrence relation} Next we approximate (\ref{inhomo}) by a recurrnce relation, as we have done in Section \ref{for}.
\begin{proposition}\label{interval2} Define as before $\beta_k(t,v)=\mathscr{F}f'(t,k,v)$, and define $(\beta_h,\beta_\theta)$ similarly. For $0\leq j\leq 2$ and each $m$ define $t_m^*$ and $k_j^*$ by
\[t_m^*=\frac{2\eta_*}{2m+1},\quad t_{k_j^*}^*\leq T_j<t_{k_j^*-1}^*.
\] Then for each $k_1^*+1\leq m\leq k_2^*$ and any $t,t'\in[t_m^*,t_{m-1}^*]$, we have
\begin{equation}\label{equiv}
\sup_k\|\beta_k(t)\|_{L^2}\leq D\big(\sup_k\|\beta_k(t')\|_{L^2}+\varepsilon_0^{-1}\max(Z(t),Z(t'))\big).
\end{equation} Moreover, for each $m\in[k_1^*+1,k_2^*]$ we have
 \begin{equation}\label{recur00}\beta_k(t_{m-1}^*,v)=\beta_{k}(t_{m}^*,v)+\mathcal{R}_{m,k}(v)+\left\{
 \begin{aligned}&\qquad\qquad\quad\,0,&k&\neq m\pm 1\textnormal{ or }\pm 1;\\
 &\mp\frac{\alpha k_0^2\eta_*}{m^2\eta_0}\varphi_b(v)\cdot\beta_m(t_{m},v),&k&=m\pm 1;\\
 &-\frac{\alpha k_0^2\eta_*}{\pi m^2\eta_0}\partial_v\varphi_b(v)\cdot \beta_{\theta}(t_{m},v),&k&=\pm 1,
 \end{aligned}
 \right.
 \end{equation} where the error term satisfies that 
 \begin{equation}\label{errest000}\sup_{k}\|\mathcal{R}_{m,k}\|_{L^2}\lesssim\varepsilon_0^{9/10}D\cdot \min\big(\sup_k\|\beta_k(t_m)\|_{L^2},\sup_k\|\beta_k(t_{m-1})\|_{L^2}\big)+\varepsilon_0^{-1}\max(Z(t_m),Z(t_{m-1})).
 \end{equation}
 \end{proposition}
\begin{proof}The proof is similar to Proposition \ref{newtoyback}. First, by using (\ref{movexibd2}) and the exponential decay of $q_{ij}$, we can freely insert cutoffs on the right hand side of (\ref{simpsys}), restricting to the region
\begin{equation}\label{restriction2}\max(|k|,|l|)\leq D\varepsilon_0^{-1}k_0,\quad \max(|\xi-\eta_0|,|\eta-\eta_0|)\leq D\varepsilon_0^{-1/20}k_0,\end{equation}up to an error term $R=R(t,k,\xi)$ that satisfies
 \[\|R(t)\|_{L^2}\lesssim Z(t)e^{-k_0}.
 \]
 
 As in the proof of Proposition \ref{newtoyback}, we will use (\ref{symbolbd2})$\sim$(\ref{symbolbd7}) to control $q_{ij}$ where $(i,j)\not\in\{(1,1),(1,3)\}$, and make the decomposition $q_{ij}=q_{ij}'+q_{ij}''$ for $(i,j)\in\{(1,1),(1,3)\}$; also notice that \[|\eta-\eta_*|\lesssim \varepsilon_0^{-1/20}k_0\ll T_2\lesssim t,\quad\textnormal{which implies that}\quad\frac{1}{(\eta-lt)^2+l^2}\lesssim\frac{1}{t^2}\textnormal{ if $l\neq m$,}
\] we conclude that (where $\rho_h'$ and $\rho_\theta'$ means $\rho_2'$ and $\rho_3'$ respectively)
\begin{equation}\label{newtoy}
\left\{\begin{aligned}\partial_t\widehat{\beta_k}(t,\xi)&=\widehat{\mathcal{R}_k}(t,\xi)+\varepsilon_0\mathbf{1}_{k=\pm 1}\int_{\mathbb{R}}\frac{\widehat{\partial_v\varphi_b}(\xi-\eta)}{2}\widehat{\beta_{\theta}}(t,\eta)\,\mathrm{d}\eta+\widehat{\rho_1'}(t,k,\xi),&k&\neq m\pm 1,\\
\partial_t\widehat{\beta_{k}}(t,\xi)&=\widehat{\mathcal{R}_{k}}(t,\xi)\mp\int_{\mathbb{R}}\frac{\widehat{\varphi_b}(\xi-\eta)}{2}\frac{\varepsilon_0\eta^*}{(\eta-tm)^2+m^2}\widehat{\beta_m}(t,\eta)\,\mathrm{d}\eta+\widehat{\rho_1'}(t,k,\xi),&k&=m\pm 1,\\
\partial_t\widehat{\beta_{k}}(t,\xi)&=\widehat{\mathcal{R}_k}(t,\xi)+\widehat{\rho_k'}(t,\xi),&k&=h,\theta,
\end{aligned}\right.
\end{equation}
and each of the error terms $\mathcal{R}_k$ satisfies either
\begin{equation}\label{errorest2}
|\widehat{\mathcal{R}_k}(t,\theta)|\lesssim\sum_{l}\int_{\mathbb{R}}\bigg(\varepsilon_0^2+\frac{\varepsilon_0\eta_*}{t^2}+\varepsilon_0^{19/10}\frac{\eta_*}{(\eta_*-tl)^2+l^2}\bigg)e^{-(C_0/8)(|k-l|+|\xi-\eta|)}|\widehat{\beta}_l(t,\eta)|\,\mathrm{d}\eta
\end{equation} (with suitable changes when $l\in\{h,\theta\}$), or
\begin{equation}\label{errorest2+}
\|\widehat{\mathcal{R}_k}(t,\xi)\|_{L_{\xi}^2}\lesssim Z(t)e^{-k_0}.
\end{equation} Here we have also used the fact that 
\begin{equation}\label{losscome}\frac{1}{(\eta-tl)^2+l^2}\lesssim\varepsilon_0^{-1/10} \frac{1}{(\eta_*-tl)^2+l^2},
\end{equation} which follows from the inequality $|\eta-\eta_*|\lesssim \varepsilon_0^{-1/20}k_0$.

We then proceed as in the proof of Proposition \ref{newtoyback}; first we prove the counterpart of (\ref{boot020}), which is
\begin{equation}\label{boot02}
\sup_k\bigg\|\sup_{t_m^*\leq t'\leq t}|\widehat{\beta}_k(t',\xi)|\bigg\|_{L_{\xi}^2}\leq
\left\{\begin{split}
&2(L+Z(t)\varepsilon_0^{-1}), &k&\in\{m,\theta\};\\
&D(L+Z(t)\varepsilon_0^{-1}), &k&\not\in\{m,\theta\},
\end{split}\right.\quad L:=\sup_{k}\|\widehat{\beta_k}(t_m^*,\xi)\|_{L_\xi^2},
\end{equation} for each $t\in[t_{m}^*,t_{m-1}^*]$, and then use this to deduce (\ref{equiv})$\sim$(\ref{recur00}). Note that the base point which is $t_m^*$ in (\ref{boot02}) can be replaced by any point in $[t_{m}^*,t_{m-1}^*]$; this allows us to cover all cases of (\ref{equiv})$\sim$(\ref{recur00}).

The proof of (\ref{boot02}) and subsequently (\ref{equiv})$\sim$(\ref{recur00}) are the same as in Proposition \ref{newtoyback}, with only two differences: (a) all gains of $O(\varepsilon_0)$ are replaced by $O(\varepsilon_0^{9/10})$ due to the power of $\varepsilon_0^{19/10}$ instead of $\varepsilon_0^2$ in (\ref{errorest2}), due to (\ref{losscome}); (b) the error term $\varepsilon_1e^{-k_0}$ is replaced by $Z(t)\varepsilon_0^{-1}$, due to the presence of the forcing term $\rho_j'$ satisfying (\ref{support}). We omit the details.
\end{proof}
\subsection{Growth of the recurrence relation} Finally we establish an upper bound on the growth of the recurrence relation (\ref{recur00}). This follows from the idea of Proposition \ref{interval30}; since the setting is a little different, we will present the full proof here.
\begin{proposition}\label{interval3} Let $\nu(\beta)$ be defined in (\ref{defnu0}). Suppose $Z(t)$ satisfies that
\begin{equation}\label{conditionr}
\left\{\begin{split}
&Z(t_m^*)\geq \nu(\alpha_m^*)\cdot Z(t_{m-1}^*),&&\textnormal{if $t_{m-1}^*\leq T_0$};
\\ &Z(t_{m-1}^*)\geq \nu(\alpha_m^*)\cdot Z(t_{m}^*),&&\textnormal{if $t_{m}^*\geq T_0$},
\end{split}\right.
\end{equation} where $\alpha_m^*=(\alpha k_0^2\eta_*)/(m^2\eta_0)$, then we have for each $m\in[k_1^*+1,k_2^*]$ and $t\in[t_{m-1}^*,t_m^*]$, that
\begin{equation}\label{nextest2}
\sup_k\|\beta_k(t)\|_{L^2}\leq\varepsilon_0^{-2}Z(t).
\end{equation}
\end{proposition}
\begin{proof} By \eqref{equiv}, we only need to consider the case $t=t_m^*$. Let $B_{k,m}=\|\beta_k(t_m^*)\|_{L^2}$, and define, as in the proof of Proposition \ref{interval30}, \[E_m=B_{m+1,m},\quad F_m=\sup_{k}B_{k,m}.\] Let us consider an interval on the left side of $T_0$, i.e., when $t_{m-1}^*\leq T_0$; the other side can be treated similarly. By (\ref{recur00}) we deduce that
\begin{equation} \label{ineq}
\left\{\begin{split}
B_{k,m}&\leq B_{k,m-1}+\varepsilon_0^{9/10}DF_{m-1}+D\varepsilon_0^{-1}Z(t_{m-1}^*),&k&\not\in\{m\pm 1,\pm1\};\\
B_{k,m}&\leq B_{k,m-1}+\alpha_m^*B_{m,m-1}+\varepsilon_0^{9/10}DF_{m-1}+D\varepsilon_0^{-1}Z(t_{m-1}^*),&k&=m\pm 1;\\
B_{k,m}&\leq B_{k,m-1}+(\alpha_m^*/\pi)B_{\theta,m-1}+\varepsilon_0^{9/10}DF_{m-1}+D\varepsilon_0^{-1}Z(t_{m-1}^*),&k&=\pm1.
\end{split}\right.
\end{equation} By iteration of \eqref{ineq}, we see that
\[B_{k,m}\leq \varepsilon_0^{9/10}D\sum_{j=k_0^*}^{m-1}F_j+D\varepsilon_0^{-1}\sum_{j=k_0^*}^{m-1}Z(t_{j}^*):=L_m,\quad \textnormal{if $k\geq m+2$ or $k\in\{h,\theta\}$};
\]notice also that due to (\ref{conditionr}),
\[D\varepsilon_0^{-1}\sum_{j=k_0^*}^{m-1}Z(t_{j}^*)\lesssim\varepsilon_0^{-2}Z(t_{m-1}^*).
\]Plugging into (\ref{ineq}) we obtain
\[E_{m}\leq\alpha_m^* E_{m-1}+L_m,\quad F_{m}\leq \max\bigg(F_{m-1},\alpha_m^*F_{m-1},\alpha_m^*E_{m-1}+E_{m-2}\bigg)+L_m.
\] Define $G_m=\varepsilon_0^{-3}Z(t_m^*)$ such that $G_{m}\geq G_{m-1}\cdot \nu(\alpha_m^*)$, it suffices to prove that $E_{m}\leq G_{m}$ and $F_m\leq 4G_m$ for all $m$. This inequality is trivially true for $m=k_0^*$; assume that it holds for all $j\leq m-1$, then if $\alpha_m\geq 2\varepsilon_0^{3/5}$ we have
\[L_m\leq 4D\varepsilon_0^{9/10}\sum_{j<m}G_j+\varepsilon_0G_{m-1}\leq \varepsilon_0^{4/9}G_{m-1},
\] hence 
\[E_m\leq \alpha_m^* G_{m-1}+ \varepsilon_0^{4/9}G_{m-1}\leq G_m
\] and 
\[F_{m}\leq\max\bigg(4G_{m-1},4\alpha_m^*G_{m-1},\alpha_m^*G_{m-1}+G_{m-2}\bigg)+ \varepsilon_0^{4/9}G_{m-1}\leq 4G_m;
\]
if $\alpha_m\leq 2\varepsilon_0^{3/5}$ we similarly have that $L_m\leq \varepsilon_0^{1/5}G_{m-1}$, and
\[E_{m}\leq \alpha_m^* G_{m-1}+\varepsilon_0^{1/5}G_{m-1}\leq G_m,\quad F_{m}\leq (1+\alpha_m^*+D\varepsilon_0)F_{m-1}+D\varepsilon_0G_{m-1}\leq 4G_m,
\]where the last inequality follows directly from (\ref{recur0}). This completes the proof.
\end{proof}
\section{Nonlinear analysis, and construction of approximate solution}\label{approximate}
Now we go back to the full nonlinear system \eqref{eulernew}$\sim$\eqref{eulersup}, and rewrite it here: 
\begin{equation}\label{eulernewrewrite}
\left\{\begin{aligned}
\partial_tf&=-\theta\cdot\partial_vf-(h+1)\nabla^{\perp}\phi\cdot\nabla f,\\
\partial_th&=-\theta\partial_vh-\frac{\mathbb{P}_0f+h}{t},\\
\partial_t\theta&=-\frac{2}{t}\theta-\theta\partial_v\theta+\frac{1}{t}\mathbb{P}_0(f\cdot\partial_z\phi).
\end{aligned}
\right.
\end{equation} In this section we will construct an approximate solution to (\ref{eulernewrewrite}) by performing the formal Taylor expansion, using the linear estimates proved in Sections \ref{for} and \ref{largeeta}.

Define $g=\underline{g}+g^*$, where recall $g=(f,h,\theta)$ and $\underline{g}=(\underline{f},\underline{h},\underline{\theta})$, then we also have $\phi=\underline{\phi}+\phi^*$, where $\underline{\phi}=\underline{\Delta_t}^{-1}\underline{f}$, and 
\begin{equation}\label{eulersuprewrite}\left\{\begin{split}
\phi^*&=\underline{\Delta_t}^{-1}f^*+\sum_{j=1}^{\infty}(-1)^j(\underline{\Delta_t}^{-1}\Delta_t^*)^j\underline{\phi},\\
\Delta_t^*&=(2(\underline{h}+1)h^*+(h^*)^2)(\partial_v-t\partial_z)^2+(h^*\partial_v(\underline{h}+h^*)+(\underline{h}+1)\partial_vh^*)(\partial_v-t\partial_z).\end{split}\right.\end{equation} Moreover we will decompose $\Delta_t^*=\Delta_t'+\Delta_t^{**}$, where
\begin{equation}\label{defdeltas}
\begin{split}\Delta_t'&=2(\underline{h}+1)h^*(\partial_v-t\partial_z)^2+(h^*\partial_v\underline{h}+(\underline{h}+1)\partial_vh^*)(\partial_v-t\partial_z),\\
\Delta_t^{**}&=(h^*)^2(\partial_v-t\partial_z)^2+\frac{1}{2}\partial_v(h^*)^2(\partial_v-t\partial_z)
\end{split}
\end{equation} are the linear and quadratic terms in $\Delta_t^*$ respectively. We start by describing the system (\ref{eulernewrewrite}) in terms of the perturbation $g^*$.
\subsection{The nonlinearities and the Taylor expansion}\label{taylorexp}
Using \eqref{eulersuprewrite} to substitute $\phi$ on the right hand side of \eqref{eulernewrewrite}, we can rewrite \eqref{eulernewrewrite} as \begin{equation}\label{g-system}\partial_tg^*=\mathcal{L}g^*+\mathcal{N}(g^*)=\mathcal{L}g^*+\sum_{p=2}^{\infty}\mathcal{N}_p(g^*,\cdots,g^*),\end{equation} where $\mathcal{N}$ represents the nonlinearity and $\mathcal{N}_p$ represents terms homogeneous of degree $p$. More precisely, we can compute $\mathcal{N}=(\mathcal{N}_{1},\mathcal{N}_{2},\mathcal{N}_{3})$, where 
\[
\begin{split}
\mathcal{N}_{1}&=-\theta^*\partial_vf^*-h^*\nabla f^*\cdot\nabla^{\perp}\underline{\phi}-(\underline{h}+1)\nabla \underline{f}\cdot\nabla^{\perp}\bigg(\sum_{j=2}^{\infty}(-1)^j(\underline{\Delta_t}^{-1}\Delta_t^*)^j-\underline{\Delta_t}^{-1}\Delta_t^{**}\bigg)\underline{\phi}\\
&\qquad-\big((h^*+\underline{h}+1)\nabla f^*+h^*\nabla\underline{f}\big)\cdot\nabla^{\perp}\bigg(\underline{\Delta_t}^{-1}f^*+\sum_{j=1}^{\infty}(-1)^j(\underline{\Delta_t}^{-1}\Delta_t^*)^j\underline{\phi}\bigg),
\end{split}
\] and
\[
\begin{split}
\mathcal{N}_{2}=-\theta^*\partial_v h^*,\quad \mathcal{N}_3&=-\theta^*\partial_v\theta^*+\frac{1}{t}\mathbb{P}_0\bigg(f^*\cdot\partial_z\bigg(\underline{\Delta_t}^{-1}f^*+\sum_{j=1}^{\infty}(-1)^j(\underline{\Delta_t}^{-1}\Delta_t^*)^j\underline{\phi}\bigg)\bigg)\\
&\qquad +\frac{1}{t}\mathbb{P}_0\bigg(\underline{f}\cdot\partial_z\bigg(\sum_{j=2}^{\infty}(-1)^j(\underline{\Delta_t}^{-1}\Delta_t^*)^j-\underline{\Delta_t}^{-1}\Delta_t^{**}\bigg)\underline{\phi}\bigg).
\end{split}
\]To separate terms of each homogeneity, we define
\begin{equation}\label{defphip}\Phi_p=\Phi_p(h^*,\cdots,h^*)=\sum_{p/2\leq r\leq p}(-1)^r\sum_{(B_i)_{1\leq i\leq r}}\prod_{i=1}^r(\underline{\Delta_t}^{-1}B_i)\underline{\phi}
\end{equation} for $p\geq 0$, where the sum is over all possible choices of $(B_i)_{1\leq i\leq r}$ such that $B_i\in\{\Delta_t',\Delta_t^{**}\}$ and exactly $p-r$ of these $B_i$ are $\Delta_t^{**}$. Then we have $\mathcal{N}_p=(\mathcal{N}_{p,1},\mathcal{N}_{p,2},\mathcal{N}_{p,3})$, where for $p\geq 2$,
\begin{equation}\label{nonlins}
\begin{aligned}
\mathcal{N}_{p,2}&=-\mathbf{1}_{p=2}\cdot\theta^*\partial_vh^*,\\
\mathcal{N}_{p,1}&=-\mathbf{1}_{p=2}\cdot \theta^*\partial_vf^*+\mathbf{1}_{p=3}\cdot h^*\nabla f^*\cdot \nabla^{\perp}\underline{\Delta_t}^{-1}f^*-(\underline{h}+1)\nabla \underline{f}\cdot \nabla^{\perp}\Phi_p\\&\quad\,\;\!\!-((\underline{h}+1)\nabla f^*+h^*\nabla\underline{f})\cdot\nabla^{\perp}\Phi_{p-1}-h^*\nabla f^*\cdot\nabla^{\perp}\Phi_{p-2},\\
\mathcal{N}_{p,3}&=-\mathbf{1}_{p=2}\cdot\theta^*\partial_v\theta^*+\mathbf{1}_{p=2}\cdot\frac{1}{t}\mathbb{P}_0\big(f^*\cdot\partial_z\underline{\Delta_t}^{-1}f^*\big)+\frac{1}{t}\mathbb{P}_0(f^*\cdot\partial_z\Phi_{p-1}+\underline{f}\cdot\partial_z\Phi_p).
\end{aligned}
\end{equation}

Now we can write down the Taylor expansion of $g^*$. Define $g^{(1)}$ such that
\begin{equation}
\partial_tg^{(1)}=\mathcal{L}g^{(1)},\quad g^{(1)}(T_0)=(\varepsilon_1 \cos(k_0z+\eta_0v)\varphi_p(k_0\sqrt{\sigma}v),0,0);\end{equation} for $n\geq 2$, we inductively define $g^{(n)}$ to be
\begin{equation}
\partial_tg^{(n)}=\mathcal{L}g^{(n)}+Z^{(n-1)},\quad g^{(n)}(T_0)=0,
\end{equation}where
\[
Z^{(n-1)}=\sum_{p=2}^n\sum_{n_1+\cdots +n_p=n}\mathcal{N}_p(g^{(n_1)},\cdots g^{(n_p)}).
\] Then for each $n$, if we define 
\begin{equation}\label{defcapgn}
G^{(n)}=g^{(1)}+\cdots +g^{(n)},\quad G^{(n)}(T_0)=(\varepsilon_1 \cos(k_0z+\eta_0v)\varphi_p(k_0\sqrt{\sigma}v),0,0),
\end{equation}
then we have
\begin{equation}\partial_tG^{(n)}=\mathcal{L}G^{(n)}+\mathcal{N}(G^{(n)})-E^{(n)},
\end{equation} where 
\begin{equation}\label{defen}
E^{(n)}=\sum_{p=2}^{\infty}\sum_{\substack{n_1,\cdots n_p\leq n;\\n_1+\cdots +n_p>n}}\mathcal{N}_p(g^{(n_1)},\cdots g^{(n_p)}).
\end{equation} Below we will prove inductively estimates for each $g^{(n)}$, in the end obtaining that $G^{(n)}$ is a sufficiently good approximate solution to \eqref{g-system} for large $n$.
\subsubsection{Multilinear estimates} Before proceeding, we first prove some basic estimates for the nonlinearities $\Phi_p$ and $\mathcal{N}_p$.
\begin{proposition}[Multilinear estimates for $\Phi_p$]\label{multilnearest} Recall the parameters and the norm $\mathcal{G}_\lambda^*$ defined in Section \ref{choicenorm}. For any $p\geq 1$, $\lambda\in[1/4,4]$, and $3\leq s\leq 2N$,  we have the following estimate:
\begin{equation}\label{multest1}\big\|\langle\nabla\rangle\Phi_p(h^1,\cdots,h^p)(t)\big\|_{\mathcal{G}_{\lambda}^*}\leq (Ct^4)^{p}\prod_{j=1}^p\|h^j(t)\|_{\mathcal{G}_{\lambda}^*};
\end{equation}
\begin{equation}\label{multest2}
\big\|\langle\nabla\rangle\Phi_p(h^1,\cdots,h^p)(t)\big\|_{H^s}\leq (Ct^4)^{p}\sum_{j=1}^p\|h^j(t)\|_{H^s}\prod_{i\neq j}\|h^j(t)\|_{H^3}.
\end{equation} Moreover, let $\lambda'=\lambda-\delta$, where $0<\delta<\lambda$; suppose $\eta_1,\cdots,\eta_p\in\mathbb{R}$, and define \[\underline{\eta}=\eta_1+\cdots+\eta_p;\quad \max_{1\leq j\leq p}|\eta_j|=\Lambda,\] then we have the ``shifted'' estimate
\begin{equation}\label{multestshift}\sum_k\int_{\mathbb{R}}e^{2\lambda'(|k|+|\xi-\underline{\eta}|)}|\mathscr{F}\langle\nabla\rangle\Phi_p(h^1,\cdots,h^p)(t,k,\xi)|^2\leq t^4(C\max(\delta^{-2},\Lambda))^{2p}\cdot\prod_{j=1}^p\int_{\mathbb{R}}e^{2\lambda|\xi- {\eta_j}|}|\widehat{h^j}(t,\xi)|^2.
\end{equation}
\end{proposition}
\begin{proof}We first prove (\ref{multest1}). Since there are at most $C^p$ terms in the expression of $\Phi_p$, it suffices to study one of them, say
\[\Phi:=\prod_{i=1}^r(\underline{\Delta_t}^{-1}B_i)\underline{\phi},
\] where each $B_i$ is either $\Delta_t'$ or $\Delta_t^{**}$. Define the norms $\mathcal{G}_{\pm}^*$ by
\[\|f\|_{\mathcal{G}_{\pm}^*}^2=\sum_k\int_{\mathbb{R}}(1+|k|+|\xi|)^{\pm 2}e^{2\kappa(k,\xi)}|\widehat{f}(k,\xi)|^2\,\mathrm{d}\xi,
\] then we have $\|\underline{\Delta_t}^{-1}f\|_{\mathcal{G}_+^*}\lesssim t^2\|f\|_{\mathcal{G}_-^*}$, due to \eqref{kernel1} and the fact that \[\frac{1}{(\xi-tk)^2+k^2}\lesssim\frac{t^2}{1+|\xi|^2+|k|^2}\] when $k\neq 0$ and $t\geq 1$. Moreover we have $\|fg\|_{\mathcal{G}^*}\lesssim\|f\|_{\mathcal{G}^*}\|g\|_{\mathcal{G}^*}$ and the same bounds for $\mathcal{G}_{\pm}^*$; this can be proved by simple kernel estimates and Schur type inequalities. Therefore we obtain that
\[\|\langle\nabla\rangle \Phi\|_{\mathcal{G}^*}\sim\|\Phi\|_{\mathcal{G}_+^*}\lesssim t^2\|B_1\Phi'\|_{\mathcal{G}_-^*}\lesssim t^4\|\langle\nabla\rangle\Phi'\|_{\mathcal{G}^*}\cdot
\left\{
\begin{split}&\|h^1\|_{\mathcal{G}^*},&B_1&=\Delta_t';\\
&\|h^1\|_{\mathcal{G}^*}\|h^2\|_{\mathcal{G}^*},&B_1&=\Delta_t^{**},
\end{split}
\right.
\] where 
\[\Phi'=\prod_{i=2}^r(\underline{\Delta_t}^{-1}B_i)\underline{\phi}.
\] By induction, and the fact that $\|\underline{\phi}\|_{\mathcal{G}_+^*}\lesssim1$, this implies \eqref{multest1}. The proof of \eqref{multest2} is similar; we just use Leibniz rule and write
\[ t^2\|B_1\Phi'\|_{H^{s-1}}\lesssim t^4\cdot
\left\{
\begin{split}&\|\langle\nabla\rangle\Phi'\|_{H^s}\|h^1\|_{H^3}+\|\Phi'\|_{H^3}\|h^1\|_{H^s},&B_1&=\Delta_t';\\
&\|\langle\nabla\rangle\Phi'\|_{H^s}\|h^1\|_{H^3}\|h^2\|_{H^3}+\|\Phi'\|_{H^3}(\|h^1\|_{H^s}\|h^2\|_{H^3}+\|h^1\|_{H^3}\|h^2\|_{H^s}),&B_1&=\Delta_t^{**},
\end{split}
\right.
\] and 
\[\|\langle\nabla\rangle \Phi\|_{H^3}\sim\|\Phi\|_{H^{4}}\lesssim t^2\|B_1\Phi'\|_{H^2}\lesssim t^4\cdot
\left\{
\begin{split}&\|\langle\nabla\rangle\Phi'\|_{H^3}\|h^1\|_{H^3},&B_1&=\Delta_t';\\
&\|\langle\nabla\rangle\Phi'\|_{H^3}\|h^1\|_{H^3}\|h^2\|_{H^3},&B_1&=\Delta_t^{**},
\end{split}
\right.
\] to complete the inductive step.

Finally, to prove (\ref{multestshift}), we will consider the multilinear symbol of $\Phi$ directly. Using the definitions of $\Delta_t'$ and $\Delta_t^{**}$ in \eqref{defdeltas}, the formula \eqref{kernel1} for $\underline{\Delta_t}$, and the estimates \eqref{useass3} and \eqref{useass4} for $\underline{h}$ and $\underline{\theta}$, we can write $\Phi$ as a linear combination of terms of form $\Psi$, where the total coefficient in the linear combination is at most $C^p$, and each single $\Psi$ is such that
\begin{equation}\label{expresspsi}
\mathscr{F}\langle\nabla\rangle\Psi(h^1,\cdots,h^p)(t,k,\xi)=\int_{\mathbb{R}^p}\mathcal{A}(k,\xi_1,\cdots,\xi_p)\prod_{j=1}^p(\mathscr{F}h^j)(t,\xi_j)\,\mathrm{d}\xi_1\cdots\mathrm{d}\xi_p
\end{equation} where the symbol $\mathcal{A}$ satisfies
\begin{equation}\label{explicitbd}|\mathcal{A}|\leq e^{-(C_0/5)(|k|+|\xi-\xi_1-\cdots-\xi_p|)}(1+|\xi|)\prod_{i=1}^r\frac{|\mu_{j_{i-1}}-kt|}{(\mu_{j_i}-kt)^2+k^2}\cdot(1+|\mu_{j_{i-1}}-kt|+|\mu_{j_{i-1}}-\mu_{j_i}|),
\end{equation}
where
\[
\mu_{j}:=\xi-\xi_1-\cdots -\xi_j;\quad p=j_0>j_1>\cdots >j_r=0,\quad j_{i-1}-j_i\in\{1,2\}.
\] Now to prove (\ref{multestshift}) for $\Psi$, we first obtain an estimate for $\mathcal{A}$. Let $\mu_{j_i}-kt=\lambda_i$, we get that
\begin{multline}|\mathcal{A}|\leq e^{-(C_0/5)(|k|+|\lambda_0+kt|)}(1+|\lambda_r|+kt)\prod_{i=1}^r\frac{|\lambda_{i-1}|(1+|\lambda_{i-1}|+|\lambda_i-\lambda_{i-1}|)}{\lambda_i^2+k^2}\\\lesssim Ct^2 e^{-(C_0/6)(|k|+|\lambda_0+kt|)}\prod_{i=1}^r\frac{1+|\lambda_{i-1}|+|\lambda_i-\lambda_{i-1}|}{1+|\lambda_i|}.
\end{multline} Let $\chi_i=\lambda_i-\lambda_{i-1}$, then
\[\max_{1\leq i\leq r}|\chi_i|\leq 2\max_{1\leq j\leq p}|\xi_j|,
\] therefore \[\frac{1+|\lambda_{i-1}|+|\lambda_i-\lambda_{i-1}|}{1+|\lambda_i|}\leq 2(1+|\chi_i|)\leq 4\big(1+\max_{1\leq j\leq p}|\xi_j|\big)\] for each $1\leq i\leq r$. This implies that
\begin{equation}\label{estimatea}|\mathcal{A}|\leq C^pt^2 e^{-(C_0/6)(|k|+|\xi-\xi_1-\cdots -\xi_p|)}\cdot\big(1+\max_{1\leq j\leq p}|\xi_j|\big)^p.
\end{equation} Now by interpolation, (\ref{multestshift}) would follow from the corresponding weighted $L^1$ and $L^\infty$ estimates. The $L^\infty$ estimate follows from the following computations:
\begin{equation*}
\begin{aligned}e^{\lambda'(|k|+|\xi-\underline{\eta}|)}\big|\mathscr{F}\langle \nabla\rangle\Psi(h^1,&\cdots,h^p)(t,k,\xi)\big|\lesssim e^{\lambda'(|k|+|\xi-\underline{\eta}|)}\int_{\mathbb{R}^p}C^pt^2e^{-(C_0/6)(|k|+|\xi-\xi_1-\cdots -\xi_p|)}\\&\times\big(1+\max_{1\leq j\leq p}|\xi_j|\big)^p\prod_{j=1}^p|\widehat{h^j}(t,\xi)|\,\mathrm{d}\xi_1\cdots\mathrm{d}\xi_p\\
&\lesssim \int_{\mathbb{R}^p}e^{\lambda'(|k|+|\xi-\underline{\eta}|)}e^{-(C_0/6)(|k|+|\xi-\xi_1-\cdots -\xi_p|)}\cdot C^pt^2\big(1+\max_{1\leq j\leq p}|\xi_j|\big)^p\\&\times\prod_{j=1}^pe^{-\lambda|\xi_j-\eta_j|}\,\mathrm{d}\xi_1\cdots\mathrm{d}\xi_p\cdot\prod_{j=1}^p\big\|e^{\lambda|\xi_j-\eta_j|}\widehat{h^j}(t,\xi_j)\big\|_{L^\infty}\\
&\lesssim\prod_{j=1}^p\big\|e^{\lambda|\xi_j-\eta_j|}\widehat{h^j}(t,\xi_j)\big\|_{L^\infty}\cdot C^pt^2\int_{\mathbb{R}^p}\prod_{j=1}^pe^{-\delta|\xi_j-\eta_j|}\cdot\big(1+\max_{1\leq j\leq p}|\xi_j|\big)^p\,\mathrm{d}\xi_1\cdots\mathrm{d}\xi_p\\
&\lesssim \prod_{j=1}^pe^{-\delta|\xi_j-\eta_j|}\cdot C^pt^2\bigg(\delta^{-2p}+\max_{1\leq j\leq p}|\eta_j|^p\bigg) \cdot \prod_{j=1}^p\big\|e^{\lambda|\xi_j-\eta_j|}\widehat{h^j}(t,\xi_j)\big\|_{L^\infty}\end{aligned}
\end{equation*} and the $L^1$ estimate is proved similarly.
\end{proof}
\begin{proposition}[Multilinear estimates for $\mathcal{N}_p$]\label{multilinearest2} For $1\leq i\leq 3$ and $3\le1 s\leq 2N$ we have the estimates
\begin{equation}\label{multest3}\big\|\mathcal{N}_{p,i}(g^1,\cdots,g^p)(t)\big\|_{\mathcal{G}_2^*}\leq (Ct^4)^p\prod_{j=1}^p\|g^j(t)\|_{\mathcal{G}_4^*};
\end{equation}
\begin{equation}\label{bdnpj}
\big\|\mathcal{N}_{p,i}(g^1,\cdots,g^p)(t)\big\|_{H^s}\leq (Ct^4)^p\sum_{j=1}^p\|g^j(t)\|_{H^{s+1}}\prod_{i\neq j}\|g^j(t)\|_{H^{3}}.
\end{equation}
 Moreover, let $\lambda\in[1/4,4]$ and $\lambda'=\lambda-\delta$, where $0<\delta<\lambda$, and suppose $(\ell_j,\eta_j)\in\mathbb{Z}\times\mathbb{R}$ for $1\leq j\leq p$, such that
\[(\underline{\ell},\underline{\eta})=\sum_{j=1}^p(\ell_j,\eta_j),\quad \max_{1\leq j\leq p}(|\ell_j|+|\eta_j|)=\Lambda,
\] then we have
\begin{multline}\label{estnpj}
\sum_{i=1}^3\sum_{k}\int_{\mathbb{R}}e^{\lambda'(|k-\underline{\ell}|+|\xi-\underline{\eta}|)}\big|\mathscr{F}\mathcal{N}_{p,i}(g^1,\cdots,g^p)(t,k,\xi)\big|^2\,\mathrm{d}\xi\\
\leq t^4(C\max(\delta^{-2},\Lambda))^{2p}\prod_{j=1}^p\sum_{i=1}^3\sum_k\int_{\mathbb{R}}e^{\lambda(|k-\ell_{j}|+|\xi-\eta_{j}|)}\big|\mathscr{F}(g^j)_i(t,k,\xi)\big|^2\,\mathrm{d}\xi.
\end{multline}
\end{proposition}
\begin{proof} By (\ref{nonlins}) we can write each $\mathcal{N}_{p,i}$ as a linear combination of products of at most three factors, at most one of them being $\nabla\Phi_q$ for $p-2\leq q\leq p$. This combined with (\ref{multest2}) imediately proves (\ref{multest3}) and (\ref{bdnpj}). Moreover, using also (\ref{expresspsi}) and (\ref{estimatea}) we can write
\[\mathscr{F}\mathcal{N}_{p,i}(g^1,\cdots,g^p)(t,k,\xi)=\sum_{k_1,\cdot,k_p}\int_{\mathbb{R}^p}\mathcal{B}(k,k_1,\cdots,k_p,\xi,\xi_1,\cdots,\xi_p)\prod_{j=1}^p\sum_{i=1}^3\big|\mathscr{F}(g^j)_i(t,k_j,\xi_j)\big|\,\mathrm{d}\xi_1\cdots\mathrm{d}\xi_p,
\] where the symbol $\mathcal{B}$ satisfies
\begin{equation*}\label{estimatea}|\mathcal{B}|\leq C^pt^2 e^{-(C_0/6)(|k-k_1-\cdots-k_p|+|\xi-\xi_1-\cdots -\xi_p|)}\cdot\big(1+\max_{1\leq j\leq p}|k_j|+\max_{1\leq j\leq p}|\xi_j|\big)^p.
\end{equation*} Then (\ref{estnpj}) is proved by interpolation and direct computation, in the similar way as (\ref{multestshift}).
\end{proof}
\subsection{Detailed information of $g^{(1)}$} The linear solution $g^{(1)}$ is the leading term in the Taylor expansion of $g-\underline{g}$. In this section we summarize all the detailed information about $g^{(1)}$ that follows from the results in Sections \ref{for} and \ref{largeeta}, in the following proposition.

Moreover, we will fix here the choice of $\varepsilon_1$, and introduce the function $B(t)$ - a suitable upper bound for $\|g^{(1)}(t)\|_{L^2}$ - which will play an important role in later proofs.
\begin{proposition}\label{choiceeps1}
Recall the parameters defined in (\ref{defpara}) and (\ref{defineint}). There exists a value of $\varepsilon_1$ and a function $B(t):[1,T_1]\to\mathbb{R}^+$, such that the followings hold.
\begin{enumerate}
\item $B(t)$ is decreasing on $[1,T_0]$ and increasing on $[T_0,T_1]$, and 
\begin{equation}\label{funcbt0} B(1)\leq e^{-\sigma^2k_0/2},\quad B(T_0)\geq e^{-2\sigma^2k_0},\quad B(T_1)=\eta_0^{-(N-1)},
\end{equation} also on $[1,T_1]$ we have
\begin{equation}\label{funcbt1}
|\partial_t\log B(t)|\geq D\bigg(\varepsilon_0\cdot\varepsilon_0^{1/100}+\mathbf{1}_{t\leq T_2}\cdot\varepsilon_0^{1/2}+\frac{1}{t}+\mathbf{1}_{t\leq T_3}\cdot\varepsilon_0(\log k_0)^4+T_0^{-1/4}\bigg).
\end{equation}
\item For any fixed $(n,q)$ and $(k_*,\eta_*)$, satisfying $1\leq q\leq n \leq\varepsilon_0^{-3/4}$, $n\geq 2$ and 
\[|k_*-qk_0|\leq Dn\sqrt{\varepsilon_0}T_0,\quad  |\eta_*-q\eta_0|\leq Dn\sqrt{\varepsilon_0}T_0,
\]
 if one defines $t_m^*$ as in Proposition \ref{interval2}, and  $\nu(\alpha_m^*)$ as in Propositions \ref{interval30} and \ref{interval3}, then
\begin{equation}\label{funcbt2}
\left\{\begin{split}  
&B(t_m^*)^n\geq \nu(\alpha_m^*)\cdot B(t_{m-1}^*)^n,&&\textnormal{if $t_{m-1}^*\leq T_0$ and $t_m^*\geq T_2$};
\\ &B(t_{m-1}^*)^n\geq \nu(\alpha_m^*)\cdot B(t_{m}^*)^n,&&\textnormal{if $t_{m}^*\geq T_0$ and $t_{m-1}^*\leq T_1$}.
\end{split}\right.
\end{equation}
\item For $t\in[1,T_1]$, we have
\begin{equation}\label{funcbt3}
\|g^{(1)}(t)\|_{L^2}\leq \eta_0^{10}B(t),\quad \|f^{(1)}(T_1)\|_{L^2}\geq B(T_1),
\end{equation} and more accurately, 
\begin{equation}\label{funcbt4}
\sum_{k_0/2\leq |k|\leq 2k_0}\int_{\eta_0/2\leq|\xi|\leq 2\eta_0}\big|\mathscr{F}f^{(1)}(T_1,k,\xi)\big|^2\,\mathrm{d}\xi\geq \frac{1}{4}\eta_0^{-2(N-1)}.
\end{equation}
\end{enumerate}
\end{proposition}
\begin{proof} First, we can write 
\begin{equation}\label{defdecomp}g^{(1)}=g^{(1,1)}+g^{(1,-1)},\quad g^{(1,-1)}=\overline{g^{(1,1)}},
\end{equation} where $g^{(1,1)}$ is defined by
\begin{equation}\label{defg11}
\partial_tg^{(1,1)}=\mathcal{L}g^{(1,1)},\quad g^{(1,1)}(T_0)=g'(T_0)=(\varepsilon_1e^{ i(k_0z+\eta_0v)}\varphi_p(k_0\sqrt{\sigma}v)/2,0,0).
\end{equation}
Note that this $g^{(1,1)}$ is precisely the $g'$ we studied in Section \ref{for}, which solves (\ref{inhomo}), case (a); it thus satisfies all the estimates proved in Section \ref{for}.

Recall the sequence $\{B_m\}$ defined in Proposition \ref{summarya}; notice that each is a linear function of $\varepsilon_1$. By \eqref{summarya5} we know that
\[
\varepsilon_1\cdot e^{\sigma^6k_0}\prod_{m=k_1+1}^{k_0}\bigg(\frac{\alpha k_0^2}{m^2}\bigg)\geq B_{k_1}\geq \varepsilon_1\cdot e^{-\sigma^6k_0}\prod_{m=k_1+1}^{k_0}\bigg(\frac{\alpha k_0^2}{m^2}\bigg),
\]by Stirling's formula, this simplifies to
\begin{equation}\label{endest}(b(\sigma)-\sigma^6)k_0\leq\log(\varepsilon_1^{-1}B_{k_1})\leq (b(\sigma)+\sigma^6)k_0,\quad b(\sigma)=\sigma^2+\frac{4\sigma^3}{3}+O(\sigma^4).\end{equation}
We now choose the unique $\varepsilon_1$ such that $B_{k_1}=2\eta_0^{-(N-1)}$, then \eqref{endest} implies that
\[
e^{-3\sigma^2k_0/2}\leq e^{-(b(\sigma)+2\sigma^6)k_0}\leq\varepsilon_1\leq e^{-(b(\sigma)-2\sigma^6)k_0}\leq e^{-3\sigma^2k_0/4}.
\]Now, define $B_1(t)$ such that $B_1(t)$ is constant on $[1,T_2]$; on $[T_2,T_1]$ we define
\begin{equation}\label{defb1t}
B_1(t_m)=\left\{
\begin{split}&B_{m}/2,&k_1\leq m\leq k_0;\\
&(B_{k_0}/2)\prod_{j=k_0+1}^m\nu(\alpha_j)^{10},&k_2\geq m> k_0,
\end{split}
\right.
\end{equation} and that $\log B_1(t)$ is linear in $t$ on each interval $[t_{m},t_{m-1}]$. Moreover, we define $B_2(t)$ by
\begin{equation}\label{defb2t}
B_2(t)=\exp\bigg(\mathbf{1}_{t\leq T_0}\cdot\int_{t}^{T_0}D\bigg(\varepsilon_0\cdot\varepsilon_0^{1/100}+\mathbf{1}_{t'\leq T_2}\cdot\varepsilon_0^{1/2}+\frac{1}{t'}+\mathbf{1}_{t'\leq T_3}\cdot\varepsilon_0(\log k_0)^4+T_0^{-1/4}\bigg)\,\mathrm{d}t'\bigg),
\end{equation}
and define $B(t)=B_1(t)B_2(t)$. It suffices to check that it satisfies all requirements.

For part (1), monotonicity and \eqref{funcbt1} are direct consequences of the definition and \eqref{summarya5},  moreover $B(T_1)=B_{k_1}/2=\eta_0^{-(N-1)}$, and $B(T_0)\geq (2\eta_0)^{-1}\varepsilon_1\geq e^{-2\sigma^2k_0}$. For $B(1)$ we compute
\[B_2(1)\leq \exp(D(\varepsilon_0^{101/100}T_0+\varepsilon_0^{1/2}T_2+\log T_0+\varepsilon_0(\log k_0)^4T_3+T_0^{3/4}))\leq\exp(\varepsilon_0^{1/100}k_0),
\] and moreover
\begin{multline}B_1(1)=B_1(T_2)\leq \varepsilon_1\cdot\prod_{j=k_0+1}^{k_2}\nu(\alpha_j)^{10}\leq\varepsilon_1\prod_{\alpha_j\geq\varepsilon_0^{3/5}/2}e^{10\varepsilon_0^{2/5}}\max(1,\alpha_j)\prod_{\alpha_j<\varepsilon_0^{3/5}/2}e^{40\varepsilon_0^{3/5}}\\
\leq \varepsilon_1   e^{40\varepsilon_0^{3/5}k_2+10\varepsilon_0^{2/5}\varepsilon_0^{-3/10}k_0}\prod_{j=k_0+1}^{k_2}\max\bigg(1,\frac{\alpha k_0^2}{j^2}\bigg)^{10}.
\end{multline} The last product can be bounded, using Stirling's formula, by
\[\prod_{j=k_0+1}^{k_2}\max\bigg(1,\frac{\alpha k_0^2}{j^2}\bigg)^{10}\leq e^{10(\alpha-1)^2k_0}= e^{10\sigma^4k_0},
\] In the end we obtain that 
\[B_1(1)\leq \exp(40\varepsilon_0^{3/5}k_2+10\varepsilon_0^{2/5}\varepsilon_0^{-3/10}k_0+\varepsilon_0^{1/100}k_0+10\sigma^4k_0)\cdot\varepsilon_1\leq e^{-\sigma^2k_0/2}.
\]

For part (2), fix $(k_*,\eta_*)$ and consider the following cases.

(a) Suppose $q\geq 2$, then we can check $\alpha_m^*\leq 7/8<1$. First assume $t_{m-1}^*\leq T_0$, we know that there exists $m'$ such that either $[t_{m}^*,t_{m-1}^*]\subset[t_{m'},t_{m'-1}]$, or $t_{m}^*\in[t_{m'+1},t_{m'}]$ and $t_{m-1}^*\in[t_{m'},t_{m'-1}]$. We may assume the first case, since the second case can be treated by dividing $[t_m^*,t_{m-1}^*]$ into two subintervals and considering the longer one. Now we compute
\[
\frac{t_{m-1}^*-t_{m}^*}{t_{m'-1}-t_{m'}}=\frac{\eta_*}{\eta_0}\cdot\frac{4(m')^2-1}{4m^2-1}\geq\frac{2}{3}\frac{\eta_*(m')^2}{\eta_0m^2}\geq\frac{1}{2}\frac{\eta_0}{\eta_*}\geq\frac{1}{3q}\geq\frac{1}{3n},
\] since $\log B_1(t)$ is linear on $[t_{m'},t_{m'-1}]$ and $B_2(t)$ is decreasing in that interval, we have that
\[\bigg(\frac{B(t_m^*)}{B(t_{m-1}^*)}\bigg)^n\geq\bigg(\frac{B_1(t_{m'})}{B_1(t_{m'-1})}\bigg)^{\frac{1}{3}}\geq(\nu(\alpha_{m'}))^{10/3}\geq \nu(\alpha_m^*),
\] the last inequality due to the fact that 
\[\frac{\alpha_m^*}{\alpha_{m'}}=\frac{\eta_*(m')^2}{\eta_0m^2}\leq \frac{10}{9}\cdot\frac{\eta_0}{\eta_*}\leq \frac{2}{3}<1.
\] When $t_{m}^*\geq T_0$ the proof is similar, using instead the inequality (which is a consequence of \eqref{summarya5},
\[
\frac{B(t_{m'-1})}{B(t_{m'})}=\frac{B_{m'-1}}{B_{m'}}\geq \alpha_{m'}(1-2\sigma^6)\geq(\nu(\alpha_m^*))^3,
\] the last inequality being true since $\alpha_m^*\leq (2/3)\alpha_{m'}<4/5$ and hence $\nu(\alpha_{m}^*)\leq 1+\varepsilon_0^{2/5}$.

(b) Suppose $q=1$, and $t_{m-1}^*\leq T_0$. Due to similar arguments as part (1), we may assume that there is $m'$ such that $[t_{m}^*,t_{m-1}^*]\subset[t_{m'},t_{m'-1}]$. Again we have
\[
\frac{t_{m-1}^*-t_{m}^*}{t_{m'-1}-t_{m'}}\geq\frac{1}{3n},
\] so we may reduce to proving $\nu(\alpha_{m'})^{10/3}\geq\nu(\alpha_m^*)$. Now since $n\leq\varepsilon_0^{-3/4}$, we know that $|\eta_0-\eta_*|\leq\varepsilon_0^{-2}T_0$, hence
\begin{equation}\label{argument}
\frac{\alpha_m^*}{\alpha_{m'}}=\frac{\eta_*(m')^2}{\eta_0m^2}\leq1+O(\eta_0^{-1/3}),
\end{equation} and therefore
$\nu(\alpha_m^*)\leq \nu(\alpha_{m'})+O(\eta_0^{-1/6})\leq \nu(\alpha_{m'})^{10/3}$, using the fact that $\partial_{\beta}\nu(\beta)\leq \varepsilon_0^{-1}$.

(c) Suppose $q=1$, and $t_{m}^*\geq T_0$. Notice that the length of the two intervals $[t_m^*,t_{m-1}^*]$ and $[t_m,t_{m-1}]$ differ by at most $O(\eta_0^{-1/6})$ due to the same argument as in \eqref{argument}, and that 
\[
\frac{\max_{t\in[T_0,T_1]}\partial_t\log B(t)}{\min_{t\in[T_0,T_1]}\partial_t\log B(t)}\leq\sigma^{-2}
\] due to \eqref{summarya5}, we know that \eqref{funcbt2} is true if $n\geq \sigma^{-6}$. Now if $n<\sigma^{-6}$, write $[t_m^*,t_{m-1}^*]\cap[t_m,t_{m-1}]=[t',t'']$, then we have $(t''-t')/(t_{m-1}-t_m)\geq 4/5$, since 
\[|t_{m}^*-t_m|+|t_{m-1}^*-t_{m-1}|\leq \frac{3|\eta-\eta_*|}{m}\leq \frac{\sigma^{6}\sqrt{\varepsilon_0}T_0}{m}\leq \varepsilon_0^{-2/3},
\] while $t_{m-1}-t_m\geq\varepsilon_0^{-1}/10$. Therefore, we have
\[
\bigg(\frac{B(t_m^*)}{B(t_{m-1}^*)}\bigg)^n\geq\bigg(\frac{B(t_{m})}{B(t_{m-1})}\bigg)^{2\cdot(4/5)}\geq \alpha_m^{3/2}\geq\nu(\alpha_m^*),
\] again because of $\alpha_m^*\leq (1+\eta_0^{-1/6})\alpha_m$.

For part (3), the upper bound in \eqref{funcbt3} on $[T_2,T_1]$ is an immediate consequence of \eqref{summarya4} and \eqref{summarya6}. To prove it on $[1,T_1]$, we simply use Proposition \ref{gevrey2} to get (under the notation in that proposition) that 
\[\sqrt{M_1(t)}\leq \sqrt{M_1(T_2)}\cdot \exp\bigg(C\int_{t}^{T_2}\bigg(\sqrt{\varepsilon_0}+\frac{1}{t'}\bigg)\,\mathrm{d}t'\bigg)\leq\eta_0^{2}B(t)
\]by using \eqref{funcbt1}, which implies the upper bound in \eqref{funcbt3}. Notice that on both intervals it is important that one can always make the restriction $|k|+|\xi|\leq\eta_0^2$ while estimating $\widehat{g'}(t,k,\xi)$, thanks to \eqref{farsmall4} and \eqref{farsmall3}, and the fact that $n\leq \varepsilon_0^{-3/4}$. For the lower bound, one simply uses that \[\|g^{(1)}(T_1)\|_{L^2}^2=2\|g^{(1,1)}(T_1)\|_{L^2}^2+2\Re\sum_k\int_{\mathbb{R}}\mathscr{F}g^{(1,1)}(T_1,k,\xi)\cdot\mathscr{F}g^{(1,1)}(T_1,-k,-\xi)\,\mathrm{d}\xi,\] and that the last term is small,
\[
\bigg\|\sum_k\int_{\mathbb{R}}\mathscr{F}g^{(1,1)}(T_1,k,\xi)\cdot\mathscr{F}g^{(1,1)}(T_1,-k,-\xi)\,\mathrm{d}\xi\bigg\|\leq e^{-\eta_0/4},
\] using Plancherel and (\ref{farsmall3}). Finally, \eqref{funcbt4} is a direct consequence of \eqref{funcbt3}, \eqref{funcbt0}, and \eqref{farsmall3}.
\end{proof}
\subsection{Frequency localization estimates for $g^{(n)}$} We now start to estimate $g^{(n)}$ for $n\geq 2$. We first prove Fourier localization estimates for these terms: Proposition \ref{roughgn}, which provides a rough estimate that holds on $[1,T_1]$, and Proposition \ref{bdgnq0} which provides more precise localization on $[T_2,T_1]$.
\begin{proposition}\label{roughgn} Recall the function $\lambda_0(t)$ defined in Proposition \ref{energy1}. Let
\begin{equation}D_n(t)=\bigg(\sum_{i=1}^3\sum_{k}\int_{\mathbb{R}}e^{\tau_n\lambda_0(t)(|k|+|\xi|)}\big|\mathscr{F}(g^{(n)})_i(t,k,\xi)\big|^2\,\mathrm{d}\xi\bigg)^{1/2},
\end{equation} where $\tau_n=1-n\cdot \eta_0^{-N'-1}$, then for all $1\leq n\leq \eta_0^{N'}$ and $t\in[1,T_1]$, we have
\begin{equation}\label{bddnt}
D_n(t)\leq e^{4(2n-1)\eta_0}.
\end{equation}
\end{proposition}
\begin{proof} Notice that $\tau_n\in[1/2,1]$. Recall the weights $A_k$ and $A_*$ in (\ref{defakwei}), Proposition \ref{energy1}. Define, as in (\ref{diffineq}), the energy
\begin{equation}\label{diffineqnew}M_n(t)=\sum_{k}\int_{\mathbb{R}}A_k(t,\xi)e^{\tau_n'\lambda_0(t)(|k|+|\xi|)}\big|\mathscr{F}(g^{(n)})_1(t,k,\xi)\big|^2\,\mathrm{d}\xi+\sum_{j=2}^3\int_{\mathbb{R}}A_*(t,\xi)e^{\tau_n'\lambda_0(t)|\xi|}\big|\mathscr{F}(g^{(n)})_j(t,0,\xi)\big|^2\,\mathrm{d}\xi,\end{equation} where $\tau_n'=(\tau_n+\tau_{n-1})/2$, then we have $D_n(t)\leq\eta_0^{4N'}\sqrt{M_n(t)}$, using the simple fact that
\[A_k(t,\xi)\geq\eta_0^{-4N'}\exp\big(\eta_0^{-N'-1}(|k|+|\xi|)\big),\quad A_*(t,\xi)\geq\eta_0^{-4N'}\exp\big(\eta_0^{-N'-1}|\xi|\big).
\] Moreover, by Proposition \ref{energy1} (with $k_*=\eta_*=0$) we have that
 \begin{equation}\label{largesup1new}
\mathrm{sgn}(t-T_0)\cdot\partial_tM_n(t)\leq C\bigg(\varepsilon_0^{1/2}M_n(t)+\frac{M_n(t)}{t}+\sqrt{M_n(t)}\cdot Z_n(t)\bigg),
\end{equation} where $R_1(t)=0$, and for $n\geq 2$ we have
\[Z_n(t)=\bigg(\sum_{i=1}^3\sum_{k}\int_{\mathbb{R}}e^{\tau_{n}'\lambda_0(t)(|k|+|\xi|)}\big|\mathscr{F}(Z^{(n-1)})_i(t,k,\xi)\big|^2\,\mathrm{d}\xi\bigg)^{1/2}.
\] Since $M_n(T_0)=0$ for $n\geq 2$, and $\sqrt{M_1(T_0)}\leq  e^{2\eta_0}$ due to the definition of $g^{(1)}$, by solving the differential inequality \eqref{largesup1new} we get that $\sqrt{M_1(t)}\leq e^{3\eta_0}$ and hence $D_1(t)\leq e^{4\eta_0}$ for all $t\in[1,T_1]$ . Now suppose \eqref{bddnt} is true for all $n'<n$, then, using \eqref{estnpj}, we get that for each $t\in[1,T_1]$,
\begin{equation*}
\begin{split}
Z_n(t)&\leq \sum_{p=2}^n\sum_{n_1+\cdots +n_p=n}\bigg(\sum_{i=1}^3\sum_{k}\int_{\mathbb{R}}e^{\tau_{n}'\lambda_0(t)(|k|+|\xi|)}\big|\mathscr{F}\mathcal{N}_{p,i}(g^{(n_1)},\cdots,g^{(n_p)})(t,k,\xi)\big|^2\,\mathrm{d}\xi\bigg)^{1/2}\\
&\leq \sum_{p=2}^n\sum_{n_1+\cdots +n_p=n} C^pt^2(\eta_0^{N'+1})^p\prod_{j=1}^p \bigg(\sum_{i=1}^3\sum_{k}\int_{\mathbb{R}}e^{\tau_{n-1}\lambda_0(t)(|k|+|\xi|)}\big|\mathscr{F}(g^{(n_j)})_i(t,k,\xi)\big|^2\,\mathrm{d}\xi\bigg)^{1/2}\\
&\leq  \sum_{p=2}^n\sum_{n_1+\cdots +n_p=n}C^pt^2(\eta_0^{N'+1})^{2p}\prod_{j=1}^p D_{n_j}(t)\leq \sum_{p=2}^n\sum_{n_1+\cdots +n_p=n} C^pt^2(\eta_0^{N'+1})^{2p}\prod_{j=1}^p e^{4(2n_j-1)\eta_0}\\
&\leq e^{4(2n-1)\eta_0}\cdot\sum_{p=2}^n (Cn)^pt^2(\eta_0^{N'+1})^{2p}e^{-4(p-1)\eta_0}\leq e^{4(2n-1)\eta_0}\cdot e^{-2\eta_0},
\end{split}
\end{equation*} since $p\leq n\leq \eta_0^{N'}$, and therefore $Cn\cdot t^2\eta_0^{N'+1}\ll e^{\eta_0}$. Solving the differential inequality \eqref{largesup1new}, we easily get that
\[
\sup_{t\in[1,T_1]}\sqrt{M_n(t)}\leq e^{\eta_0}\cdot\sup_{t\in[1,T_1]}R_n(t)\leq e^{4(2n-1)\eta_0}e^{-\eta_0},
\] and hence 
\[\sup_{t\in[1,T_1]}D_n(t)\leq \eta_0^{4N'}\cdot\sup_{t\in[1,T_1]}M_n(t)\leq e^{4(2n-1)\eta_0}.
\]This completes the inductive proof.
\end{proof}
\begin{proposition}\label{bdgnq0}For each $1\leq n\leq \eta_0^{N'}$, we can decompose
\begin{equation}
g^{(n)}=\sum_{\substack{|q|\leq n\\q\equiv n\!\!\!\!\!\pmod 2}}g^{(n,q)},
\end{equation} such that for each $q$ and all $t\in[T_2,T_1]$ we have
\begin{equation}\label{bdgnq}
\bigg(\sum_{i=1}^3\sum_{k}\int_{\mathbb{R}}e^{\tau_n\lambda_0(t)(|k-qk_0|+|\xi-q\eta_0|)}\big|\mathscr{F}(g^{(n,q)})_i(t,k,\xi)\big|^2\,\mathrm{d}\xi\bigg)^{1/2}\leq e^{D(2n-1)\sqrt{\varepsilon_0}T_0}\cdot e^{Dn\varepsilon_0\eta_0|\frac{1}{t}-\frac{1}{T_0}|}.
\end{equation}
\end{proposition}
\begin{proof} Define
\[F_{n,q}(t)=\bigg(\sum_{i=1}^3\sum_{k}\int_{\mathbb{R}}e^{\tau_n\lambda_0(t)(|k-qk_0|+|\xi-q\eta_0|)}\big|\mathscr{F}(g^{(n,q)})_i(t,k,\xi)\big|^2\,\mathrm{d}\xi\bigg)^{1/2}
\] and, as in Proposition \ref{energy1},
\begin{multline}
M_{n,q}(t)=\sum_{k}\int_{\mathbb{R}}A_k(t,\xi)e^{\tau_n'\lambda_0(t)(|k-qk_0|+|\xi-q\eta_0|)}\big|\mathscr{F}(g^{(n,q)})_1(t,k,\xi)\big|^2\,\mathrm{d}\xi\\+\sum_{j=2}^3\int_{\mathbb{R}}A_*(t,\xi)e^{\tau_n'\lambda_0(t)(|qk_0|+|\xi-k\eta_0|)}\big|\mathscr{F}(g^{(n,q)})_j(t,0,\xi)\big|^2\,\mathrm{d}\xi,
\end{multline}where $\tau_n'=(\tau_n+\tau_{n-1})/2$. Again we have $F_{n,q}(t)\leq  \eta_0^{4N'}\sqrt{M_{n,q}(t)}$ as long as $n\leq \eta_0^{N'}$.

When $n=1$, we simply define $g^{(1,\pm1)}$ as in (\ref{defdecomp}) and (\ref{defg11}); by Proposition \ref{energy1} (with $(k_*,\eta_*)=\pm(k_0,\eta_0)$), we have that 
 \begin{equation}\label{largesup1new2}
\mathrm{sgn}(t-T_0)\cdot\partial_tM_{1,\pm 1}(t)\leq C\bigg(\varepsilon_0^{1/2}M_{1,\pm 1}(t)+\frac{M_{1,\pm 1}(t)}{t}\bigg);
\end{equation} since $\sqrt{M_{1,\pm 1}(T_0)}\leq e^{Ck_0}\leq e^{C\sqrt{\varepsilon_0}T_0}$, we can solve this differential inequality to get that
\[
\sqrt{M_{1,\pm 1}(t)}\leq e^{D\sqrt{\varepsilon_0}T_0},
\] which implies \eqref{bdgnq}. 

Now suppose the decomposition can be made, and \eqref{bdgnq} holds for all $n'<n$, then we can define, for $|q|\leq n$ and $q\equiv n\pmod 2$, that
\[\partial_t g^{(n,q)}=\mathcal{L}g^{(n,q)}+Z^{(n-1,q)},
\] where 
\begin{equation}\label{defrnq}
Z^{(n-1,q)}=\sum_{p=2}^n\sum_{n_1+\cdots +n_p=n}\sum_{\substack{q_1+\cdots +q_p=q\\|q_j|\leq n_j,\,q_j\equiv n_j\!\!\!\!\!\pmod 2}}\mathcal{N}_{p}(g^{(n_1,q_1)},\cdots,g^{(n_p,q_p)}).
\end{equation} Clearly these form a decomposition of $g^{(n)}$; to prove \eqref{bdgnq}, we use \eqref{estnpj} and argue in the same way as in the proof of Proposition \ref{roughgn} above, to deduce that
\begingroup
\allowdisplaybreaks
\begin{align*}
Z_{n,q}(t)&:=\bigg(\sum_{i=1}^3\sum_{k}\int_{\mathbb{R}}e^{\tau_{n}'\lambda_0(t)(|k-qk_0|+|\xi-q\eta_0|)}\big|\mathscr{F}(Z^{(n-1,q)})_i(t,k,\xi)\big|^2\,\mathrm{d}\xi\bigg)^{1/2}\\
&\leq \sum_{p=2}^n\sum_{\substack{n_1+\cdots +n_p=n\\q_1+\cdots +q_p=q}}\bigg(\sum_{i=1}^3\sum_{k}\int_{\mathbb{R}}e^{\tau_{n}'\lambda_0(t)(|k-qk_0|+|\xi-q\eta_0|)}\big|\mathscr{F}\mathcal{N}_{p,i}(g^{(n_1,q_1)},\cdots,g^{(n_p,q_p)})(t,k,\xi)\big|^2\,\mathrm{d}\xi\bigg)^{1/2}\\
&\leq \sum_{p=2}^n\sum_{\substack{n_1+\cdots +n_p=n\\q_1+\cdots +q_p=q}} C^pt^2(\eta_0^{N'+1})^{2p}\prod_{j=1}^p \bigg(\sum_{i=1}^3\sum_{k}\int_{\mathbb{R}}e^{\tau_{n-1}\lambda_0(t)(|k-q_jk_0|+|\xi-q_j\eta_0|)}\big|\mathscr{F}(g^{(n_j,q_j)})_i(t,k,\xi)\big|^2\,\mathrm{d}\xi\bigg)^{1/2}\\
&\leq  \sum_{p=2}^n\sum_{\substack{n_1+\cdots +n_p=n\\q_1+\cdots +q_p=q}} C^pt^2(\eta_0^{N'+1})^{2p}\prod_{j=1}^p F_{n_j,q_j}(t)
\\
&\leq \sum_{p=2}^n\sum_{\substack{n_1+\cdots +n_p=n\\q_1+\cdots +q_p=q}} C^pt^2(\eta_0^{N'+1})^{2p}\prod_{j=1}^p e^{D(2n_j-1)\sqrt{\varepsilon_0}T_0}\cdot e^{Dn_j\varepsilon_0\eta_0|\frac{1}{t}-\frac{1}{T_0}|}\\
&\leq e^{D(2n-1)\sqrt{\varepsilon_0}T_0}e^{Dn\varepsilon_0\eta_0|\frac{1}{t}-\frac{1}{T_0}|}\cdot\sum_{p=2}^n (Cn)^pt^2(\eta_0^{N'+1})^{2p}e^{-D(p-1)\sqrt{\varepsilon_0}T_0}\\
&\leq e^{D(2n-1)\sqrt{\varepsilon_0}T_0}e^{Dn\varepsilon_0\eta_0|\frac{1}{t}-\frac{1}{T_0}|}\cdot e^{-D\sqrt{\varepsilon_0}T_0/2}:=\widetilde{N}(t).
\end{align*} 
\endgroup Moreover, by Proposition \ref{energy1} (with $(k_*,\eta_*)=(qk_0,q\eta_0)$), we have, for $t\in[T_2,T_1]$, that
 \begin{equation}\label{largesup1new3}
\mathrm{sgn}(t-T_0)\cdot\partial_tM_{n,q}(t)\leq C\bigg(\varepsilon_0^{1/2}M_{n,q}(t)+\frac{n\varepsilon_0\eta_0}{t^2}M_{n,q}(t)+\sqrt{M_{n,q}(t)}\cdot \widetilde{N}(t)\bigg).
\end{equation} To solve this differential inequality, we define
\[\widetilde{M}(t)=e^{-\gamma(t)} M_{n,q}(t),\quad \gamma(t)=C\varepsilon_0^{1/2}|t-T_0|+Cn\varepsilon_0\eta_0\bigg|\frac{1}{t}-\frac{1}{T_0}\bigg|,
\] then we have
\[\mathrm{sgn}(t-T_0)\cdot\partial_t\widetilde{M}(t)\leq C e^{-\gamma(t)/2}\sqrt{\widetilde{M}(t)}\cdot \widetilde{R}(t),
\] and hence
\[\widetilde{M}(t)\leq C^2\bigg(\int_{T_0}^t e^{-\gamma(t')/2}\widetilde{R}(t')\,\mathrm{d}t'\bigg)^2\leq C^2\bigg(\int_{T_0}^t e^{-\frac{Cn\varepsilon_0\eta_0}{2}|\frac{1}{t'}-\frac{1}{T_0}|}\cdot\widetilde{R}(t')\,\mathrm{d}t'\bigg)^2
\] Since $D\gg C$, the function \[\exp\bigg(-\frac{Cn\varepsilon_0\eta_0}{2}\bigg|\frac{1}{t'}-\frac{1}{T_0}\bigg|\bigg)\cdot\widetilde{R}(t')\] is decreasing on $[T_2,T_0]$ and increasing on $[T_0,T_1]$, thus
\[\sqrt{\widetilde{M}(t)}\leq 2CT_0 e^{-\frac{Cn\varepsilon_0\eta_0}{2}|\frac{1}{t}-\frac{1}{T_0}|}\cdot\widetilde{R}(t)
\leq e^{2C\sqrt{\varepsilon_0}T_0}\cdot e^{-\gamma(t)/2}\widetilde{R}(t),\] and hence
\[\sqrt{M_{n,q}(t)}\leq e^{2C\sqrt{\varepsilon_0}T_0}\widetilde{R}(t);\quad F_{n,q}(t)\leq e^{4C\sqrt{\varepsilon_0}T_0}\widetilde{R}(t)\leq e^{D(2n-1)\sqrt{\varepsilon_0}T_0}e^{Dn\varepsilon_0\eta_0|\frac{1}{t}-\frac{1}{T_0}|}.
\] This completes the proof of \eqref{roughgn}.
\end{proof}
\subsection{$L^2$ estimates for $g^{(n)}$} We then proceed to bound the $L^2$ norm of $g^{(n)}$, essentially by $(B(t))^n$. Here we use crucially the upper bounds proved in Proposition \ref{interval3}.
\begin{proposition}\label{l2allgn}
For all $1\leq n\leq \eta_0^{N'}$, $0\leq |q|\leq n$, and all $t\in[T_2,T_1]$, we have
\begin{equation}
\label{l2allgn2}
\|g^{(n,q)}(t)\|_{L^2}\leq \eta_0^{20N'(2n-1)}(B(t))^n.
\end{equation} Moreover, for all $1\leq n\leq \eta_0^{N'}$ and all $t\in[1,T_1]$, we have
\begin{equation}
\label{l2allgn3}
\|g^{(n)}(t)\|_{L^2}\leq \eta_0^{20N'(2n-1)}(B(t))^n.
\end{equation}
\end{proposition}
\begin{proof} First we prove \eqref{l2allgn2} by inducting in $n$. When $n=1$, \eqref{l2allgn2} is a consequence of \eqref{funcbt3}; suppose \eqref{l2allgn2} is true for all $n'<n$, then we have, for $Z^{(n-1,q)}$ defined as in \eqref{defrnq}, that
\begin{equation}\label{bdrnq}
\bigg(\sum_{i=1}^3\sum_{k}\int_{\mathbb{R}}e^{\lambda_0(t)(|k-qk_0|+|\xi-q\eta_0|)/2}\big|\mathscr{F}(Z^{(n-1,q)})_i(t,k,\xi)\big|^2\,\mathrm{d}\xi\bigg)^{1/2}\leq e^{Dn\sqrt{\varepsilon_0}T_0}
\end{equation}
as shown in the proof of Proposition \ref{bdgnq0}, and the fact that
\[Dn\varepsilon_0\eta_0\bigg|\frac{1}{t}-\frac{1}{T_0}\bigg|\leq Dn\varepsilon_0\eta_0T_2^{-1}\leq\varepsilon_0^{1/10}\cdot Dn\sqrt{\varepsilon_0}T_0
\] for $t\in[T_2,T_1]$. Moreover, by \eqref{bdnpj} we can estimate
\begin{equation}\label{controlrnq}
\begin{split}
\|Z^{(n-1,q)}(t)\|_{L^2}&\leq \sum_{p=2}^n\sum_{\substack{n_1+\cdots+n_p=n\\q_1+\cdots+q_p=q}}\sum_{i=1}^3\|\mathscr{F}\mathcal{N}_{p,i}(g^{(n_1,q_1)},\cdots,g^{(n_p,q_p)})(t)\|_{L^2}\\
&\leq \sum_{p=2}^n\sum_{\substack{n_1+\cdots+n_p=n\\q_1+\cdots+q_p=q}}\sum_{i=1}^3 C^pt^{4p}\prod_{j=1}^p\sum_{i=1}^3\big\|(|k_j|+|\xi_j|)^{3}\mathscr{F}g^{(n_j,q_j)}(t,k_j,\xi_j)\big\|_{L^2}\\
&\leq \sum_{p=2}^n C^pt^{4p}\eta_0^{4N'p}\prod_{j=1}^p\eta_0^{20N'(2n_j-1)}(B(t))^{n_j}\leq \eta_0^{20N'(2n-1)}(B(t))^n\cdot \eta_0^{-12N'}
\end{split}
\end{equation} using the induction hypothesis and \eqref{bdgnq}. We now proceed to the estimate of $g^{(n,q)}$, which satisfies the equation 
\[\partial_tg^{(n,q)}=\mathcal{L}g^{(n,q)}+Z^{(n-1,q)},
\] by dividing into several cases.

\underline{Case 1}: Suppose $n\geq \varepsilon_0^{-21/40}$. Define the energy $M_1(t)$ associated to the function $g^{(n,q)}$ as in Proposition \ref{gevrey2}, then we have
\begin{equation}\label{diffineqnew}\mathrm{sgn}(t-T_0)\cdot\partial_tM_1(t)\leq C\bigg(\varepsilon_0^{1/2}M_1(t)+\frac{M_1(t)}{t}+\sqrt{M_1(t)}\cdot \|Z^{(n-1,q)}(t)\|_{L^2}\bigg),
\end{equation} where
\[\|Z^{(n-1,q)}(t)\|_{L^2}\leq  \eta_0^{20N'(2n-1)}\eta_0^{-12N'}\cdot (B(t))^n.
\] We know $M_1(T_0)=0$; moreover since $D\gg C$, $\mathrm{sgn}(t-T_0)\cdot \partial_tB(t)\geq 0$, and
\[
|\partial_t\log (B(t)^n)|=n|\partial_t\log B(t)|\geq\varepsilon_0^{-21/40}\cdot\bigg(\varepsilon_0^{101/100}+\frac{1}{t}\bigg)\geq D\bigg(\varepsilon_0^{1/2}+\frac{1}{t}\bigg),\]we know that $(B(t))^ne^{-\gamma(t)/2}$ is decreasing on $[1,T_0]$ and decreasing on $[T_0,T_1]$, where
\[
\gamma(t)=\bigg|C\int_{T_0}^t\bigg(\varepsilon_0^{1/2}+\frac{1}{t'}\bigg)\,\mathrm{d}t'\bigg|.\] Let $\widetilde{M}(t)=M_1(t)e^{-\gamma(t)}$, then we have \[\mathrm{sgn}(t-T_0)\cdot\partial_t\widetilde{M}(t)\leq C e^{-\gamma(t)/2}\sqrt{\widetilde{M}(t)}\cdot\eta_0^{20N'(2n-1)}\eta_0^{-12N'}\cdot (B(t))^n.\] By the above monotonicity we get that
\[\sqrt{\widetilde{M}(t)}\leq \eta_0^{20N'(2n-1)}\eta_0^{-10N'}\cdot e^{-\rho(t)/2}(B(t))^n,
\]hence 
\[\|g^{(n,q)}\|_{L^2}\leq \eta_0^{10N'}\sqrt{M_1(t)}\leq \eta_0^{20N'(2n-1)}\eta_0^{-8N'}(B(t))^n,
\] where the first inequality following from the fact that $\mathscr{F}g^{(n,q)}(t)$ is essentially supported in $|k|+|\xi|\leq \eta_0^{4N'}$, see (\ref{bdgnq}).

\underline{Case 2}: Suppose $n<\varepsilon_0^{-21/40}$. If we restrict, by adding suitable cutoff functions to $\mathscr{F}Z^{(n-1,q)}(t,k,\xi)$, to the region where $|k-qk_0|+|\xi-q\eta_0|\geq 20Dn\sqrt{\varepsilon_0}T_0$. By \eqref{bdrnq}, this contribution will have $L^2$ norm bounded by $e^{-5nD\sqrt{\varepsilon_0}T_0}$; by Proposition \ref{gevrey2}, the corresponding contribution to $g^{(n,q)}$ satisfies a bound that is worse only by a factor of $e^{C\sqrt{\varepsilon_0}T_0}$. Therefore, the $L^2$ norm of this contribution to $g^{(n,q)}$ is bounded by
\[e^{-4nD\sqrt{\varepsilon_0}T_0}\leq \eta_0^{20N'(2n-1)}\eta_0^{-10N'}(B(t))^n,
\] since $B(t)\geq B(T_0)\geq e^{-2\sigma^2k_0}$.

Now we can focus on the case where $n<\varepsilon_0^{-21/40}$, and that we restrict to the region where $|k-qk_0|+|\xi-q\eta_0|\geq 20Dn\sqrt{\varepsilon_0}T_0$ in $\mathscr{F}Z^{(n-1,q)}(t,k,\xi)$. For simplicity we will denote $20D$ by $D$, and further decompose 
\[\mathscr{F}Z^{(n-1,q)}(t,k,\xi)=\sum_{|k_*-qk_0|+|\eta_*-q\eta_0|< Dn\sqrt{\varepsilon_0}T_0}\mathscr{F}\rho_{(k_*,\eta_*)}(t,k,\xi),
\] where $(k_*,\eta_*)\in\mathbb{Z}^2$, each $\rho_{(k_*,\eta_*)}$ satisfies that
\[\mathrm{supp}(\mathscr{F}\rho_{(k_*,\eta_*)})\subset\{k_*\}\times[\eta_*-2,\eta_*+2],\] \[\|\rho_{(k_*,\eta_*)}(t)\|_{L^2}\leq\|Z^{(n-1,q)}(t)\|_{L^2}\leq \eta_0^{20N'(2n-1)}(B(t))^n\cdot \eta_0^{-12N'}.
\] Since $n< \varepsilon_0^{-21/40}$, there are at most $O(\eta_0^{10})$ choices of $(k_*,\eta_*)$, so below we will fix one particular choice of $(k_*,\eta_*)$, and denote this $\rho_{(k_*,\eta_*)}$ by $\rho$.

\underline{Case 2.1}: Suppose $q=0$. Then we have $|\eta_*|\leq T_0^{21/20}$ and  $|k_*|\leq T_0^{21/20}$. Moreover we know that $|\partial_t\log Z(t)|\gg T_0^{-1/4}$ on $[T_3,T_1]$ and 
\begin{equation}\label{gap}|\partial_t\log Z(t)|\gg\varepsilon_0(\log k_0)^3
\end{equation} on $[T_2,T_3]$ due to \eqref{funcbt1}, where \[Z(t)= \eta_0^{90N'(2n-1)}(B(t))^n\cdot \eta_0^{-50N'},\] thus we can apply Proposition \ref{lastone}. Let $M_2$ associated to the function $g^{(n,q)}$ be defined in \eqref{defm4}, then by \eqref{improve} we have $\|g^{(n,q)}(t)\|_{L^2}\leq \eta_0^\cdot Z(t)$ on $[T_3,T_1]$ and in particular $M_2(T_3)\leq \|g^{(n,q)}(T_3)\|_{L^2}^2\leq \eta_0^2(Z(T_3))^2$; by \eqref{improvedgrowth} we then have
\[\sqrt{M_2(t)}\leq\max(\eta_0 R(t),e^{C\varepsilon_0(\log k_0)^3}(T_3-t)R(T_3))\leq \eta_0Z(t)
\] for all $t\in[T_2,T_3]$. By restricting to the range $|k|+|\xi|\leq \eta_0^{2N'}$ for $\mathscr{F}g^{(n,q)}(t,k,\xi)$ using (\ref{bdgnq}), we get that
\[\|g^{(n,q)}(t)\|_{L^2}\leq \eta_0^{6N'}Z(t)\leq \eta_0^{20N'(2n-1)}\eta_0^{-6N'}(B(t))^n.
\]

\underline{Case 2.2}: Suppose $|q|\geq 1$. Since $g^{(n,-q)}=\overline{g^{(n,q)}}$, we may assume $q>0$; moreover, by \eqref{funcbt2} we know that for \[Z(t)= \eta_0^{20N'(2n-1)}(B(t))^n\cdot \eta_0^{-12N'},\] the assumption (\ref{conditionr}) in Proposition \ref{interval3}, is satisfied, so we can directly use \eqref{nextest2}, and also by restricting to the range $|k|+|\xi|\leq \eta_0^{2N'}$, to deduce that
\[\|g^{(n,q)}(t)\|_{L^2}\leq\varepsilon_0^{-2}Z(t),\,\forall t\in[T_2,T_1].
\]This completes the proof of \eqref{l2allgn2}, which also implies \eqref{l2allgn3} on $[T_2,T_1]$.

Finally, the proof of \eqref{l2allgn3} on $[1,T_2]$ is similar. In the inductive step, we just use Proposition \ref{roughgn}, instead of Proposition \ref{bdgnq0}, to deduce that $\mathscr{F}g^{(n)}(t,k,\xi)$ is essentially supported in $|k|+|\xi|\leq \eta_0^{2N'}$, then control $\|Z^{(n-1,q)}(t)\|_{L^2}$ using \eqref{controlrnq}, and finally use Proposition \ref{gevrey2} to transform this bound into the bound for $g^{(n)}$. Here we have used the fact that 
\[|\partial_t(B(t))^n|\geq|\partial_tB(t)|\gg C\bigg(\varepsilon_0^{1/2}+\frac{1}{t}\bigg)
\] for all $t\in[1,T_2]$. We omit the details.
\end{proof}
\subsection{Bounds for the error terms} In this subsection we establish the bounds for the approximate solution $G^{(n)}$, and the error term $E^{(n)}$. They are essentially consequences of Proposition \ref{l2allgn}.
\begin{proposition}\label{approx} Recall the norms $\mathcal{G}_{\lambda}^*$ defined in \eqref{defglambda}. Let $G^{(n)}$ and $E^{(n)}$ be defined in \eqref{defcapgn} and \eqref{defen}. Then for $10\leq n\leq \eta_0^{N'}$ and $t\leq T_0$ we have that
\begin{equation}\label{multest5}
\|G^{(n)}\|_{\mathcal{G}_2^*}\leq e^{-\sigma^2k_0/4},\quad \|E^{(n)}\|_{\mathcal{G}_2^*}\leq e^{-n\sigma^2k_0/4};
\end{equation} for $10\leq n\leq \eta_0^{N'}$ and $t\geq T_0$ we have that
\begin{equation}\label{multest6}
\|G^{(n)}\|_{H^{N+1}}\leq\eta_0^{15},\quad \|G^{(n)}\|_{H^{20}}\leq\eta_0^{-N/2},\quad \|E^{(n)}\|_{H^{N+1}}\leq \eta_0^{-nN/8}.
\end{equation}
\end{proposition}
\begin{proof} Considering $G^{(n)}$, by combining Proposition \ref{roughgn} and \eqref{l2allgn3}, we can bound
\[\|g^{(m)}(t)\|_{\mathcal{G}_4^*}\leq \eta_0^{40mN'}(B(t))^m\exp\bigg(\frac{C\sqrt{m\eta_0}}{(\log k_0)^{N_1}}\bigg)\leq\exp\bigg(Cm-\frac{\sigma^2mk_0}{2}+\frac{C\sqrt{\eta_0m}}{(\log k_0)^{N_1}}\bigg)
\] for $t\in[1,T_0]$. Note that $\sqrt{\eta_0}\sim k_0(\log k_0)^{N_3/2}$; since by our choice $N_1=2N_3$, we can prove that the exponent is bounded above by $-\sigma^2mk_0/3$, thus
\[\|G^{(n)}(t)\|_{\mathcal{G}_4^*}\leq \sum_{m=1}^ne^{-\sigma^2mk_0/3}\leq e^{-\sigma^2k_0/4}.
\] On $[T_0,T_1]$ we similarly have
\[\|g^{(m)}(t)\|_{H^{20}}\leq\eta_0^{40mN'}(Cm\eta_0)^{20}(B(t))^m\leq \eta_0^{40mN'}(Cm\eta_0)^{20}\cdot\eta_0^{m(-N+1)}.
\] By our choice $N=100N'$, so $\|g^{(m)}(t)\|_{H^{20}}\leq \eta_0^{-mN/2}$ and hence $\|G^{(n)}(t)\|_{H^{20}}\leq \eta_0^{-N/2}$. As for the $H^{N+1}$ norm, we have \[\|g^{(1)}(t)\|_{H^{N+1}}\leq(C\eta_0)^{N+1}\eta_0^{10}B(t)\leq \eta_0^{15}\] due to \eqref{funcbt0} and \eqref{funcbt3}, and for $m\geq 2$, by \eqref{l2allgn3} we have
\begin{equation}\label{highhnbd}
\|g^{(m)}(t)\|_{H^{N+1}}\leq (Cm\eta_0)^{N+1}\eta_0^{40mN'}(B(t))^m\leq (Cm\eta_0)^{N+1}\eta_0^{40mN'}\cdot\eta_0^{-m(N-1)},
\end{equation} which again implies that 
\begin{equation}\label{order2}\|G^{(n)}(t)-g^{(1)}(t)\|_{H^{N+1}}\leq\sum_{m=2}^{n}\|g^{(m)}(t)\|_{H^{N+1}}\leq\eta_0^{-N/2}
\end{equation} due to our choice of $N=100N'$. This proves the estimates for $G^{(n)}$. Considering $E^{(n)}$, for $t\in[1,T_0]$ we use \eqref{multest3} to bound
\[
\begin{split}\|E^{(n)}(t)\|_{\mathcal{G}_2^*}&\leq\sum_{p=2}^{\infty}\sum_{\substack{n_1,\cdots,n_p\leq n;\\n_1+\cdots+n_p>n}}\sum_{i=1}^3\big\|\mathcal{N}_{p,i}(g^{(n_1)},\cdots,g^{(n_p)})\big\|_{\mathcal{G}_2^*}
\leq \sum_{p=2}^{\infty}\sum_{\substack{n_1,\cdots,n_p\leq n;\\n_1+\cdots+n_p>n}}C^pt^{4p}\prod_{j=1}^p\|g^{(n_j)}(t)\|_{\mathcal{G}_4^*}\\
&\leq \sum_{p=2}^{\infty}\sum_{\substack{n_1,\cdots,n_p\leq n;\\n_1+\cdots+n_p>n}}C^pt^{4p}\prod_{j=1}^pe^{-\sigma^2n_jk_0/3}\leq\sum_{p=2}^{\infty}\sum_{n'\geq\max(n,p)}C^pn^pt^{4p}e^{-\sigma^2n'k_0/3}\\&\leq\sum_{n'\geq n}\big(Cnt^{4}e^{-\sigma^2k_0/3}\big)^{n'}\leq (Cnt^{4})^ne^{-n\sigma^2k_0/3}\leq e^{-n\sigma^2k_0/4};
\end{split}
\] for $t\in[T_0,T_1]$, we can use \eqref{bdnpj} to bound
\[
\begin{split}\|E^{(n)}(t)\|_{H^{N+1}}&\leq\sum_{p=2}^{\infty}\sum_{\substack{n_1,\cdots,n_p\leq n;\\n_1+\cdots+n_p>n}}\sum_{i=1}^3\big\|\mathcal{N}_{p,i}(g^{(n_1)},\cdots,g^{(n_p)})\big\|_{H^{N+1}}\\
&\leq \sum_{p=2}^{\infty}\sum_{\substack{n_1,\cdots,n_p\leq n;\\n_1+\cdots+n_p>n}}C^pt^{4p}\sum_{j=1}^p\|g^{(n_j)}(t)\|_{H^{N+2}}\prod_{i\neq j}\|g^{(n_i)}(t)\|_{H^{3}}.
\end{split}
\] By similar arguments as in \eqref{highhnbd} we can show that $\|g^{(m)}(t)\|_{H^{N+2}}\leq \eta_0^{N+5}\cdot\eta_0^{-nN/2}$, thus
\[
\|E^{(n)}(t)\|_{H^{N+1}}\leq\sum_{p=2}^{\infty}\sum_{n'\geq\max(n,p)}C^pn^pt^{4p}\cdot\eta_0^{N+5}\cdot\eta_0^{-n'N/2}\leq (Cnt^4)^n\eta_0^{-nN/2}\eta_0^{N+5}\leq \eta_0^{-nN/4},
\] provided $10\leq n\leq \eta_0^{N'}$. This completes the proof.
\end{proof}
\section{Proof of Proposition \ref{main2}}\label{error}
In this section we finish the proof of Proposition \ref{main2}. Let $\varepsilon_1$ be fixed as in Proposition \ref{choiceeps1}, and the data $g(T_0)$ be fixed as in (\ref{perturbdata}). The system \eqref{eulernewrewrite} is locally well-posed in any reasonable space (say $H^N$ or $\mathcal{G}_{\lambda}^*$), so we may construct the solution $g=(f,h,\theta)$ on a time interval containing $T_0$. Note that $\underline{g}$ satisfies \eqref{aux} at time $1$, so it satisfies \eqref{aux} at time $T_0$; since \[\mathbb{P}_0f(T_0)=\mathbb{P}_0\underline{f}(T_0),\quad h(T_0)=\underline{h}(T_0),\quad\theta(T_0)=\underline{\theta}(T_0)\] due to \eqref{perturbdata}, we know that $g$ also satisfies \eqref{aux} at time $T_0$, and hence at all times. 

We now turn to (\ref{newbd})$\sim$(\ref{newbd2}). Let $n_0=\eta_0^{N'}$, $G^{(n_0)}$ and $E^{(n_0)}$ be defined in \eqref{defcapgn} and \eqref{defen}. We denote $(g^{\#},\rho^{\#})=(G^{(n_0)},E^{(n_0)})$, and define $\widetilde{g}=g^*-g^{\#}$. We then have
\begin{equation}\label{approxeqn}
\partial_tg^{\#}=\mathcal{L}g^{\#}+\mathcal{N}(g^{\#})-\rho^{\#},
\end{equation} where 
\begin{equation}\label{bdfinal1}
\|g^{\#}(t)\|_{\mathcal{G}_{2}^{*}}\leq e^{-\sigma^2k_0/4},\quad \|\rho^{\#}(t)\|_{\mathcal{G}_2^*}\leq e^{-\eta_0^{N'}}
\end{equation} for $t\in[1,T_0]$, and 
\begin{equation}\label{bdfinal2}
\|g^{\#}(t)\|_{H^{N+1}}\leq\eta_0^{15},\quad \|g^{\#}(t)\|_{H^{20}}\leq\eta_0^{-N/2},\quad \|\rho^{\#}\|_{H^{N+1}}\leq e^{-\eta_0^{N'}}.
\end{equation} Note that $\widetilde{g}$ satisfies the equation
\begin{equation}\label{eqndiff}\partial_t\widetilde{g}=\mathcal{L}\widetilde{g}+\sum_{p=2}^{\infty}\big(\mathcal{N}_p(g^*,\cdots,g^*)-\mathcal{N}_p(g^{\#},\cdots,g^{\#})\big)+\rho^{\#}; 
\end{equation} moreover, using \eqref{nonlins} and \eqref{multest1}$\sim$\eqref{multest2}, we can write 
\begin{equation}\label{semilinear}\mathcal{L}_i(\widetilde{g})=\Psi_{0,i}\cdot\nabla \widetilde{g_i}+\Phi_{1,i}(\widetilde{g}),\quad \mathcal{N}_{p,i}(g^*,\cdots,g^*)=\Psi_{p-1,i}(g^*,\cdots,g^*)\cdot\nabla g_i+\Phi_{p,i}(g^*,\cdot,g^*),
\end{equation} and similarly for $\mathcal{N}_{p,i}(g^\#,\cdots,g^\#)$, where $\Psi_{p,i}$ and $\Phi_{p,i}$ are nonlinearities (or multilinear operators) of order $p$, such that
\begin{equation}\label{semiest}
\|\Psi_{p,i}(g^1,\cdots,g^p)(t)\|_{\mathcal{G}_\lambda^*}+\|\Phi_{p,i}(g^1,\cdots,g^p)(t)\|_{\mathcal{G}_\lambda^*}\leq C^pt^{4p}\prod_{j=1}^p\|g^j(t)\|_{\mathcal{G}_{\lambda}^*}
\end{equation} for any $1/4\leq \lambda\leq 4$, and
\begin{equation}\label{semiest2}
\|\Psi_{p,i}(g^1,\cdots,g^p)(t)\|_{H^N}+\|\Phi_{p,i}(g^1,\cdots,g^p)(t)\|_{H^N}\leq C^pt^{4p}\sum_{j=1}^p\|g^j(t)\|_{H^N}\prod_{i\neq j}\|g^i(t)\|_{H^{3}}.
\end{equation} This implies that
\begin{equation}\label{eqndiff2}
\begin{split}
\partial_t\widetilde{g_i}&=\Psi_{0,i}\cdot\nabla \widetilde{g_i}+\Phi_{1,i}(\widetilde{g})+\sum_{p=2}^{\infty}\big(\nabla\widetilde{g_i}\cdot\Psi_{p-1,i}(g^*,\cdots,g^*)\big)+\rho_i^{\#}\\
&+\sum_{p=2}^{\infty}\sum_{j=1}^{p-1}\nabla g_i^{\#}\cdot\Psi_{p-1,i}(\underbrace{g^*,\cdots,g^*}_{j-1},\widetilde{g},\underbrace{g^{\#},\cdots,g^{\#}}_{p-1-j})+\sum_{p=2}^{\infty}\sum_{j=1}^p\Phi_{p,i}(\underbrace{g^*,\cdots,g^*}_{j-1},\widetilde{g},\underbrace{g^{\#},\cdots,g^{\#}}_{p-j}).
\end{split}
\end{equation}
Below we will consider the interval $[T_0,T_1]$ and $[1,T_0]$ separately.

(1) When $t\geq T_0$, we will control $\widetilde{g}$ in $H^{5}$. Suppose $\|\widetilde{g}\|_{H^{5}}\leq \eta_0^{-N/2}$ up to some time, then we have $\|g^{\#}(t)\|_{H^{6}}+\|g^*(t)\|_{H^{5}}\leq2\eta_0^{-N/2}$ up to this time. Now we compute, using \eqref{eqndiff2}, that
\begin{equation}\label{diffenergy1}
\begin{split}
\partial_t\|\widetilde{g}\|_{H^{5}}^2&=2\sum_{i=1}^3\sum_{|\mu|\leq 5}\int_{\mathbb{T}\times\mathbb{R}}\bigg(\nabla^{\mu}\widetilde{g_i}\cdot\nabla^{\mu}(\Psi_{0,i}\nabla\widetilde{g_i})+\sum_{p=2}^{\infty}\nabla^{\mu}\widetilde{g_i}\cdot\nabla^{\mu}\big(\Psi_{p-1,i}(g^*,\cdots,g^*)\cdot\nabla\widetilde{g_i}\big)\\
&+\sum_{p=2}^{\infty}\sum_{j=1}^{p-1}\nabla^{\mu}\widetilde{g_i}\cdot\nabla^{\mu}\big(\nabla g_i^{\#}\cdot\Psi_{p-1,i}(\underbrace{g^*,\cdots,g^*}_{j-1},\widetilde{g},\underbrace{g^{\#},\cdots,g^{\#}}_{p-1-j})\big)\\
&+\sum_{p=2}^{\infty}\sum_{j=1}^{p}\nabla^{\mu}\widetilde{g_i}\cdot\nabla^{\mu}\Phi_{p,i}(\underbrace{g^*,\cdots,g^*}_{j-1},\widetilde{g},\underbrace{g^{\#},\cdots,g^{\#}}_{p-j})+\nabla^{\mu}\widetilde{g_i}\cdot\nabla^{\mu}\rho_i^{\#}\bigg)\,\mathrm{d}x\mathrm{d}y.
\end{split}
\end{equation}  Whenever $\nabla^{\mu}$ hits $\nabla\widetilde{g_i}$, we can integrate by parts to reduce to lower order terms; if $\nabla^{\mu}$ hits $\nabla g_i^{\#}$, we can control the corresponding term using the bound \eqref{bdfinal2} for $\|g^{\#}\|_{H^{6}}$. Notice that we never have $\nabla^{\mu}$ hitting $\nabla g$; therefore by \eqref{semiest2} we have
\[
\begin{split}\partial_t\|\widetilde{g}\|_{H^{5}}^2&\lesssim\sum_{p\geq 0}C^{p}t^{4(p+2)}\|\widetilde{g}\|_{H^{5}}^2\cdot(\|g^{\#}\|_{H^{6}}+\|g^*\|_{H^{5}})^{p}+\|\widetilde{g}\|_{H^{5}}\|\rho^{\#}\|_{H^{5}}\\
&\lesssim \sum_{p\geq 0}C^{p}t^{4(p+2)}\|\widetilde{g}\|_{H^5}^2\cdot \eta_0^{5}\eta_0^{-N\max(p-1,0)/2}+\|\widetilde{g}\|_{H^{5}}\|\rho^{\#}\|_{H^{5}}\\
&\lesssim\eta_0^{15}\|\widetilde{g}\|_{H^{5}}^2+e^{-\eta_0^{N'}}\|\widetilde{g}\|_{H^{5}}.
\end{split}
\] Since $\widetilde{g}(T_0)=0$, and $N'=30$, we can solve this to get that $\|\widetilde{g}\|_{H^{5}}\leq e^{-\eta_0^{N'}/2}\ll\eta_0^{-N}$, which completes the bootstrap.

Now a similar estimate allows us to control $\|\widetilde{g}\|_{H^{N}}$; in fact, suppose $\|\widetilde{g}\|_{H^{N}}\leq \eta_0^{-N/2}$ up to some time, then using \eqref{semiest2} similar to above, we have that
\allowdisplaybreaks
\begin{align*}
\partial_t\|\widetilde{g}\|_{H^{N}}^2&\lesssim\sum_{p\geq 0}C^{p}t^{4(p+2)}\|\widetilde{g}\|_{H^{N}}\cdot(\|g^{\#}\|_{H^{N+1}}+\|\widetilde{g}\|_{H^N})\cdot(\|g^{\#}\|_{H^{6}}+\|g^*\|_{H^{5}})^{p}+\|\widetilde{g}\|_{H^N}\|\rho^{\#}\|_{H^N}\\
&\lesssim\eta_0^{25}\|\widetilde{g}\|_{H^{N}}^2+e^{-\eta_0^{N'}/2}\|\widetilde{g}\|_{H^{N}}.
\end{align*} Again this closes the bootstrap for the $H^{N}$ norm, since $N'=30$, resulting in the final estimate
\begin{equation}\label{finalestg}
\sup_{t\geq T_0}\|\widetilde{g}(t)\|_{H^{N}}\leq e^{-\eta_0^{N'}/3}.
\end{equation}

The above shows that the solution $g$ exists at least up to time $t=T_1$, and moreover
\[\|g^*(T_1)-g^{\#}(T_1)\|_{H^{N}}\leq e^{-\eta_0^{N'}/3}.
\] Since $\|g^{\#}(T_1)-g^{(1)}(T_1)\|_{H^N}\leq\eta_0^{-N/2}$ by \eqref{order2}, we know $\|f^*(T_1)-f^{(1)}(T_1)\|_{H^N}\leq 2\eta_0^{-N/2}$. This in particular implies that
\[\bigg(\sum_{k_0/2\leq |k|\leq 2k_0}\int_{\eta_0/2\leq|\xi|\leq 2\eta_0}\big|\mathscr{F}f^*(T_1,k,\xi)-\mathscr{F}f^{(1)}(T_1,k,\xi)\big|^2\,\mathrm{d}\xi\bigg)^{1/2}\leq 2\eta_0^{-3N/2}.
\] Combining this with \eqref{funcbt4}, we obtain that
\[
\bigg(\sum_{k_0/2\leq |k|\leq 2k_0}\int_{\eta_0/2\leq|\xi|\leq 2\eta_0}\big|\widehat{f^*}(T_1,k,\xi)\big|^2\,\mathrm{d}\xi\bigg)^{1/2}\geq \frac{1}{4}\eta_0^{-(N-1)},
\] which implies that
\[
\big\|\langle \partial_z\rangle^{N_0}f^*(T_1)\big\|_{L^2}\geq\frac{1}{4}k_0^{N}\eta_0^{-(N-1)}\geq \eta_0^2
\] since $N_0=3N$. Using that $f^*=f-\underline{f}$ and the estimate (\ref{useass1}) for $\underline{f}$, this proves the first inequality in \eqref{newbd2}. The second inequality is a consequence of \eqref{finalestg} which controls $h^*(T_1)-h^{\#}(T_1)$, \eqref{farsmall3} which controls $h^{(1)}(T_1)$, \eqref{order2} which controls $h^{\#}(T_1)-h^{(1)}(T_1)$, and (\ref{useass3}) which controls $\underline{h}(T_1)$.

(2) When $t\leq T_0$, we will control $\widetilde{g}$ in a suitable Gevrey space. Define
\[M(t)=\sum_{i=1}^3\sum_{k}\int_{\mathbb{R}}e^{2\beta(t)\kappa(k,\xi)}|\widehat{\widetilde{g_i}}(t,k,\xi)|^2\,\mathrm{d}\xi,
\] where $\nu$ increases from $\beta(1)=3/2$ to $\beta(T_0)=7/4$, and $\partial_t\beta(t)\gtrsim T_0^{-1}$, then we compute
\[\partial_tM(t)\geq \frac{1}{Ct}\sum_{i=1}^3\sum_k\int_{\mathbb{R}}e^{2\beta(t)\kappa(k,\xi)}\cdot\kappa(k,\xi)|\widehat{\widetilde{g_i}}(t,k,\xi)|^2\,\mathrm{d}\xi+\sum_{i=1}^3K_i,
\] where 
\[|K_1|\leq\sqrt{M(t)}\cdot\|\rho^{\#}(t)\|_{\mathcal{G}_{\beta(t)}^*}\leq e^{-\eta_0^{N'}}\sqrt{M(t)},
\]and
\[|K_2|\leq M(t)\cdot\bigg(\sum_{p=2}^{\infty}\|\nabla g^{\#}\|_{\mathcal{G}_{\beta(t)}^*}\cdot C^pt^{4p}(\|g^{\#}\|_{\mathcal{G}_{\beta(t)}^*}+\|g^*\|_{\mathcal{G}_{\beta(t)}^*})^{p-2}+\sum_{p=1}^{\infty}C^pt^{4p}(\|g^{\#}\|_{\mathcal{G}_{\beta(t)}^*}+\|g^*\|_{\mathcal{G}_{\beta(t)}^*})^{p-1}\bigg).
\] For $K_3$ which comes from the transport term, we have
\[K_3=-2\Im\sum_{p=1}^{\infty}\sum_{i=1}^3\sum_{k,l}\int_{\mathbb{R}^2}e^{2\beta(t)\kappa(k,\xi)}\overline{\widehat{\widetilde{g_i}}(t,k,\xi)}\cdot\widehat{\widetilde{g_i}}(t,l,\eta)\cdot\big(l\cdot\widehat{\Psi_{(1)}}(t,k-l,\xi-\eta)+\eta\cdot\widehat{\Psi_{(2)}}(t,k-l,\xi-\eta)\big)\,\mathrm{d}\xi\mathrm{d}\eta,
\] where $\Psi_{(1)}$ and $\Psi_{(2)}$ are components of $\Psi_{p-1,i}(g^*,g^*,\cdots,g^*)$, and are real valued. Fixing $p$ and $i$, by adding smooth cutoffs, we can decompose this integral into two parts, $K'$ where 
\[|k-l|+|\xi-\eta|\geq \frac{1}{10}(|k|+|l|+|\xi|+|\eta|),
\] and $K''$ where 
\[|k-l|+|\xi-\eta|\leq \frac{1}{5}(|k|+|l|+|\xi|+|\eta|).
\] In the support of $K'$ we have 
\[\kappa(k,\xi)\leq \kappa(k-l,\xi-\eta)+0.99\kappa(l,\eta)+O(1),
\] thus by Young's inequality and \eqref{semiest} we can estimate
\begin{equation*}
\begin{split}|K'|&\lesssim\big\|e^{\beta(t)\kappa(k,\xi)}\widehat{\widetilde{g_i}}(t)\big\|_{L^2}\cdot\big\|e^{\beta(t)\kappa(m,\zeta)}(\widehat{\Psi_{(1)}}(t),\widehat{\Psi_{(2)}}(t))\big\|_{L^2}\cdot \big\|(|l|+|\eta|+1)e^{0.99\beta(t)\kappa(l,\eta)}\widehat{\widetilde{g_i}}(t)\big\|_{L^1}\\
&\qquad\lesssim M(t)\cdot C^pt^{4p}\|g\|_{\mathcal{G}_{\beta(t)}^{*}}^{p}.
\end{split}
\end{equation*} For $K''$, we use symmetry to replace the integrand by
\[\overline{\widehat{\widetilde{g_i}}(t,k,\xi)}\cdot\widehat{\widetilde{g_i}}(t,l,\eta)\cdot\widehat{\Psi_{(2)}}(t,k-l,\xi-\eta)\cdot(\eta\cdot e^{2\beta(t)\kappa(k,\xi)}-\xi\cdot e^{2\beta(t)\kappa(l,\eta)}),
\] and the other term involving $\Psi_{(1)}$ which can be treated in the similar way. Now using elementary inequalities we can show
\[|\eta\cdot e^{2\beta(t)\kappa(k,\xi)}-\xi\cdot e^{2\beta(t)\kappa(l,\eta)}|\lesssim|\xi-\eta|\cdot (1+\kappa(k,\xi))\cdot \max(e^{2\beta(t)\kappa(k,\xi)},e^{2\beta(t)\kappa(l,\eta)}).
\] Combining this with the inequality
\[\kappa(k,\xi)\leq \kappa(l,\eta)+0.99\kappa(k-l,\xi-\eta)+O(1),
\] and using Young's inequality and \eqref{semiest} again, we deduce that
\begin{equation*}|K'|\lesssim\bigg(\sum_k\int_{\mathbb{R}}e^{2\beta(t)\kappa(k,\xi)}\cdot\kappa(k,\xi)|\widehat{\widetilde{g_i}}(t,k,\xi)|^2\,\mathrm{d}\xi\bigg)\cdot\big\|e^{\beta(t)\kappa(m,\zeta)}(\widehat{\Psi_{(1)}}(t),\widehat{\Psi_{(2)}}(t))\big\|_{L^2},
\end{equation*} while the second factor in the above line is bounded by $t^{-2}\varepsilon_0$ (which is an upper bound for functions appearing in $\Psi_0$) if $p=1$, and by $C^pt^{4p}\|g\|_{\mathcal{G}_{\beta(t)}^{*}}^{p-1}$ if $p\geq 2$. Using also \eqref{bdfinal1} and bootstrap $\sqrt{M(t)}\leq e^{-\sigma^2k_0}$, we can bound 
\[|K''|\leq \frac{1}{10Ct}\sum_{i=1}^3\sum_k\int_{\mathbb{R}}e^{2\beta(t)\kappa(k,\xi)}\cdot\kappa(k,\xi)|\widehat{\widetilde{g_i}}(t,k,\xi)|^2\,\mathrm{d}\xi,
\] and therefore we conclude that
\[\partial_tM(t)\geq -Ct^4M(t)-e^{-\eta_0^{N'}}\sqrt{M(t)},\quad M(T_0)=0,
\] which easily closes the bootstrap. Since $\|\widetilde{g}(t)\|_{\mathcal{G}_{3/2}^*}\lesssim M(t)$ and $\|g^{\#}(1)\|_{\mathcal{G}_{2}^{*}}\leq e^{-\sigma^2k_0/4}$, this proves $\|g^*(1)\|_{\mathcal{G}_{2}^{*}}\lesssim e^{-\sigma^2k_0/6}$ which, when combines with the estimates (\ref{useass1})$\sim$(\ref{useass4}) which controls $\underline{g}$, proves the first inequality in \eqref{newbd}.

Finally, to prove the second inequality of (\ref{newbd}), we use the first equation in (\ref{eulernew}) to analyze $f$ (instead of $f^*$) directly. Define
\[S(t)=\int_{\mathbb{T}\times\mathbb{R}}(1+|v|)^2|f(t,z,v)|^2\,\mathrm{d}z\mathrm{d}v,
\] then by (\ref{eulernew}) we can compute
\[\partial_tS(t)=2\int_{\mathbb{T}\times\mathbb{R}}(1+|v|)^2f\cdot(Y_1\partial_zf+Y_2\partial_vf)\,\mathrm{d}z\mathrm{d}v,
\] where 
\[Y_1=(h+1)\partial_v\phi,\quad Y_2=-(h+1)\partial_z\phi-\theta;\quad \|\langle \nabla\rangle Y_j\|_{L^{\infty}}\lesssim \varepsilon_0t^{-2},
\]thanks to (\ref{useass3})$\sim$(\ref{useass4}) and the estimates for $g^*=g-\underline{g}$ obtained above. Integrating by parts, we get
\[|\partial_tS(t)|\lesssim\int_{\mathbb{T}\times\mathbb{R}}|f|^2\cdot\big(|(1+|v|)^2\partial_zY_1|+|\partial_v((1+|v|)^2Y_2)|\big)\,\mathrm{d}z\mathrm{d}v\lesssim \varepsilon_0t^{-2}S(t),
\] which implies that $S(1)\lesssim S(T_0)$. Since 
\[S(T_0)\lesssim\int_{\mathbb{T}\times\mathbb{R}}(1+|v|)^2|\underline{f}(T_0,z,v)|^2\,\mathrm{d}z\mathrm{d}v+\varepsilon_1^2\int_{\mathbb{T}\times\mathbb{R}}(1+|v|)^2|\varphi_p(k_0\sqrt{\sigma}v)|^2\,\mathrm{d}z\mathrm{d}v,
\] and the functional involving $\underline{f}$ can be controlled by applying the above arguments to $\underline{f}$,  we eventually get that $S(1)\lesssim\varepsilon_0^2$. By Cauchy-Schwartz, this implies the second inequality of (\ref{newbd}). \qed
\appendix
\section{Proofs of the auxiliary estimates}\label{appproof} In this appendix we prove the supporting estimates stated in Section \ref{linearize}.
\subsection{Kernels of the linearized system} In this subsection we prove Propositions \ref{linstep0} and \ref{systemrough}.
\begin{proof}[Proof of Proposition \ref{linstep0}] 
Recall that $\underline{h}=\underline{h}(t,v)$, and 
\begin{equation}
\underline{\Delta_t}=\partial_z^2+(\underline{h}+1)^2(\partial_v-t\partial_z)^2+(\underline{h}+1)(\partial_v\underline{h})(\partial_v-t\partial_z),
\end{equation}
 we know that $\underline{\Delta_t}$ maps the Fourier mode $e^{ik z}$ to itself; moreover by conjugating with $e^{itkv}$, we only need to study the operator 
 \begin{equation}
 Y:=-k^2+(\underline{h}+1)^2\partial_v^2+(\underline{h}+1)(\partial_v\underline{h})\partial_v,\end{equation}
  and if 
  \begin{equation}\label{1dkernel}
  \widehat{Y^{-1}F}(\xi)=\frac{-1}{\xi^2+k^2}\bigg(\widehat{F}(\xi)+\int_{\mathbb{R}}M'(t,k,\xi,\eta)\widehat{F}(\eta)\,\mathrm{d}\eta\bigg),
  \end{equation} 
  then \eqref{kernel1} holds with $M(t,k,\xi,\eta)=M'(t,k,\xi-tk,\eta-tk)$. Now write $Y=Z+W$, where \begin{equation}
  Z=\partial_v^2-k^2,\quad W=\underline{h}(\underline{h}+2)\partial_v^2+(\underline{h}+1)(\partial_v\underline{h})\partial_v:=h_1\partial_v^2+h_2\partial_v,
  \end{equation} 
  then
  \begin{equation}\label{neumann}
  Y^{-1}=\sum_{n=0}^{\infty}Z^{-1}(WZ^{-1})^n.
  \end{equation}
   The term $n=0$ corresponds to the first term in \eqref{1dkernel}; for the other terms, since\begin{equation}
   WZ^{-1}=h_1\frac{\partial_v^2}{\partial_v^2-k^2}+h_2\frac{\partial_v}{\partial_v^2-k^2},
   \end{equation}
    we know that
    \begin{equation}
    \begin{split}
    \mathscr{F}Z^{-1}(WZ^{-1})^nF(\xi)&=\frac{-1}{\xi^2+k^2}\int_{\mathbb{R}^n}\prod_{j=1}^n\bigg(\widehat{h_1}(\eta_{j-1}-\eta_j)\frac{\eta_j^2}{\eta_j^2+k^2}+\widehat{h_2}(\eta_{j-1}-\eta_j)\frac{\eta_j}{\eta_j^2+k^2}\bigg)\widehat{F}(\eta_n)\,\mathrm{d}\eta_1\cdots\mathrm{d}\eta_n\\
    &=\frac{-1}{\xi^2+k^2}\int_{\mathbb{R}}M_n(t,k,\xi,\eta)\widehat{F}(\eta)\,\mathrm{d}\eta,
\end{split}
\end{equation}
 where $\eta_0=\xi$, and 
 \begin{equation}
 |M_n(t,k,\xi,\eta)|\leq C^n\int_{\zeta_1+\cdots+\zeta_n=\xi-\eta}\prod_{j=1}^n(|\widehat{h_1}(\zeta_j)|+|\widehat{h_2}(\zeta_j)|)\,\mathrm{d}\zeta_1\cdots\mathrm{d}\zeta_{n-1}
 \end{equation}
  for some absolute constant $C$. By \eqref{useass3} we have that
  \begin{equation}
  |\widehat{h_j}(t,\zeta)|\lesssim\varepsilon_0 e^{-(C_0-3)|\zeta|},\quad j\in\{1,2\},\end{equation} therefore we have \begin{equation}|M_n(t,k,\xi,\eta)|\leq (C\varepsilon_0)^n\int_{\zeta_1+\cdots+\zeta_n=\xi-\eta}\exp\bigg(-(C_0-3)\sum_{j=1}^n|\zeta_j|\bigg)\,\mathrm{d}\zeta_1\cdots\mathrm{d}\zeta_{n-1}\lesssim (C\varepsilon_0)^ne^{-(C_0-4)|\xi|},\end{equation} which implies \eqref{kernel2}.

In physical space, similarly we know that if
\begin{equation}Y^{-1}F(v)=\int_{\mathbb{R}}K(t,k,v,w)F(w)\,\mathrm{d}w,\end{equation} then \eqref{kernel3} holds with the same $K=K(t,k,v,w)$. Now we have constants $c$, $c_1$ and $c_2$ such that 
\begin{equation}\frac{1}{\partial_v^2-k^2}f(v)=\int_{\mathbb{R}}c|k|^{-1}e^{-|k||v-w|}f(w)\,\mathrm{d}w,\end{equation}\begin{equation}\frac{\partial_v^2}{\partial_v^2-k^2}f(v)=f(v)+\int_{\mathbb{R}}c_1|k|e^{-|k||v-w|}f(w)\,\mathrm{d}w,\quad \frac{\partial_v}{\partial_v^2-k^2}f(v)=\int_{\mathbb{R}}c_2\mathrm{sgn}(v-w)\cdot e^{-|k||v-w|}f(w)\,\mathrm{d}w,\end{equation} so choosing $n=0$ in \eqref{neumann} yields a term $K_0:=c|k|^{-1}e^{-|k||v-w|}$ in the kernel $q$; for $n\geq 1$, we have \begin{equation}Z^{-1}(WZ^{-1})^nf(v)=\int_{\mathbb{R}}K_n(t,k,v,w)f(w)\,\mathrm{d}w,\end{equation} where \begin{multline}K_n(t,k,v,w)=\int_{\mathbb{R}^n}|k|^{-1}e^{-|k||v-w_1|}\prod_{j=1}^n\bigg(h_1(w_j)(\delta(w_j-w_{j+1})+|k|e^{-|k||w_j-w_{j+1}|})\\+h_2(w_j)\mathrm{sgn}(w_j-w_{j+1})e^{-k|w_j-w_{j+1}|}\bigg)\,\mathrm{d}w_1\cdots\mathrm{d}w_n,\end{multline} with $w_{n+1}=w$. Absorbing the possible $\delta$ factors, we only need to study terms of form \begin{equation}\int_{\mathbb{R}^m}|k|^{-1}e^{-|k||v-w_1|}\prod_{j=1}^{m+1}(h_1(w_j))^{q_j-1}\prod_{j=1}^m(|k|h_1(w_j)+h_2(w_j)\mathrm{sgn}(w_j-w_{j+1}))e^{-|k||w_j-w_{j+1}|}\,\mathrm{d}w_1\cdots\mathrm{d}w_m,\end{equation} where $w_{m+1}=w$, $q_j\geq 1$ and $q_1+\cdots+q_{m+1}=n+1$. In particular\begin{equation}|K_n|\lesssim C^n(\|h_1\|_{L^{\infty}}+\|h_2\|_{L^{\infty}})^n|k|^{-2}\int_{z_1+\cdots +z_{m+1}=v-w}\prod_{j=1}^m(|k|e^{-|k||z_j|})\mathrm{d}z_1\cdots\mathrm{d}z_m\lesssim (C\varepsilon_0)^n|k|^{-1}e^{-|k||v-w|/2},\end{equation} and similar estimates hold for $\partial_{v,w}K_n$ and $\partial_{w}^2K_n$, except for the term $\partial_w^2K_n$ with $m=0$, which equals
\begin{equation}\partial_w^2\big(|k|^{-1}e^{-|k||v-w|}(h_1(w))^n\big)=O(\varepsilon_0^n)\delta(v-w)+O(\varepsilon_0^n)|k|e^{-|k||v-w|/2}.\end{equation} Therefore all $K_n$ satisfy the estimates \eqref{kernel4}$\sim$\eqref{kernel6}, and so does\[K=\sum_{n=0}^{\infty}K_n.\qedhere\]
\end{proof}
 \begin{proof}[Proof of Proposition \ref{systemrough}] We first prove part (1). Using \eqref{eulernewl} and \eqref{eulersupl}, we can compute
 \begin{align}
\nonumber\partial_tf'&=-\underline{\theta}\cdot\partial_v f'-(\underline{h}+1)\nabla^{\perp}\underline{\phi}\cdot\nabla f'-(\underline{h}+1)\big[\partial_v\underline{f}\cdot\partial_z\mathbb{P}_{\neq 0}\underline{\Delta_t}^{-1}f'-\partial_z\underline{f}\cdot\partial_v\mathbb{P}_{\neq 0}\underline{\Delta_t}^{-1}f'\big]\\
\nonumber&+(\underline{h}+1)(\partial_v\underline{f}\cdot\partial_z-\partial_z\underline{f}\cdot\partial_v)\mathbb{P}_{\neq 0}\underline{\Delta_t}^{-1}\big[2(\underline{h}+1)h'(\partial_v-t\partial_z)^2+h'\partial_v\underline{h}(\partial_v-t\partial_z)\big]\underline{\phi}\\
\nonumber&+(\underline{h}+1)(\partial_v\underline{f}\cdot\partial_z-\partial_z\underline{f}\cdot\partial_v)\mathbb{P}_{\neq 0}\underline{\Delta_t}^{-1}(\underline{h}+1)\partial_vh'(\partial_v-t\partial_z)\underline{\phi}\\
\label{expanded1}&-\nabla\underline{f}\cdot\nabla^{\perp}\underline{\phi}\cdot h'-\partial_v\underline{f}\cdot\theta',\\
\label{expanded2}\partial_th'&=-\frac{\mathbb{P}_0f'+h'}{t}-\underline{\theta}\cdot\partial_vh'-\partial_v\underline{h}\cdot\theta',\\
\nonumber\partial_t\theta'&=\frac{-2\theta'}{t}-\partial_v(\underline{\theta}\theta')+\frac{1}{t}\partial_z\underline{\phi}\cdot f'+\frac{1}{t}\underline{f}\cdot\partial_z\mathbb{P}_{\neq 0}\Delta_t^{-1}f'\\
\label{expanded3}&-\frac{1}{t}\underline{f}\cdot\partial_z\underline{\Delta_t}^{-1}\big(2(\underline{h}+1)h'(\partial_v-t\partial_z)^2+(h'\partial_v\underline{h}+(\underline{h}+1)\partial_vh')(\partial_v-t\partial_z)\big)\underline{\phi}.
\end{align} Using \eqref{useass1}$\sim$\eqref{useass4} to control the functions $(\underline{f},\underline{h},\underline{\theta},\underline{\phi})$ and their combinations, and using \eqref{kernel1}$\sim$\eqref{kernel2} to control the kernel of the operator $\underline{\Delta_t}^{-1}$, we first get the simple bounds\[|q_{13}(t,k,0,\xi,\eta)|\lesssim\varepsilon_0e^{-(C_0/2)(|k|+|\xi-\eta|)},\quad q_{21}(t,0,l,\xi,\eta)\equiv0,\quad |q_{22}(t,0,0,\xi,\eta)|\lesssim\varepsilon_0\frac{|\eta|+1}{t^2}e^{-(C_0/2)|\xi-\eta|},\] and \[|q_{33}(t,0,0,\xi,\eta)|\lesssim \varepsilon_0\frac{|\eta|+1}{t^2}e^{-(C_0/2)|\xi-\eta|},\quad |q_{23}(t,0,0,\xi,\eta)|\lesssim\bigg(\varepsilon_0^2+\frac{\varepsilon_0}{t}\bigg)e^{-(C_0/2)|\xi-\eta|},\] where the last inequality follows from decomposing \[\underline{h}=\mathbb{P}_0f_\infty+(\underline{h}-\mathbb{P}_0f_\infty)\] and using \eqref{useass2} and \eqref{useass3}, which proves \eqref{symbolbd3}$\sim$\eqref{symbolbd4} and \eqref{symbolbd7}.

Then, using \eqref{expanded1} we have \begin{align}|q_{11}(t,k,l,\xi,\eta)|&\lesssim\varepsilon_0e^{-(C_0/2)(|k-l|+|\xi-\eta|)}\frac{|l|+|\eta|}{t^2}+\varepsilon_0e^{-(C_0/2)|k-l|}\bigg(e^{-(C_0/2)|\xi-\eta|}\frac{\mathbf{1}_{l\neq 0}}{(\eta-tl)^2+l^2}(|l|+|k-l|\cdot|\eta|)\nonumber\\
\label{q11est}&+\int_{\mathbb{R}}e^{-(C_0/2)|\xi-\zeta|}\frac{\mathbf{1}_{l\neq 0}}{(\zeta-tl)^2+l^2}(|l|+|k-l|\cdot|\zeta|)|m(t,l,\zeta,\eta)|\,\mathrm{d}\zeta\bigg),\end{align} and \begin{align}|q_{12}(t,k,0,\xi,\eta)|&\lesssim\varepsilon_0^2\sum_{q}e^{-(C_0/2)(|k-q|+|q|)}\int_{\mathbb{R}}e^{-(C_0/2)|\xi-\zeta|}\frac{\mathbf{1}_{q\neq 0}}{(\zeta-tq)^2+q^2}(|q|+|k-q|\cdot|\zeta|)\bigg(1+\frac{1+|\eta|}{t}\bigg)\,\mathrm{d}\zeta\nonumber\\
&\times\bigg[e^{-(C_0/2)|\zeta-\eta|}+\int_{\mathbb{R}}e^{-(C_0/2)|\beta-\eta|}|m(t,q,\zeta,\nu)|\,\mathrm{d}\nu\bigg]+\varepsilon_0^2\frac{e^{-(C_0/2)(|k|+|\xi-\eta|)}}{t^2},\end{align}which imply \eqref{symbolbd1}$\sim$\eqref{symbolbd2} by elementary computation. Similarly, using \eqref{expanded3} we have
\begin{multline}
|q_{31}(t,0,l,\xi,\eta)|\lesssim\varepsilon_0 e^{-(C_0/2)|l|}\bigg[e^{-(C_0/2)|\xi-\eta|}\bigg(\frac{1}{t^3}+\frac{|l|}{t((\eta-tl)^2+l^2)}\bigg)\\+\int_{\mathbb{R}}e^{-(C_0/2)|\xi-\zeta|}\frac{|l|}{t((\zeta-tl)^2+l^2)}\cdot|m(t,l,\zeta,\eta)|\,\mathrm{d}\zeta\bigg],
\end{multline} and 
\begin{multline}
|q_{32}(t,0,0,\xi,\eta)|\lesssim\varepsilon_0^2\sum_{q}e^{-(C_0/2)|q|}\int_{\mathbb{R}}e^{-(C_0/2)|\xi-\zeta|}\frac{|q|}{t((\zeta-tq)^2+q^2)}\bigg(1+\frac{1+|\eta|}{t}\bigg)\\
\times\bigg(e^{-(C_0/2)|\zeta-\eta|}+\int_{\mathbb{R}}e^{-(C_0/2)|\nu-\eta|}|m(t,q,\zeta,\nu)|\,\mathrm{d}\nu\bigg),
\end{multline} which imply \eqref{symbolbd5}$\sim$\eqref{symbolbd6} by elementary computation. 

For part (2), we first study $q_{11}^R$. Using \eqref{expanded1} we can compute
\begin{align}|q_{11}^R(t,k,l,\xi,\eta)|&\lesssim\varepsilon_0e^{-(C_0/2)|k-l|}\bigg(e^{-(C_0/2)|\xi-\eta|}\frac{\mathbf{1}_{l\neq 0}}{(\eta-tl)^2+l^2}(|l|+|k-l|\cdot|\eta|)\nonumber\\
&+\int_{\mathbb{R}}e^{-(C_0/2)|\xi-\zeta|}\frac{\mathbf{1}_{l\neq 0}}{(\zeta-tl)^2+l^2}(|l|+|k-l|\cdot|\zeta|)|m(t,l,\zeta,\eta)|\,\mathrm{d}\zeta\bigg),\end{align}
where the right hand side is simply \eqref{q11est} without the first term. This then implies \eqref{reactonly}, in the same way as \eqref{symbolbd1}.

Next, the term $q_{11}^T$ already satisfies (\ref{defq11''}); we can moreover decompose $q_{11}^R$ by decomposing $\underline{f}$ using (\ref{useass1})$\sim$(\ref{useass2}), and decompose $\underline{\Delta_t}$ using (\ref{kernel1}). This gives rise to the term $q_{11}'$, plus an error term that satisfies (\ref{symbolbd1}) with an extra $\varepsilon_0$ factor in front. This provides the required decomposition $q_{11}=q_{11}'+q_{11}''$.
 \end{proof}
 \subsection{Energy estimates} In this subsection we prove Propositions \ref{energy1}$\sim$\ref{lastone}. First let us collect some auxiliary estimate for the symbols used in these estimates.
 
 For $\beta\geq 1$, define
 \begin{equation}\label{defakweigen}A_{k,\beta}(t,\xi)=\frac{1}{(\xi-kt)^2+\beta^2k^2},\,k\neq 0,\quad A_{0,\beta}(t,\xi)=\frac{1}{\xi^2+\beta^2};\quad A_{*,\beta}(t,\xi)=\sum_{k}e^{-2|k|}A_{k}(t,\xi),\end{equation}
 then we have $(A_{k},A_*)=(A_{k,\varepsilon_0^{-1/2}},A_{*,\varepsilon_0^{-1/2}})$, and $(\widetilde{A_k},\widetilde{A_*})=(A_{k,Q},A_{*,Q})$, see Propositions \ref{energy1} and \ref{lastone} for definitions.
 \begin{proposition} Let $\beta\geq 1$. We have the following estimates.
 \begin{align}
\label{weight1}|\partial_tA_{k,\beta}(t,\xi)|\lesssim \beta^{-1}A_k(t,\xi),\quad A_k(t,\xi)&\lesssim e^{|\xi-\eta|}A_k(t,\eta),\quad A_*(t,\xi)\lesssim e^{|\xi-\eta|}A_*(t,\eta),\\
\label{simple0}\frac{\mathbf{1}_{l\neq 0}}{(\xi-tl)^2+l^2}&\lesssim \beta^{-2}A_{l,\beta}(t,\xi),\\
\label{simple}e^{-|k-l|}\mathbf{1}_{k\neq l}\sqrt{A_{k,\beta}(t,\xi)A_{l,\beta}(t,\xi)}&\lesssim\begin{dcases}\min(|\xi|^{-2},\beta^{-2}),&|\xi|\gg t\max(|k|,|l|);\\
\beta^{-1}(t\max(|k|,|l|))^{-1},&|\xi|\lesssim t\max(|k|,|l|),\end{dcases}\\
\label{simple2}e^{-|k-l|}\sqrt{\frac{A_{k,\beta}(t,\xi)}{A_{l,\beta}(t,\xi)}}&\lesssim\begin{dcases}1,&|\xi-tk|\gtrsim t,\\
1+\beta^{-1}t\cdot(\max(|k|,|l|))^{-1},&|\xi-tk|\ll t,\end{dcases}\\
\label{simple3}e^{-2|k|}\lesssim\sqrt{\frac{A_{*,\beta}(t,\xi)}{A_{k,\beta}(t,\xi)}}&\lesssim e^{|k|}\big(\beta^{-1}te^{-|\xi|/t}+1\big),\\
\label{extra}e^{-|k-l|}\mathbf{1}_{k\neq l}|\xi|\cdot\sqrt{\frac{A_{k,\beta}(t,\xi)}{(\xi-tl)^2+l^2}}&\lesssim \frac{|k|}{\sqrt{(\xi-tk)^2+k^2}}+\frac{|l|}{\sqrt{(\xi-tl)^2+l^2}}:=F_k(t,\xi)+F_l(t,\xi),
\end{align}
 \end{proposition}
 \begin{proof} These follow from elementary computations. For example, to prove (\ref{simple}) for $kl\neq 0$ and $k\neq l$, we compute
 \[\sqrt{A_{k,\beta}(t,\xi)A_{l,\beta}(t,\xi)}=\frac{1}{\sqrt{((\xi-kt)^2+\beta^2k^2)((\xi-lt)^2+\beta^2l^2)}}.
 \] If $|\xi|\gg t\max(|k|,|l|)$ then we can estimate
 \[\frac{1}{\sqrt{((\xi-kt)^2+\beta^2k^2)((\xi-lt)^2+\beta^2l^2)}}\lesssim\min(|\xi|^{-2},\beta^{-2}|kl|^{-1});
 \] if $|\xi|\lesssim t\max(|k|,|l|)$, since $k\neq l$, we have $\max(|\xi-kt|,|\xi-lt|)\gtrsim t$, so
 \[\frac{1}{\sqrt{((\xi-kt)^2+\beta^2k^2)((\xi-lt)^2+\beta^2l^2)}}\lesssim\frac{1}{\beta t\min(|k|,|l|)}\lesssim    e^{|k-l|}\frac{1}{\beta t\max(|k|,|l|)}.
 \] The case $kl=0$ as well as the other inequalities are proved by similar arguments.
 \end{proof}
 \begin{proof}[Proof of Proposition \ref{energy1}] By symmetry, we only need to consider the interval $[T_2,T_0]$. We compute\begin{equation}\partial_tM_0(t)=\sum_{i=1}^3I_{i}(t)+\sum_{i,j=1}^3J_{ij}(t)+K(t),\end{equation} where 
 \begin{align*}
I_{1}(t)&=\sum_{k}\int_{\mathbb{R}}\partial_tA_k(t,\xi)\cdot e^{\tau\lambda_0(t)(|k- k_*|+|\xi- \eta_*|)}\big|\widehat{g_1'}(t,k,\xi)\big|^2\,\mathrm{d}\xi+\sum_{j=2}^3\int_{\mathbb{R}}\partial_tA_*(t,\xi)\cdot e^{\tau\lambda_0(t)(|k_*|+|\xi- \eta_*|)}\big|\widehat{g_j'}(t,0,\xi)\big|^2\,\mathrm{d}\xi,\\
I_2(t)&=\partial_t\lambda_0(t)\cdot\bigg(\sum_{k}\int_{\mathbb{R}}A_k(t,\xi)(|k- k_*|+|\xi- \eta_*|)e^{\tau\lambda_0(t)(|k- k_*|+|\xi- \eta_*|)}\big|\widehat{g_1'}(t,k,\xi)\big|^2\,\mathrm{d}\xi\\
&\qquad\quad+\sum_{j=2}^3\int_{\mathbb{R}}A_*(t,\xi)(|k_*|+|\xi- \eta_*|)e^{\tau\lambda_0(t)(|k_*|+|\xi- \eta_*|)}\big|\widehat{g_j'}(t,0,\xi)\big|^2\,\mathrm{d}\xi\bigg),\\
I_3(t)&=-2t^{-1}\Re\int_{\mathbb{R}}A_*(t,\xi)e^{\tau\lambda_0(t)(|k_*|+|\xi- \eta_*|)}\big(\big|\widehat{g_2'}(t,0,\xi)\big|^2+2\big|\widehat{g_3'}(t,0,\xi)\big|^2+\overline {\widehat{g_2'}(t,0,\xi)}\cdot \widehat{g_1'}(t,0,\xi)\big)\,\mathrm{d}\xi,\\
J_{1j}&=2\Re\sum_{k,l}\int_{\mathbb{R}^2}A_k(t,\xi)q_{1j}(t,k,l,\xi,\eta)e^{\tau\lambda_0(t)(|k- k_*|+|\xi- \eta_*|)}\overline{\widehat{g_1'}(t,k,\xi)}\widehat{g_j'}(t,l,\eta)\,\mathrm{d}\xi\mathrm{d}\eta\,\,\,\,(1\leq j\leq 3),\\
J_{ij}&=2\Re\sum_{l}\int_{\mathbb{R}^2}A_*(t,\xi)q_{ij}(t,0,l,\xi,\eta)e^{\tau\lambda_0(t)(|k_*|+|\xi- \eta_*|)}\overline{\widehat{g_i'}(t,0,\xi)}\widehat{g_j'}(t,l,\eta)\,\mathrm{d}\xi\mathrm{d}\eta\,\,\,\,(2\leq i\leq 3,1\leq j\leq 3),\\
K(t)&=2\Re\sum_{k}\int_{\mathbb{R}}A_k(t,\xi)e^{\tau\lambda_0(t)(|k- k_*|+|\xi- \eta_*|)}\overline{\widehat{g_1'}(t,k,\xi)}\widehat{\rho_1'}(t,k,\xi)\mathrm{d}\xi\\
&\qquad\quad+2\Re\sum_{j=2}^3\int_{\mathbb{R}}A_*(t,\xi)e^{\tau\lambda_0(t)(|k_*|+|\xi- \eta_*|)}\cdot\overline{\widehat{g_j'}(t,0,\xi)}\widehat{\rho_j'}(t,0,\xi)\mathrm{d}\xi.
\end{align*} Now by the first inequality in \eqref{weight1} we get $|I_1(t)|\lesssim  \varepsilon_0^{1/2}M_0(t)$, and moreover $I_2(t)\geq 0$. The term $K(t)$ is estimated by Cauchy-Schwartz,
\begin{equation*}
\begin{split}K(t)&\lesssim\bigg(\sum_{k}\int_{\mathbb{R}}A_k(t,\xi)e^{\tau\lambda_0(t)(|k- k_*|+|\xi- \eta_*|)}\big|\widehat{g_1'}(t,k,\xi)\big|^2\mathrm{d}\xi+\sum_{j=2}^3\int_{\mathbb{R}}A_*(t,\xi)e^{\tau\lambda_0(t)(|k_*|+|\xi- \eta_*|)}\cdot\big|\widehat{g_j'}(t,0,\xi)\big|^2\mathrm{d}\xi\bigg)^{1/2}\\
&\times \bigg(\sum_{k}\int_{\mathbb{R}}A_k(t,\xi)e^{\tau\lambda_0(t)(|k- k_*|+|\xi- \eta_*|)}\big|\widehat{\rho_1'}(t,k,\xi)\big|^2\mathrm{d}\xi+\sum_{j=2}^3\int_{\mathbb{R}}A_*(t,\xi)e^{\tau\lambda_0(t)(|k_*|+|\xi- \eta_*|)}\cdot\big|\widehat{\rho_j'}(t,0,\xi)\big|^2\mathrm{d}\xi\bigg)^{1/2}\\
&\lesssim\sqrt{M_0(t)}\cdot Z(t),
\end{split}
\end{equation*} and using (\ref{simple3}) and Cauchy-Schwartz we can bound $|I_3(t)|\lesssim (\varepsilon_0^{1/2}+t^{-1})M_0(t)$.

We now proceed to analyze $J_{ij}(t)$. Using (\ref{symbolbd3})$\sim$(\ref{symbolbd7}) and Cauchy-Schwartz, we easily see that
\begin{equation}\sum_{(i,j)\neq(1,1),(1,2)}|J_{ij}(t)|\lesssim\varepsilon_0\bigg(1+\frac{|\eta_*|}{t^2}\bigg)M_0(t)+\varepsilon_0^{1/2}I_2(t).\end{equation}
For $J_{11}$ we have
\begin{equation*}\begin{aligned}A_k(t,\xi)|q_{11}(t,k,l,\xi,\eta)|&\lesssim\varepsilon_0e^{-(C_0/4)(|k-l|+|\xi-\eta|)}A_k(t,\xi)\bigg(\frac{|k|+|\xi|+1}{t^2}+\varepsilon_0^{-1}A_l(t,\xi)(|k|+1+\mathbf{1}_{k\neq l}\cdot|\xi|)\bigg)\\
&\lesssim\varepsilon_0e^{-(C_0/5)(|k-l|+|\xi-\eta|)}\sqrt{A_k(t,\xi)A_l(t,\eta)}\sqrt{\frac{A_k(t,\xi)}{A_l(t,\xi)}}\frac{|k|+|\xi|+1}{t^2}\\
&\qquad+e^{-(C_0/5)(|k-l|+|\xi-\eta|)}\sqrt{A_k(t,\xi)A_l(t,\eta)}\cdot \big(e^{-|k-l|}(|k|+1)\sqrt{A_k(t,\xi)A_l(t,\xi)}\\
&\qquad\qquad+\mathbf{1}_{k\neq l}\cdot|\xi|e^{-|k-l|}\sqrt{A_k(t,\xi)A_l(t,\xi)}\big).\end{aligned}\end{equation*} Using \eqref{simple}$\sim$\eqref{simple2}, we can bound 
\begin{equation*}\frac{A_k(t,\xi)q_{11}(t,k,l,\xi,\eta)}{e^{-(C_0/5)(|k-l|+|\xi-\eta|)}\sqrt{A_k(t,\xi)A_l(t,\eta)}}\lesssim\varepsilon_0^{1/2}+\varepsilon_0\bigg(\frac{|k_*|+|\eta_*|}{t^2}+\frac{|k-k_*|+|\xi-\eta_*|}{t^2}\bigg).\end{equation*} By Cauchy-Schwartz, this gives\[|J_{11}(t)|\lesssim \bigg(\varepsilon_0^{1/2}+\varepsilon_0\frac{|k_*|+|\eta_*|}{t^2}\bigg)M_0(t)+\varepsilon_0^{1/2}I_2(t).\] For $J_{12}$, we have that
\begin{equation*}\frac{A_k(t,\xi)|q_{12}(t,k,0,\xi,\eta)|}{e^{-(C_0/5)(|k|+|\xi-\eta|)}\sqrt{A_k(t,\xi)A_*(t,\xi)}}\lesssim \varepsilon_0^2t^{-1}+\varepsilon_0^2\varepsilon_0^{-1}\sum_{q}e^{-(C_0/4)|q|}(1+\mathbf{1}_{q\neq k}\cdot|\xi|)\sqrt{A_k(t,\xi)A_l(t,\xi)}\lesssim\varepsilon_0^{3/2}\end{equation*} using \eqref{simple}, which gives $|J_{12}(t)|\lesssim\varepsilon_0^{3/2}M_0(t)$.

Summing up, we get that\[\partial_tM_0(t)\geq -C\bigg(\varepsilon_0^{1/2}+\frac{1}{t}+\varepsilon_0\frac{|k_*|+|\eta_*|}{t^2}\bigg)M_0(t)-C\sqrt{M_0(t)}\cdot Z(t),\] which implies \eqref{diffineq2}.
\end{proof}
\begin{proof}[Proof of Proposition \ref{gevrey2}] We compute, as in the proof of Proposition \ref{energy1}, that
\begin{equation}\partial_tM_1(t)=\sum_{i=1}^2I_{i}(t)+\sum_{i,j=1}^3J_{ij}(t)+K(t),\end{equation} where 
\begin{align*}
I_{1}(t)&=\sum_{k}\int_{\mathbb{R}}\partial_tA_k(t,\xi)\cdot \big|\widehat{g_1'}(t,k,\xi)\big|^2\,\mathrm{d}\xi+\sum_{j=2}^3\int_{\mathbb{R}}\partial_tA_*(t,\xi)\cdot \big|\widehat{g_j'}(t,0,\xi)\big|^2\,\mathrm{d}\xi,\\I_2(t)&=-2t^{-1}\Re\int_{\mathbb{R}}A_*(t,\xi)\big(\big|\widehat{g_2'}(t,0,\xi)\big|^2+2\big|\widehat{g_3'}(t,0,\xi)\big|^2+\overline {\widehat{g_2'}(t,0,\xi)}\cdot \widehat{g_1'}(t,0,\xi)\big)\,\mathrm{d}\xi,\\
J_{1j}&=2\Re\sum_{k,l}\int_{\mathbb{R}^2}A_k(t,\xi)q_{1j}(t,k,l,\xi,\eta)\overline{\widehat{g_1'}(t,k,\xi)}\widehat{g_j'}(t,l,\eta)\,\mathrm{d}\xi\mathrm{d}\eta,\,\,\,\,(1\leq j\leq 3),\\
J_{ij}&=2\Re\sum_{l}\int_{\mathbb{R}^2}A_*(t,\xi)q_{ij}(t,0,l,\xi,\eta)\overline{\widehat{g_i'}(t,0,\xi)}\widehat{g_j'}(t,l,\eta)\,\mathrm{d}\xi\mathrm{d}\eta,\,\,\,\,(2\leq i\leq 3,1\leq j\leq 3),\\ 
K(t)&=2\Re\sum_{k}\int_{\mathbb{R}}A_k(t,\xi)\overline{\widehat{g_1'}(t,k,\xi)}\widehat{\rho_1'}(t,k,\xi)\mathrm{d}\xi+2\Re\sum_{j=2}^3\int_{\mathbb{R}}A_*(t,\xi)\cdot\overline{\widehat{g_j'}(t,0,\xi)}\widehat{\rho_j'}(t,0,\xi)\mathrm{d}\xi.\end{align*} Moreover we will decompose $q_{11}=q_{11}^T+q_{11}^R$, see Proposition \ref{systemrough}, and define correspondingly the terms $J_{11}^T$ and $J_{11}^R$. 

Now the terms $I_1$, $I_2$, $J_{ij} (i\neq j)$, and $J_{11}^R$ can be estimated in the same way as in the proof of Proposition \ref{energy1}, using (\ref{symbolbd1})$\sim$(\ref{symbolbd7}), to obtain that \[|I_1|+|I_2|+\sum_{i\neq j}|J_{ij}|+|J_{11}^R|\lesssim(\varepsilon_0^{1/2}+t^{-1})M_1(t).\] The term $K(t)$ is estimated by Cauchy-Schwartz,
\begin{equation*}
|K(t)|\lesssim\sqrt{M_1(t)}\bigg(\sum_{k}\int_{\mathbb{R}}A_k(t,\xi)\big|\widehat{\rho_1'}(t,k,\xi)\big|^2\,\mathrm{d}\xi+\sum_{j=2}^3\int_{\mathbb{R}}A_*(t,\xi)\big|\widehat{\rho_j'}(t,0,\xi)\big|^2\,\mathrm{d}\xi\bigg)^{1/2}
\lesssim \sqrt{M_1(t)}Z(t).
\end{equation*}

It then suffices to estimate the terms $J_{11}^T$, $J_{22}$ and $J_{33}$. We only consider the first one, as the other two are similar. Using the definition (\ref{decompose}) we can write
\begin{equation}\label{transform}
J_{11}^R(t)=t^{-2}\varepsilon_0\Re\sum_{k,l}\int_{\mathbb{R}^2}A_k(t,\xi)(r_1(k-l,\xi-\eta)\cdot l+r_2(k-l,\xi-\eta)\cdot\eta)\overline{\widehat{g_1'}(t,k,\xi)}\widehat{g_1'}(t,l,\eta)\,\mathrm{d}\xi\mathrm{d}\eta,
\end{equation} where $r_j$ satisfies that $|r_j(m,\zeta)|\lesssim e^{-10(|m|+|\zeta|)}$, and $r_j(-m,-\zeta)=-\overline{r_j(m,\zeta)}$. Using symmetry we can bound this by
\begin{align*}
|\widetilde{J_{11}}(t)|&\lesssim \sum_{k,l}\int_{\mathbb{R}^2}\big( |r_1(k-l,\xi-\eta)|\cdot |q_{11}^k(t,k,l,\xi,\eta)|+|r_2(k-l,\xi-\eta)|\cdot |q_{11}^\xi(t,k,l,\xi,\eta)|\big)\\&\qquad\quad\times\sqrt{A_k(t,\xi)A_l(t,\eta)}|\widehat{g_1'}(t,k,\xi)|\cdot|\widehat{g_1'}(t,l,\eta)|\,\mathrm{d}\xi\mathrm{d}\eta,
\end{align*} where
\begin{equation}
q_{11}^k(t,k,l,\xi,\eta)=\frac{\varepsilon_0e^{-5(|k-l|+|\xi-\eta|)}}{t^2}\cdot\frac{l\cdot A_k(t,\xi)-k\cdot A_l(t,\eta)}{\sqrt{A_k(t,\xi)A_l(t,\eta)}},
\end{equation}
and
\begin{equation}
q_{11}^{\xi}(t,k,l,\xi,\eta)=\frac{\varepsilon_0e^{-5(|k-l|+|\xi-\eta|)}}{t^2}\cdot\frac{\eta\cdot A_k(t,\xi)-\xi\cdot A_l(t,\eta)}{\sqrt{A_k(t,\xi)A_l(t,\eta)}}.
\end{equation} To obtain the bound $|\widetilde{J_{11}}|\lesssim \varepsilon_0^{1/2}M_1(t)$, by Cauchy-Schwartz, it then suffices to prove that
\begin{equation}\label{newqest}|q_{11}^k|+|q_{11}^{\xi}|\lesssim \varepsilon_0^{1/2}.
\end{equation}

We will only prove \eqref{newqest} for the term $q_{11}^{\xi}$, since the proof for the term $q_{11}^k$ is similar (and easier). Decompose
\[
\eta\cdot A_k(t,\xi)-\xi\cdot A_l(t,\eta)=(\eta-\xi)\cdot A_k(t,\xi)+\xi\cdot(A_k(\xi)-A_l(\eta)),
\] and denote the contribution of the two terms above to $q_{11}^{\xi}$ by $q_{11}^{(j)}$ with $1\leq j\leq 2$. Then we have $|q_{11}^{(1)}|\lesssim\varepsilon_0$ using \eqref{simple0}$\sim$\eqref{simple3}. For $q_{11}^{(2)}$ we may assume $|k|\sim|l|$, and compute
\[
\begin{split}
|q_{11}^{(2)}|&\lesssim\frac{\varepsilon_0|\xi|}{t^2}e^{-5(|k-l|+|\xi-\eta|)}\frac{|A_k(\xi)-A_l(\eta)|}{\sqrt{A_k(\xi)A_l(\eta)}}\\
&\lesssim\frac{\varepsilon_0|\xi|}{t^2}e^{-4(|k-l|+|\xi-\eta|)}\sqrt{A_k(\xi)A_l(\eta)}\cdot\big|\varepsilon_0^{-1}k^2+(\xi-tk)^2-\varepsilon_0^{-1}l^2-(\eta-tl)^2\big|.
\end{split}
\] Let $X=|\xi-tk|$ and $Y=|\eta-tl|$, if $|\xi|\gg t|k|$, then $\sqrt{A_k(\xi)A_l(\eta)}\lesssim|\xi|^{-2}$ (notice that $|k|\sim|l|$), which implies $|q_{11}^{(2)}|\lesssim\varepsilon_0$; if $|\xi|\lesssim t|k|$, then
\[
|q_{11}^{(2)}|\lesssim\frac{\varepsilon_0|k|}{t}\cdot\frac{\varepsilon_0^{-1}(|k|+|l|)+t(|X|+|Y|)}{\sqrt{(X^2+\varepsilon_0^{-1}k^2)(Y^2+\varepsilon_0^{-1}l^2)}}\lesssim\varepsilon_0.
\] This completes the proof of \eqref{monotinicity2}.
\end{proof}
\begin{proof}[Proof of Proposition \ref{lastone}] Recall $M_0(t)$ defined in \eqref{diffineq}, with $k_*=\eta_*=0$. Using the differential inequality \eqref{diffineq2}, the support condition for $\widehat{\rho'}$, and the fact that $g'(T_0)=0$, we obtain, for $t\in[T_2,T_1]$, that
\begin{equation}\label{movexismall}\sum_{j=1}^3\sum_k\int_{\mathbb{R}}e^{(|k|+|\eta|)/2}\big|\widehat{g_j'}(t,k,\xi)\big|^2\,\mathrm{d}\xi\lesssim t^2M_0(t)\lesssim e^{CT_0^{21/20}}(Z(t))^2.
\end{equation}In particular, this implies that
\begin{equation}\label{errorsmallnew}
\sum_{j=1}^3\sum_k\int_{\mathbb{R}}\big(\mathbf{1}_{|k|\geq D T_0^{21/20}}+\mathbf{1}_{|\xi|\geq DT_0^{21/20}}\big)e^{(|k|+|\eta|)/4}\big|\widehat{g_j'}(t,k,\xi)\big|^2\,\mathrm{d}\xi\lesssim e^{-DT_0^{21/20}/4}(Z(t))^2.
\end{equation} Therefore, on the right hand side of \eqref{simpsys}, we may freely insert cutoffs of form
\[\mathbf{1}_{|k|\lesssim DT_0^{21/20}}\cdot \mathbf{1}_{|l|\lesssim DT_0^{21/20}}\cdot\psi\bigg(\frac{\xi}{DT_0^{21/20}}\bigg)\psi\bigg(\frac{\eta}{DT_0^{21/20}}\bigg)
\] at the expense of introducing error terms bounded by $e^{-DT_0^{21/20}/8}Z(t)$, which can always be absorbed by the $\rho'$ term in (\ref{inhomo}), and will thus be omitted below.

Now we define 
\[
\widetilde{\beta_k}(t,\xi)=\widetilde{A_k}(t,\xi)^{1/2}f'(t,k,\xi),\quad (\widetilde{\beta_h},\widetilde{\beta_\theta})(t,\xi)=\widetilde{A_*}(t,\xi)^{1/2}(h',\theta')(t,\xi).
\] Since we may restrict to $|k|+|l|\leq DT_0^{21/20}$ and $|\xi|+|\eta|\lesssim DT_0^{21/20}$, by using (\ref{symbolbd1})$\sim$(\ref{symbolbd7}) and (\ref{weight1})$\sim$(\ref{extra}), we can compute 
\begin{equation}\label{boundabs1}
\begin{split}
\big|\partial_t\widetilde{\beta_k}(t,\xi)\big|&\lesssim Q^{-1}\widetilde{\beta_k}(t,\xi)+\sum_{l}\int_{\mathbb{R}}\varepsilon_0e^{-(C_0/4)(|k-l|+|\xi-\eta|)}\frac{|k|+|\xi|+1}{t^2}\bigg(\frac{\widetilde{A_k}(t,\xi)}{\widetilde{A_l}(t,\eta)}\bigg)^{1/2}\big|\widetilde{\beta_l}(t,\eta)\big|\,\mathrm{d}\eta
\\&+\sum_{l\neq 0}\int_{\mathbb{R}}\varepsilon_0 e^{-(C_0/4)(|k-l|+|\xi-\eta|)}(|k|+\mathbf{1}_{l\neq k}\cdot|\xi|+1)\bigg(\frac{\widetilde{A_k}(t,\xi)}{(\eta-tl)^2+l^2}\bigg)^{1/2}\cdot Q\big|\widetilde{\beta_l}(t,\eta)\big|\,\mathrm{d}\eta\\
&+\sum_{q\neq 0}\int_{\mathbb{R}}\varepsilon_0^2 e^{-(C_0/4)(|k|+|q|+|\xi-\eta|)}(\mathbf{1}_{q\neq k}\cdot|\xi|+1)\bigg(\frac{\widetilde{A_k}(t,\xi)}{(\eta-tq)^2+q^2}\bigg)^{1/2}\cdot Q\big|\widetilde{\beta_h}(t,\eta)\big|\,\mathrm{d}\eta
\\&+\int_{\mathbb{R}}e^{-(C_0/4)(|k|+|\xi-\eta|)}\bigg(\frac{\widetilde{A_k}(t,\xi)}{\widetilde{A_*}(t,\eta)}\bigg)^{1/2}\bigg(\frac{\varepsilon_0^2}{t}\big|\widetilde{\beta_h}(t,\eta)\big|+\varepsilon_0\big|\widetilde{\beta_\theta}(t,\eta)\big|\bigg)\,\mathrm{d}\eta+\big|\widehat{\rho_1'}(t,k,\xi)\big|
\\
&\lesssim Q^{-1}\bigg(\sum_l\int_{\mathbb{R}}e^{-(C_0/5)(|k-l|+|\xi-\eta|)}\big|\widetilde{\beta_l}(t,\eta)\big|\,\mathrm{d}\eta+\int_{\mathbb{R}}e^{-(C_0/5)(|k|+|\xi-\eta|)}\big(\big|\widetilde{\beta_h}(t,\eta)\big|+\big|\widetilde{\beta_\theta}(t,\eta)\big|\big)\,\mathrm{d}\eta\bigg)\\
&+Q^{-1}\big|\widetilde{\beta_k}(t,\xi)\big|+\varepsilon_0Q\sum_l\int_{\mathbb{R}}\big(F_k(t,\xi)+F_l(t,\xi)\big)e^{-(C_0/4)(|k-l|+|\xi-\eta|)}\big|\widetilde{\beta_l}(t,\eta)\big|\,\mathrm{d}\eta+\big|\widehat{\rho_1'}(t,k,\xi)\big|,
\end{split}
\end{equation}and similarly,
\begin{equation}\label{boundabs2}
\begin{split} \big|\partial_t\widetilde{\beta_h}(t,\xi)\big|&\lesssim Q^{-1}\big(\big|\widetilde{\beta_h}(t,\xi)\big|+\big|\widetilde{\beta_\theta}(t,\xi)\big|\big)+\frac{1}{t}\bigg(\frac{\widetilde{A_*}(t,\xi)}{\widetilde{A_0}(t,\xi)}\bigg)^{1/2}\big|\widetilde{\beta_0}(t,\xi)\big|+\big|\widehat{\rho_2'}(t,k,\xi)\big|\\
&\qquad+Q^{-1}\int_{\mathbb{R}}e^{-(C_0/5)|\xi-\eta|}\big(\big|\widetilde{\beta_h}(t,\eta)\big|+\big|\widetilde{\beta_\theta}(t,\eta)\big|\big)\,\mathrm{d}\eta\\
&\lesssim Q^{-1}\big(\big|\widetilde{\beta_h}(t,\xi)\big|+\big|\widetilde{\beta_\theta}(t,\xi)\big|+\big|\widetilde{\beta_0}(t,\xi)\big|\big)+\big|\widehat{\rho_2'}(t,k,\xi)\big|\\
&\qquad+Q^{-1}\int_{\mathbb{R}}e^{-(C_0/5)|\xi-\eta|}\big(\big|\widetilde{\beta_h}(t,\eta)\big|+\big|\widetilde{\beta_\theta}(t,\eta)\big|\big)\,\mathrm{d}\eta,
\end{split}
\end{equation} and 
\begin{equation}\label{boundabs3}
\begin{split}
 \big|\partial_t\widetilde{\beta_\theta}(t,\xi)\big|&\lesssim Q^{-1}\big|\widetilde{\beta_\theta}(t,\xi)\big|+Q^{-1}\int_{\mathbb{R}}e^{-(C_0/5)|\xi-\eta|}\big(\big|\widetilde{\beta_h}(t,\eta)\big|+\big|\widetilde{\beta_\theta}(t,\eta)\big|\big)\,\mathrm{d}\eta+\big|\widehat{\rho_3'}(t,k,\xi)\big|\\
 &\qquad+Q^{-1}\sum_l\int_{\mathbb{R}}e^{-(C_0/5)(|l|+|\xi-\eta|)}\big|\widetilde{\beta_l}(t,\eta)\big|\,\mathrm{d}\eta.
 \end{split}
\end{equation} Suppose we are in the case $t''>t'>T_0$ (the other case being similar); let
\[B_k(t,\xi)=\sup_{s\in[T_0,t]}|\widetilde{\beta_k}(s,\xi)|,\quad B(t)=\bigg(\sum_k\int_{\mathbb{R}}|B_k(t,\xi)|^2\,\mathrm{d}\xi\bigg)^{1/2},
\] with suitable changes made when $k\in\{h,\theta\}$, then by \eqref{boundabs1}$\sim$\eqref{boundabs3}, on $[t',t'']$ we have
\[
\begin{split}
B_k(t,\xi)&\leq \big|\widetilde{\beta_k}(t',\xi)\big|+CQ^{-1}\int_{t'}^t B_k(s,\xi)\,\mathrm{d}s+CQ^{-1}\int_{t'}^t\sum_{l}\int_{\mathbb{R}}e^{-10(|k-l|+|\xi-\eta|)}B_l(s,\eta)\,\mathrm{d}\eta\mathrm{d}s\\
&+C\varepsilon_0Q\sum_l\int_{\mathbb{R}}e^{-10(|k-l|+|\xi-\eta|)}B_l(t,\eta)\cdot\int_{t'}^{t}(F_k(s,\xi)+F_l(s,\xi))\,\mathrm{d}s+\sum_{j=1}^3\int_{t'}^{t}\big|\widehat{\rho_j'}(s,k,\xi)\big|\,\mathrm{d}s.
\end{split} 
\]Taking the $l_k^2L_{\xi}^2$ norm and using the fact that 
\[\sup_{|k|+|\xi|\lesssim T_0^{9/8}}\int_{1}^{T_1}F_k(s,\xi)\,\mathrm{d}s\lesssim\log k_0,
\] we get
\[B(t)\leq M_2(t')+CQ^{-1}\int_{t'}^t B(s)\,\mathrm{d}s+C\varepsilon_0Q\cdot \log k_0\cdot B(t)+\int_{t'}^tZ(s)\,\mathrm{d}s.
\] By our choice of $Q$ we have $C\varepsilon_0Q\cdot \log k_0\ll 1$, so we have 
\[B(t)\leq \sqrt{M_2(t')}+CQ^{-1}\int_{t'}^t B(s)\,\mathrm{d}s+\int_{t'}^tZ(s)\,\mathrm{d}s,
\] and thus by Gronwall,
\[
\sqrt{M_2(t'')}\leq B(t'')\leq e^{CQ^{-1}(t''-t')}\sqrt{M_2(t')}+\int_{t'}^{t''}e^{CQ^{-1}(t''-s)}Z(s)\,\mathrm{d}s. 
\] Since $\partial_t\log Z(t)\gg Q^{-1}$, the right hand side is bounded by $e^{CQ^{-1}(t''-t')}M_2(t')+T_0\cdot Z(t'')$, which proves \eqref{improvedgrowth}.

Finally, suppose $\widehat{\rho'(t)}$ is supported in the smaller region $|\xi|\lesssim \varepsilon_0^{1/30}T_0$, then by using (\ref{diffineq2})and arguing as in (\ref{movexismall}) and (\ref{errorsmallnew}), we can restrict $\widehat{g'}$ to the region $|\xi|\lesssim \varepsilon^{1/35}T_0$, up to an acceptable error, on the whole of $[T_2,T_1]$.

Since $t\geq T_3\sim \varepsilon_0^{1/40}T_0$, we have $|\xi|\ll t$, so in particular
\[
\frac{1}{(\xi-tl)^2+l^2}\lesssim\frac{1}{t^2l^2}
\] for all $l\neq 0$. By \eqref{symbolbd1}$\sim$\eqref{symbolbd7}, we thus have 
\[|\partial_t\widetilde{L}(t)|\lesssim T_0^{-1/4}\widetilde{L}(t)+T_0^{1/4}\sqrt{\widetilde{L}(t)}\cdot Z(t),
\] where 
\[\widetilde{L}(t)=\sum_{j=1}^2\sum_k\int_{\mathbb{R}}|\widehat{g_j'}(t,k,\xi)|^2\,\mathrm{d}\xi+T_0^{1/2}\sum_k\int_{\mathbb{R}}|\widehat{g_3'}(t,k,\xi)|^2\,\mathrm{d}\xi.
\] Since $|\partial_t\log R(t)|\gg T_0^{-1/4}$, we get that $\widetilde{L}(t)\lesssim T_0^{1/2}(Z(t))^2$, this proves \eqref{improve}.

\end{proof}

\bibliographystyle{abbrv} \bibliography{Gevrey-instab}

\end{document}